\documentclass[final]{siamart1116}
\usepackage{amsmath,amsfonts}
\usepackage[latin1]{inputenc}
\usepackage{mathabx}
\usepackage{mathtools}
\usepackage{nicefrac}
\usepackage{paralist}
\usepackage{pgfplots}
\usepackage{subcaption}

\newcommand{\TheTitle}{A Hybrid High-Order method for Darcy flows in fractured porous media} 
\newcommand{\TheAuthors}{F Chave, D. A. Di Pietro, and L. Formaggia}

\ifpdf
\hypersetup{
  pdftitle={\TheTitle},
  pdfauthor={\TheAuthors}
}
\fi

\newtheorem{remark}[theorem]{Remark}
\newtheorem{assumption}[theorem]{Assumption}

\newcommand{\Real}{\mathbb{R}}

\DeclareMathOperator{\Hdiv}{div}
\DeclareMathOperator{\vect}{span}
\def\normal{{\boldsymbol{n}}}
\def\f{{\mathcal{F}}}
\def\t{{\mathcal{T}}}

\newcommand{\DT}{D_T^k}
\newcommand{\FT}{\boldsymbol{F}_T^{k+1}}
\newcommand{\rF}{r_F^{k+1}}

\newcommand{\bulk}{{\rm B}}

\newcommand{\uKB}{\overline{K}_\bulk}
\newcommand{\lKB}{\underline{K}_\bulk}
\newcommand{\uKBT}{\overline{K}_{\bulk,T}}
\newcommand{\lKBT}{\underline{K}_{\bulk,T}}

\newcommand{\uKG}{\overline{K}_\Gamma}
\newcommand{\lKG}{\underline{K}_\Gamma}

\newcommand{\ulkn}{\overline{\lambda}_\Gamma}
\newcommand{\llkn}{\underline{\lambda}_\Gamma}

\newcommand{\logLogSlopeTriangle}[5]
{
    \pgfplotsextra
    {
        \pgfkeysgetvalue{/pgfplots/xmin}{\xmin}
        \pgfkeysgetvalue{/pgfplots/xmax}{\xmax}
        \pgfkeysgetvalue{/pgfplots/ymin}{\ymin}
        \pgfkeysgetvalue{/pgfplots/ymax}{\ymax}

        \pgfmathsetmacro{\xArel}{#1}
        \pgfmathsetmacro{\yArel}{#3}
        \pgfmathsetmacro{\xBrel}{#1-#2}
        \pgfmathsetmacro{\yBrel}{\yArel}
        \pgfmathsetmacro{\xCrel}{\xArel}

        \pgfmathsetmacro{\lnxB}{\xmin*(1-(#1-#2))+\xmax*(#1-#2)} 
        \pgfmathsetmacro{\lnxA}{\xmin*(1-#1)+\xmax*#1} 
        \pgfmathsetmacro{\lnyA}{\ymin*(1-#3)+\ymax*#3} 
        \pgfmathsetmacro{\lnyC}{\lnyA+#4*(\lnxA-\lnxB)}
        \pgfmathsetmacro{\yCrel}{\lnyC-\ymin)/(\ymax-\ymin)}

        \coordinate (A) at (rel axis cs:\xArel,\yArel);
        \coordinate (B) at (rel axis cs:\xBrel,\yBrel);
        \coordinate (C) at (rel axis cs:\xCrel,\yCrel);

        \draw[#5]   (A)-- node[pos=0.5,anchor=north] {\scriptsize{1}}
                    (B)-- 
                    (C)-- node[pos=0.,anchor=west] {\scriptsize{#4}} 
                    (A);
    }
}

\headers{A HHO method for Darcy flows in fractured porous media}{\TheAuthors}

\title{{\TheTitle}\thanks{The second author acknowledges the partial support of Agence Nationale de la Recherche grant HHOMM (ref. ANR-15-CE40-0005-01). The third author acknowledges the support of INdaM-GNCS under the program Progetti 2017. The authors also acknowledge the support of the Vinci Programme of Universit\'e Franco Italienne.}}

\author{
  Florent Chave\footnotemark[2]~\footnotemark[3]
  \and 
  Daniele A. Di Pietro\thanks{University of Montpellier, Institut Montpelli\'erain Alexander Grothendieck, 34095 Montpellier, France (\email{daniele.di-pietro@umontpellier.fr})}
  \and
    Luca Formaggia\thanks{Politecnico di Milano, MOX, 20133 Milano, Italy (\email{florent.chave@polimi.fr},\email{luca.formaggia@polimi.it})}
}

\begin{document}

\maketitle

\begin{abstract}
  We develop a novel Hybrid High-Order method for the simulation of Darcy flows in fractured porous media.  
  The discretization hinges on a mixed formulation in the bulk region and a primal formulation inside the fracture.
  Salient features of the method include a seamless treatment of nonconforming discretizations of the fracture, as well as the support of arbitrary approximation orders on fairly general meshes.
  For the version of the method corresponding to a polynomial degree $k\ge 0$, we prove convergence in $h^{k+1}$ of the discretization error measured in an energy-like norm.
  In the error estimate, we explicitly track the dependence of the constants on the problem data, showing that the method is fully robust with respect to the heterogeneity of the permeability coefficients, and it exhibits only a mild dependence on the square root of the local anisotropy of the bulk permeability.
  The numerical validation on a comprehensive set of test cases confirms the theoretical results.
\end{abstract}

\section{Introduction}

In this work we develop a novel Hybrid High-Order (HHO) method for the numerical simulation of steady flows in fractured porous media.

The modelling of flow and transport in fractured porous media, and the correct identification of the fractures as hydraulic barriers or conductors are of utmost importance in several applications.
In the context of nuclear waste management, the correct reproduction of flow patterns plays a key role in identifying safe underground storage sites.
In petroleum reservoir modelling, accounting for the presence and hydraulic behaviour of the fractures can have a sizeable impact on the identification of drilling sites, as well as on the estimated production rates.
In practice, there are several possible ways to incorporate the presence of fractures in porous media models.
Our focus is here on the approach developed in \cite{Martin.Jaffre.Roberts:05}, where an averaging process is applied, and the fracture is treated as an interface that crosses the bulk region.
The fracture is additionally assumed to be filled of debris, so that the flow therein can still be modelled by the Darcy law.
To close the problem, interface conditions are enforced that relate the average and jump of the bulk pressure to the normal flux and the fracture pressure.
Other works where fractures are treated as interfaces include, e.g., \cite{Bastian.Chen.Ewing.Helmig.Jakobs.Reichenberger:00,Angot.Gallouet.Herbin:99,Faille.Flauraud.Nataf.Pegaz.Schneider.Willien:02}. 

Several discretization methods for flows in fractured porous media have been proposed in the literature.
In \cite{Brenner.Hennicker.Masson.Samier:16}, the authors consider lowest-order vertex- and face-based Gradient Schemes, and prove convergence in $h$ for the energy-norm of the discretization error; see also~\cite{Brenner.Groza.ea:16} and the very recent work~\cite{Droniou.Hennicker.ea:16} on two-phase flows.
Extended Finite Element methods (XFEM) are considered in \cite{Berrone.Pieraccini.Scialo:13, Antonietti.Formaggia.Scotti.Verani.Verzotti:15} in the context of fracture networks, and their convergence properties are numerically studied.
In \cite{Benedetto.Berrone.Pieraccini.Scialo:14}, the authors compare XFEM with the recently introduced Virtual Element Method (VEM), and numerically observe in both cases convergence in $N_{\rm DOF}^{\nicefrac{1}{2}}$ for the energy-norm of the discretization error, where $N_{\rm DOF}$ stands for the number of degrees of freedom; see also \cite{Benedetto:2016:HMV:2868298.2868447, BENEDETTO201623}.
Discontinuous Galerkin methods are also considered in \cite{Antonietti.Facciola.Russo.Verani.bis:16} for a single-phase flow; see also~\cite{Antonietti.Facciola.Russo.Verani:16}.
Therein, an $hp$-error analysis in the energy norm is carried out on general polygonal/polyhedral meshes possibly including elements with unbounded number of faces, and numerical experiments are presented.
A discretization method based on a mixed formulation in the mortar space has also been very recently proposed in \cite{Boon.Nordbotten:16}, where an energy-error estimate in $h$ is proved.

Our focus is here on the Hybrid High-Order (HHO) methods originally introduced in~\cite{Di-Pietro.Ern:15} in the context of linear elasticity, and later applied in~\cite{Aghili.Boyaval.ea:15,Di-Pietro.Ern.ea:14,Di-Pietro.Ern:16,Di-Pietro.Ern.ea:16} to anisotropic heterogeneous diffusion problems.
HHO methods are based on degrees of freedom (DOFs) that are broken polynomials on the mesh and on its skeleton, and rely on two key ingredients:
\begin{inparaenum}[(i)]
\item physics-dependent local reconstructions obtained by solving small, embarassingly parallel problems and
\item high-order stabilization terms penalizing face residuals.
\end{inparaenum}
These ingredients are combined to formulate local contributions, which are then assembled as in standard FE methods.
In the context of fractured porous media flows, HHO methods display several key advantages, including:
(i) the support of general meshes enabling a seamless treatment of nonconforming geometric discretizations of the fractures (see Remark~\ref{rem:nonconforming} below);
(ii) the robustness with respect to the heterogeneity and anisotropy of the permeability coefficients (see Remark~\ref{rem:robustness} below);
(iii) the possibility to increase the approximation order, which can be useful when complex phenomena such as viscous fingering or instabilities linked to thermal convection are present;
(iv) the availability of mixed and primal formulations, whose intimate connection is now well-understood~\cite{Boffi.Di-Pietro:16};
(v) the possibility to obtain efficient implementations thanks to static condensation (see Remark~\ref{rem:implementation} below).

The HHO method proposed here hinges on a mixed formulation in the bulk coupled with a primal formulation inside the fracture.
To keep the exposition as simple as possible while retaining all the key difficulties, we focus on the two-dimensional case, and we assume that the fracture is a line segment that cuts the bulk region in two.
For a given polynomial degree $k\ge 0$, two sets of DOFs are used for the flux in the bulk region: (i) polynomials of total degree up to $k$ on each face (representing the polynomial moments of its normal component) and (ii) fluxes of polynomials of degree up to $k$ inside each mesh element.
Combining these DOFs, we locally reconstruct (i) a discrete counterpart of the divergence operator and (ii) an approximation of the flux one degree higher than element-based DOFs.
These local reconstructions are used to formulate discrete counterparts of the permeability-weighted product of fluxes and of the bluk flux-pressure coupling terms.
The primal formulation inside the fracture, on the other hand, hinges on fracture pressure DOFs corresponding to (i) polynomial moments of degree up to $k$ inside the fracture edges and (ii) point values at face vertices.
From these DOFs, we reconstruct inside each fracture face an approximation of the fracture pressure of degree $(k+1)$, which is then used to formulate a tangential diffusive bilinear form in the spirit of~\cite{Di-Pietro.Ern.ea:14}.
Finally, the terms stemming from interface conditions on the fractures are treated using bulk flux DOFs and fracture pressure DOFs on the fracture edges.

A complete theoretical analysis of the method is carried out.
In Theorem~\ref{thm:stability} below we prove stability in the form of an inf-sup condition on the global bilinear form collecting the bulk, fracture, and interface contributions.
An important intermediate result is the stability of the bulk flux-pressure coupling, whose proof follows the classical Fortin argument based on a commuting property of the divergence reconstruction.
In Theorem~\ref{thm:energyerror} below we prove an optimal error estimate in $h^{k+1}$ for an energy-like norm of the error.
The provided error estimate additionally shows that the error on the bulk flux and on the fracture pressure are (i) fully robust with respect to the heterogeneity of the bulk and fracture permeabilities, and (ii) partially robust with respect to the anisotropy of the bulk permeability (with a dependence on the square root of the local anisotropy ratio).
These estimates are numerically validated, and the performance of the method is showcased on a comprehensive set of problems.
The numerical computations additionally show that the $L^{2}$-norm of the errors on the bulk and fracture pressure converge as $h^{k+2}$.

The rest of the paper is organized as follows.
In Section~\ref{sec:setting} we introduce the continuous setting and state the problem along with its weak formulation.
In Section~\ref{sec:discrete.setting} we define the mesh and the corresponding notation, and recall known results concerning local polynomial spaces and projectors thereon.
In Section~\ref{sec:hho} we formulate the HHO approximation: in a first step, we describe the local constructions in the bulk and in the fracture; in a second step, we combine these ingredients to formulate the discrete problem; finally, we state the main theoretical results corresponding to Theorems~\ref{thm:stability} (stability) and~\ref{thm:energyerror} (error estimate).
Section~\ref{sec:num.tests} contains an extensive numerical validation of the method.
Finally, Sections~\ref{sec:stability:proof} and~\ref{sec:error.analysis} contain the proofs of Theorems~\ref{thm:stability} and~\ref{thm:energyerror}, respectively.
Readers mainly interested in the numerical recipe and results can skip these sections at first reading.


\section{Continuous setting}\label{sec:setting}

\subsection{Notation}\label{sec:setting:notation}

We consider a porous medium saturated by an incompressible fluid that occupies the space region $\Omega\subset\Real^2$ and is crossed by a fracture $\Gamma$.
We next give precise definitions of these objects.
The corresponding notation is illustrated in Figure~\ref{fig:notation}.
The extension of the following discussion to the three-dimensional case is possible but is not considered here in order to alleviate the exposition; see Remark~\ref{rem:3d} for further details.

From the mathematical point of view, $\Omega$ is an open, bounded, connected, polygonal set with Lipschitz boundary $\partial\Omega$, while $\Gamma$ is an open line segment of nonzero length.
We additionally assume that $\Omega$ lies on one side of its boundary.
The set $\Omega_\bulk \coloneq \Omega \setminus\overline{\Gamma}$ represents the bulk region.
We assume that the fracture $\Gamma$ cuts the domain $\Omega$ into two disjoint connected polygonal subdomains with Lipschitz boundary, so that the bulk region can  be decomposed as $\Omega_\bulk \coloneq \Omega_{\bulk,1} \cup \Omega_{\bulk,2}$.

We denote by $\partial\Omega_\bulk \coloneq \bigcup_{i=1}^2\partial\Omega_{\bulk,i}\setminus\overline{\Gamma}$ the external boundary of the bulk region, which is decomposed into two subsets with disjoint interiors: the Dirichlet boundary $\partial\Omega^\text{D}_\bulk$, for which we assume strictly positive 1-dimensional Haussdorf measure, and the (possibly empty) Neumann boundary $\partial\Omega^\text{N}_\bulk$.
We denote by $\normal_{\partial\Omega}$ the unit normal vector pointing outward $\Omega_\bulk$.
For $i\in\{1,2\}$, the restriction of the boundary $\partial\Omega^\text{D}_\bulk$ (respectively, $\partial\Omega^\text{N}_\bulk$) to the $i$th subdomain is denoted by $\partial\Omega^\text{D}_{\bulk,i}$ (respectively, $\partial\Omega^\text{N}_{\bulk,i}$).

We denote by $\partial\Gamma$ the boundary of the fracture $\Gamma$ with the corresponding outward unit tangential vector $\boldsymbol{\tau}_{\partial\Gamma}$.
$\partial\Gamma$ is also decomposed into two disjoint subsets: the nonempty Dirichlet fracture boundary $\partial\Gamma^\text{D}$ and the (possibly empty) Neumann fracture boundary $\partial\Gamma^\text{N}$.
Notice that this decomposition is completely independent from that of $\partial\Omega_\bulk$.
Finally, $\normal_\Gamma$ and $\boldsymbol{\tau}_\Gamma$ denote, respectively, the unit normal vector to $\Gamma$ with a fixed orientation and the unit tangential vector on $\Gamma$ such that $(\boldsymbol{\tau}_\Gamma,\normal_\Gamma)$ is positively oriented.
Without loss of generality, we assume in what follows that the subdomains are numbered so that $\normal_\Gamma$ points out of $\Omega_{\bulk,1}$.

For any function $\varphi$ sufficiently regular to admit a (possibly two-valued) trace on $\Gamma$, we define the jump and average operators such that
\begin{equation*} 
  [\![\varphi ]\!]_{\Gamma} \coloneq \varphi_{|\Omega_{\bulk,1}} - \varphi_{|\Omega_{\bulk,2}},\qquad
  \{\!\{\varphi \}\!\}_{\Gamma} \coloneq \dfrac{\varphi_{|\Omega_{\bulk,1}} + \varphi_{|\Omega_{\bulk,2}}}2.
\end{equation*}
When applied to vector functions, these operators act component-wise.

\begin{figure}\centering
  \begin{tikzpicture}[scale=1.25]
    \draw[fill=green!10!white] (0,0) -- (0,3) -- (3,3) -- (3,0) -- (0,0);
    \draw[green!50!black,very thick] (0,1.5) -- (0,0) -- (3,0) -- (3,1.5);
    \draw (0.75,1.5) node {$\Omega_{\bulk,1}$};
    \draw[green!50!black,very thick] (0,1.5) -- (0,3) -- (3,3) -- (3,1.5);
    \draw (2.25,1.5) node {$\Omega_{\bulk,2}$};
    \draw[red!70!white,very thick] (1.5,-0) -- (1.5,3);
    \draw[red!70!white] (1.75,0.75) node {$\Gamma$};

    \draw[red!60!black] plot[mark=*, mark size = 2pt] (1.5,0);
    \draw[red!60!black] plot[mark=*, mark size = 2pt] (1.5,3);

	\draw[green!50!black, very thick] (3.5,1.) -- (3.6,1.2) -- (3.6,1.) -- (3.7,1.2) -- (3.7,1.) -- (3.8,1.2);  
	\draw[green!50!black] (4.2,1.08) node {$\partial\Omega_{\bulk}$};
	
	\draw[fill=green!10!white] (3.5,1.4) -- (3.5,1.7) -- (3.8,1.7) -- (3.8,1.4) -- cycle;
	\draw (5.1,1.5) node {$\Omega_{\bulk} \coloneq \Omega_{\bulk,1} \cup \Omega_{\bulk,2}$};
	
	\draw[red!60!black] plot[mark=*, mark size = 2pt] (3.65,2); 
    \draw[red!60!black] (4.1,2) node {$\partial\Gamma$};
    
    \draw[very thick, ->] (1.5,2.25) -- (2,2.25);
    \draw (2, 2.5) node {$\normal_\Gamma$};
  \end{tikzpicture}
  \caption{Illustration of the notation introduced in Section~\ref{sec:setting:notation}.\label{fig:notation}}
\end{figure}
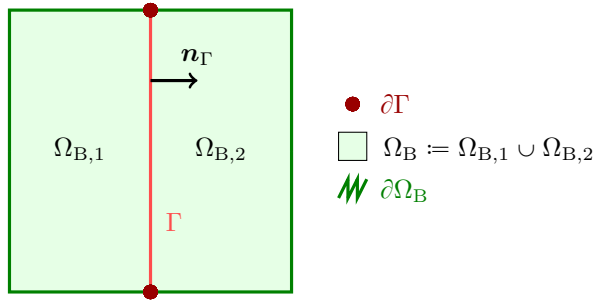

\subsection{Continuous problem}

We discuss in this section the strong formulation of the problem: the governing equations for the bulk region and the fracture, and the interface conditions that relate these subproblems.

\subsubsection{Bulk region}\label{sec:setting:strong:bulk}

In the bulk region $\Omega_\bulk$, we model the motion of the incompressible fluid by Darcy's law in mixed form, so that the pressure $p :\Omega_\bulk\rightarrow\Real$ and the flux $\boldsymbol{u} : \Omega_\bulk\rightarrow \Real^2$ satisfy
\begin{subequations}
  \label{eq:strong:bulk}
  \begin{alignat}{2}
    \boldsymbol{K} \nabla p + \boldsymbol{u} &= 0 &\qquad&\text{in $\Omega_\bulk$},\label{eq:strong:bulk:1}
    \\
    \nabla\cdot \boldsymbol{u}&= f &\qquad&\text{in $\Omega_\bulk$},\label{eq:strong:bulk:2}
    \\
    p &= g_\bulk &\qquad&\text{on $\partial\Omega^\text{D}_\bulk$},\label{eq:strong:bulk:3}
    \\
    \boldsymbol{u}\cdot \boldsymbol{n}_{\partial\Omega}&=0 &\qquad&\text{on $\partial\Omega^\text{N}_\bulk$},\label{eq:strong:bulk:4}
  \end{alignat}
\end{subequations}
where $f\in L^2(\Omega_\bulk)$ denotes a volumetric source term, $g_\bulk\in H^{\nicefrac12}(\partial\Omega^\text{D}_\bulk)$ the boundary pressure, and $\boldsymbol{K}:\Omega_\bulk \rightarrow \mathbb{R}^{2\times 2}$ the bulk permeability tensor, which is assumed to be symmetric, piecewise constant on a fixed polygonal partition ${\cal P}_\bulk=\{\omega_\bulk\}$ of $\Omega_\bulk$, and uniformly elliptic so that there exist two strictly positive real numbers $\lKB$ and $\uKB$ satisfying, for a.e. $\boldsymbol{x}\in\Omega_\bulk$ and all $\boldsymbol{z}\in\Real^2$ such that $|\boldsymbol{z}| = 1$,
$$
0 < \lKB \leq \boldsymbol{K}(\boldsymbol{x})\boldsymbol{z}\cdot\boldsymbol{z} \leq \uKB.
$$
For further use, we define the global anisotropy ratio
\begin{equation}\label{eq:varrho.B}
  \varrho_\bulk\coloneq\frac{\uKB}{\lKB}.
\end{equation}

\subsubsection{Fracture}\label{sec:setting:strong:fracture}

Inside the fracture, we consider the motion of the fluid as governed by Darcy's law in primal form, so that the fracture pressure $p_\Gamma :\Gamma\rightarrow\Real$ satisfies 
\begin{subequations}\label{eq:strong:fract}
  \begin{alignat}{2}\label{eq:strong:fract:1}
    -\nabla_{\tau}\cdot(K_\Gamma\nabla_\tau p_\Gamma)  &=\ell_\Gamma f_\Gamma + [\![ \boldsymbol{u} ]\!]_\Gamma \cdot\boldsymbol{n}_\Gamma &\qquad&\text{in $\Gamma$},
    \\\label{eq:strong:fract:2}
    p_\Gamma &= g_{\Gamma} &\qquad&\text{on $\partial\Gamma^{\rm  D}$},
    \\\label{eq:strong:fract:3}
    K_\Gamma\nabla_\tau p_\Gamma\cdot \boldsymbol{\tau}_{\partial\Gamma} &=0 &\qquad&\text{on $\partial\Gamma^{\rm N}$},
  \end{alignat}
\end{subequations}
where $f_\Gamma\in L^2(\Gamma)$ and $K_\Gamma \coloneq \kappa_\Gamma^\tau\ell_\Gamma$ with $\kappa_\Gamma^\tau:\Gamma\rightarrow \mathbb{R}$ and $\ell_\Gamma:\Gamma\rightarrow \mathbb{R}$ denoting the tangential permeability and thickness of the fracture, respectively.
The quantities $\kappa_\Gamma^\tau$ and $\ell_\Gamma$ are assumed piecewise constant on a fixed partition ${\cal P}_\Gamma=\{\omega_\Gamma\}$ of $\Gamma$, and such that there exist strictly positive real numbers $\lKG$,$\uKG$ such that, for a.e. $\boldsymbol{x}\in\Gamma$,
\begin{equation*} 
  0<\lKG \le K_\Gamma(\boldsymbol{x}) \le \uKG.
\end{equation*}
In~\eqref{eq:strong:fract}, $\nabla_\tau$ and $\nabla_\tau\cdot$ denote the tangential gradient and divergence operators along $\Gamma$, respectively.

\begin{remark}[Immersed fractures]
  The Neumann boundary condition \eqref{eq:strong:fract:3} has been used for immersed fracture tips.
  The case where the fracture is fully immersed in the domain $\Omega$ can be also considered, and it leads to a homogeneous Neumann boundary condition \eqref{eq:strong:fract:3} on the whole fracture boundary; for further details, we refer to \cite[Section 2.2.3]{Angot.Boyer.Hubert:09}, \cite{Brenner.Hennicker.Masson.Samier:16} or more recently \cite{scotti.formaggia.sottocasa:17}.
\end{remark}

\subsubsection{Coupling conditions}

The subproblems \eqref{eq:strong:bulk} and \eqref{eq:strong:fract} are coupled by the following interface conditions:
\begin{equation}\label{eq:strong:couplcond}
  \begin{alignedat}{2}
    \lambda_\Gamma\{\!\{\boldsymbol{u} \}\!\}_{\Gamma}\cdot\normal_\Gamma  &= [\![ p ]\!]_\Gamma &\qquad&\text{on $\Gamma$,}
    \\
    \lambda_\Gamma^\xi [\![\boldsymbol{u} ]\!]_{\Gamma}\cdot\normal_\Gamma &= \{\!\{p \}\!\}_\Gamma - p_\Gamma &\qquad&\text{on $\Gamma$},
  \end{alignedat}
\end{equation}
where $\xi \in (\frac12,1\rbrack$ is a model parameter chosen by the user and we have set
\begin{equation}\label{eq:lambda.lambdaxi}
  \lambda_\Gamma\coloneq\frac{\ell_\Gamma}{\kappa_\Gamma^n},\qquad
  \lambda_\Gamma^\xi\coloneq\lambda_\Gamma\left(\frac\xi2-\frac14\right).
\end{equation}
As above, $\ell_\Gamma$ is the fracture thickness, while $\kappa_\Gamma^n :\Gamma\rightarrow\mathbb{R}$ represents the normal permeability of the fracture, which is assumed piecewise constant on the partition ${\cal P}_\Gamma$ of $\Gamma$ introduced in Section~\ref{sec:setting:strong:fracture}, and 
such that, for a.e. $\boldsymbol{x}\in\Gamma$,
\begin{equation}\label{eq:bnd.kappan}
  0<\llkn \leq \lambda_\Gamma(\boldsymbol{x}) \leq \ulkn,
\end{equation}
for two given strictly positive real numbers $\ulkn$ and $\llkn$. 
\begin{remark}[Coupling condition and choice of the formulation]
  The coupling conditions \eqref{eq:strong:couplcond} arise from the averaging process along the normal direction to the fracture, and are necessary to close the problem. They relate the jump and average of the bulk flux to the jump and average of the bulk pressure and the fracture pressure.
  Using as a starting point the mixed formulation \eqref{eq:strong:bulk} in the bulk enables a natural discretization of the coupling conditions, as both the normal flux and the bulk pressure are present as unknowns. On the other hand, the use of the primal formulation \eqref{eq:strong:fract} seems natural in the fracture, since only the fracture pressure appears in \eqref{eq:strong:couplcond}. HHO discretizations using a primal formulation in the bulk as a starting point will make the object of a future work.
\end{remark}

\begin{remark}[Extension to discrete fracture networks] The model could be extended to fracture networks.
  In this case, additional coupling conditions enforcing the mass conservation and pressure continuity at fracture intersections should be included; see e.g., \cite{Brenner.Hennicker.Masson.Samier:16, Brenner.Groza.Jeannin.Masson.Pellerin:16}.
\end{remark}

\subsection{Weak formulation}\label{subsec:weakform}

The weak formulation of problem~\eqref{eq:strong:bulk}--\eqref{eq:strong:fract}--\eqref{eq:strong:couplcond} hinges on the following function spaces:
\begin{gather*}
  \boldsymbol{U} \coloneq \{\boldsymbol{u} \in\boldsymbol{H}(\Hdiv ;\Omega_\bulk) ~|~  \boldsymbol{u}\cdot\boldsymbol{n}_{\partial\Omega} = 0 \text{ on } \partial\Omega^{\rm N}_\bulk \text{ and }(\boldsymbol{u}_{|\Omega_{\bulk,1}}\cdot\boldsymbol{n}_{\Gamma}, \boldsymbol{u}_{|\Omega_{\bulk,2}}\cdot\boldsymbol{n}_{\Gamma})\in L^2(\Gamma)^2\},
  \\
  P_\bulk \coloneq L^2(\Omega_\bulk), \qquad P_\Gamma \coloneq\{p_\Gamma \in H^1(\Gamma) ~|~ p_\Gamma = 0 \text{ on } \partial\Gamma^\text{D}\},
\end{gather*}
where $\boldsymbol{H}(\Hdiv ;\Omega_\bulk)$ is spanned by vector-valued functions on $\Omega_\bulk$ whose restriction to every bulk subregion $\Omega_{\bulk,i}$, $i\in\{1,2\}$, is in $\boldsymbol{H}(\Hdiv; \Omega_{\bulk,i})$.

For any $X\subset\overline{\Omega}$, we denote by $(\cdot,\cdot)_X$ and $\|{\cdot}\|_X$ the usual inner product and norm of $L^2(X)$ or $L^2(X)^2$, according to the context.
We define the bilinear forms $a_\xi : \boldsymbol{U}\times \boldsymbol{U} \rightarrow \Real$, $b : \boldsymbol{U}\times P_\bulk \rightarrow \Real$, $c : \boldsymbol{U}\times P_\Gamma \rightarrow \Real$, and $d : P_\Gamma\times P_\Gamma \rightarrow \Real$ as follows:
\begin{equation}\label{def:cont.bilin.forms}
  \begin{aligned}
    a_\xi(\boldsymbol{u},\boldsymbol{v})
    &\coloneq (\boldsymbol{K}^{-1}\boldsymbol{u} , \boldsymbol{v})_{\Omega_\bulk}
    \hspace*{-4px}+ \hspace*{-2px}(\lambda_\Gamma^\xi[\![\boldsymbol{u}]\!]_\Gamma{\cdot}\boldsymbol{n}_\Gamma,[\![\boldsymbol{v}]\!]_\Gamma{\cdot}\boldsymbol{n}_\Gamma)_\Gamma
    \hspace*{-2px}+ \hspace*{-2px} (\lambda_\Gamma\{\!\{\boldsymbol{u}\}\!\}_\Gamma{\cdot}\normal_\Gamma,\{\!\{\boldsymbol{v}\}\!\}_\Gamma{\cdot}\normal_\Gamma)_\Gamma,
    \\
    b(\boldsymbol{u},q) &\coloneq (\nabla\cdot \boldsymbol{u},q)_{\Omega_\bulk},
    \\
    c(\boldsymbol{u},q_\Gamma) &\coloneq ([\![\boldsymbol{u}]\!]_\Gamma\cdot\normal_\Gamma,q_\Gamma)_\Gamma,
    \\
    d(p_\Gamma, q_\Gamma) &\coloneq (K_\Gamma\nabla_\tau p_\Gamma,\nabla_\tau q_\Gamma)_\Gamma.    
  \end{aligned}
\end{equation}
With these spaces and bilinear forms, the weak formulation of problem~\eqref{eq:strong:bulk}--\eqref{eq:strong:fract}--\eqref{eq:strong:couplcond} reads:
Find $(\boldsymbol{u}, p, p_{\Gamma,0})\in\boldsymbol{U}\times P_\bulk\times P_\Gamma$ such that%
\begin{equation}\label{eq:weak}
  \begin{alignedat}{5}
  	&a_\xi(\boldsymbol{u},\boldsymbol{v}) &- b(\boldsymbol{v},p) &+ c(\boldsymbol{v},p_{\Gamma,0}) &=&- (g_\bulk,\boldsymbol{v}\cdot\normal_{\partial\Omega})_{\partial\Omega^\text{D}_\bulk} &\qquad&\forall \boldsymbol{v}\in\boldsymbol{U},
  	\\
        &b(\boldsymbol{u},q)&& &=&~ (f,q)_{\Omega_\bulk} &\qquad&\forall q\in P_\bulk,
        \\
        -& c(\boldsymbol{u},q_\Gamma) &&+d(p_{\Gamma,0}, q_\Gamma)  ~&=&~ (\ell_\Gamma f_\Gamma,q_\Gamma)_\Gamma - d(p_{\Gamma,\mathrm{D}}, q_\Gamma) &\qquad&\forall q_\Gamma\in P_\Gamma,
  \end{alignedat}
\end{equation}
where $p_{\Gamma,\mathrm{D}}\in H^1(\Gamma)$ is a lifting of the fracture Dirichlet boundary datum such that $(p_{\Gamma,\mathrm{D}})_{|\partial\Gamma^{\mathrm{D}}}=g_\Gamma$.
The fracture pressure is then computed as $p_\Gamma = p_{\Gamma,0} + p_{\Gamma,\mathrm{D}}.$ This problem is well-posed; we refer the reader to \cite[Proposition 2.4]{Antonietti.Formaggia.Scotti.Verani.Verzotti:15} for a proof.


\section{Discrete setting}\label{sec:discrete.setting}

\subsection{Mesh}

The HHO method is built upon a polygonal mesh of the domain $\Omega$ defined prescribing a set of mesh elements $\t_h$ and a set of mesh faces $\f_h$.

The set of mesh elements $\t_h$ is a finite collection of open disjoint polygons with nonzero area such that $\overline\Omega = \bigcup_{T\in \t_h} \overline T$ and $h = \max_{T\in \t_h} h_T$, with $h_T$ denoting the diameter of $T$.
We also denote by $\partial T$ the boundary of a mesh element $T\in\t_{h}$.
The set of mesh faces $\f_h$ is a finite collection of open disjoint line segments in $\overline{\Omega}$ with nonzero length such that, for all $F\in\f_h$, (i) either there exist two distinct mesh elements $T_1,T_2\in\t_h$ such that $F\subset\partial T_1\cap\partial T_2$ (and $F$ is called an interface) or (ii) there exist a (unique) mesh element $T\in\t_h$ such that $F\subset\partial T\cap\partial\Omega$ (and $F$ is called a boundary face).
We assume that $\f_h$ is a partition of the mesh skeleton in the sense that $\bigcup_{T\in\t_h}\partial T=\bigcup_{F\in\f_h} \overline{F}$.

\begin{remark}[Mesh faces]
  Despite working in two space dimensions, we have preferred the terminology ``face'' over ``edge'' in order to (i) be consistent with the standard HHO nomenclature and (ii) stress the fact that faces {\em need not} coincide with polygonal edges (but can be subsets thereof); see also Remark~\ref{rem:nonconforming} on this point.
\end{remark}

We denote by $\f_h^{\rm i}$ the set of all interfaces and by $\f_h^{\rm b}$ the set of all boundary faces, so that $\f_h = \f_h^{\rm i} \cup \f_h^{\rm b}$. The length of a face $F\in \f_h$ is denoted by $h_F$. For any mesh element $T \in \t_h$, $\f_T$ is the set of faces that lie on $\partial T$ and, for any $F\in \f_T$, $\normal_{TF}$ is the unit normal to $F$ pointing out of $T$. Symmetrically, for any $F \in \f_h$, $\t_F$  is the set containing the mesh elements sharing the face $F$ (two if $F$ is an interface, one if $F$ is a boundary face).

To account for the presence of the fracture, we make the following
\begin{assumption}[Geometric compliance with the fracture]\label{ass:geom.compl}
  The mesh is compliant with the fracture, i.e., there exists a subset $\f_h^\Gamma\subset\f_h^{\rm i}$ such that $\overline{\Gamma} = \bigcup_{F\in\f_h^\Gamma}\overline{F}.$
  As a result, $\f_h^\Gamma$ is a (1-dimensional) mesh of the fracture.
\end{assumption}
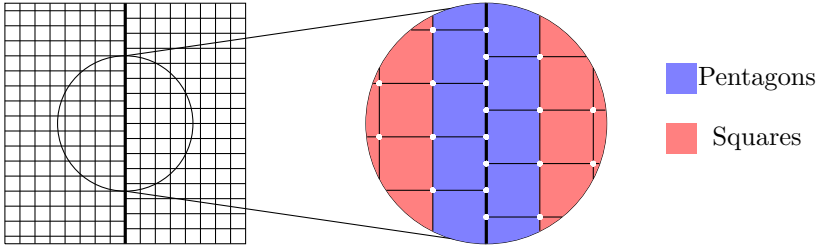
\begin{figure}\centering
\begin{tikzpicture}[scale=0.8]
\draw (0,0) -- (4,0) -- (4,4) -- (0,4) -- (0,0);
\draw (0.25,0) -- (0.25,4);
\draw (0.5,0) -- (0.5,4);
\draw (0.75,0) -- (0.75,4);
\draw (1,0) -- (1,4);
\draw (1.25,0) -- (1.25,4);
\draw (1.5,0) -- (1.5,4);
\draw (1.75,0) -- (1.75,4);
\draw [very thick] (2,0) -- (2,4);
\draw (2.25,0) -- (2.25,4);
\draw (2.5,0) -- (2.5,4);
\draw (2.75,0) -- (2.75,4);
\draw (3,0) -- (3,4);
\draw (3.25,0) -- (3.25,4);
\draw (3.5,0) -- (3.5,4);
\draw (3.75,0) -- (3.75,4);

\draw (0,0.125) -- (2,0.125);
\draw (0,0.375) -- (2,0.375);
\draw (0,0.625) -- (2,0.625);
\draw (0,0.875) -- (2,0.875);
\draw (0,1.125) -- (2,1.125);
\draw (0,1.375) -- (2,1.375);
\draw (0,1.625) -- (2,1.625);
\draw (0,1.875) -- (2,1.875);
\draw (0,2.125) -- (2,2.125);
\draw (0,2.375) -- (2,2.375);
\draw (0,2.625) -- (2,2.625);
\draw (0,2.875) -- (2,2.875);
\draw (0,3.125) -- (2,3.125);
\draw (0,3.375) -- (2,3.375);
\draw (0,3.625) -- (2,3.625);
\draw (0,3.875) -- (2,3.875);

\draw (2,0.25) -- (4,0.25);
\draw (2,0.5) -- (4,0.5);
\draw (2,0.75) -- (4,0.75);
\draw (2,1) -- (4,1);
\draw (2,1.25) -- (4,1.25);
\draw (2,1.5) -- (4,1.5);
\draw (2,1.75) -- (4,1.75);
\draw (2,2) -- (4,2);
\draw (2,2.25) -- (4,2.25);
\draw (2,2.5) -- (4,2.5);
\draw (2,2.75) -- (4,2.75);
\draw (2,3) -- (4,3);
\draw (2,3.25) -- (4,3.25);
\draw (2,3.5) -- (4,3.5);
\draw (2,3.75) -- (4,3.75);

\draw (2,2) circle(1.125);
\draw (8,2) circle(2);
\draw (2,3.125) -- (8,4);
\draw (2,0.875) -- (8,0);


\newcommand{\zoompentagon}{(8-2*4/9,0) rectangle (8+2*4/9,10*4/9)};
\newcommand{\zoomsquareLeft}{(8-5*4/9,0) rectangle (8-2*4/9,10*4/9)};
\newcommand{\zoomsquareRight}{(8+2*4/9,0) rectangle (8+5*4/9,10*4/9)};
\newcommand{\zoomcircle}{(8,2) circle(2)};
\begin{scope}
\clip  \zoompentagon;
\fill[blue!50] \zoomcircle;
\end{scope}
\begin{scope}
\clip  \zoomsquareLeft;
\fill[red!50] \zoomcircle;
\end{scope}
\begin{scope}
\clip  \zoomsquareRight;
\fill[red!50] \zoomcircle;
\end{scope}
\fill[blue!50] (11,2.5) -- (11.5,2.5) -- (11.5,3) -- (11,3) -- cycle;
\draw (12.5,2.75) node {Pentagons};
\fill[red!50] (11,1.5) -- (11.5,1.5) -- (11.5,2) -- (11,2) -- cycle;
\draw (12.5,1.75) node {Squares};

\draw [very thick] (8,0) -- (8,4);
\draw (8,2*4/9) -- (6.34,2*4/9);
\draw (8,4*4/9) -- (6.02,4*4/9);
\draw (8,6*4/9) -- (6.12,6*4/9);
\draw (8,8*4/9) -- (6.75,8*4/9);
\draw (8-2*4/9,0.215) -- (8-2*4/9,3.785);
\draw (8-4*4/9,1.075) -- (8-4*4/9,2.925);
\draw (8,1*4/9) -- (8+1.26,1*4/9);
\draw (8,3*4/9) -- (8+1.88,3*4/9);
\draw (8,5*4/9) -- (8+1.98,5*4/9);
\draw (8,7*4/9) -- (8+1.66,7*4/9);
\draw (8+2*4/9,0.215) -- (8+2*4/9,3.785);
\draw (8+4*4/9,1.075) -- (8+4*4/9,2.925);

\draw[white] plot[mark=*, mark size = 1pt] (8,1*4/9);
\draw[white] plot[mark=*, mark size = 1pt] (8,2*4/9);
\draw[white] plot[mark=*, mark size = 1pt] (8,3*4/9);
\draw[white] plot[mark=*, mark size = 1pt] (8,4*4/9);
\draw[white] plot[mark=*, mark size = 1pt] (8,5*4/9);
\draw[white] plot[mark=*, mark size = 1pt] (8,6*4/9);
\draw[white] plot[mark=*, mark size = 1pt] (8,7*4/9);
\draw[white] plot[mark=*, mark size = 1pt] (8,8*4/9);
\draw[white] plot[mark=*, mark size = 1pt] (8-2*4/9,2*4/9);
\draw[white] plot[mark=*, mark size = 1pt] (8-2*4/9,4*4/9);
\draw[white] plot[mark=*, mark size = 1pt] (8-2*4/9,6*4/9);
\draw[white] plot[mark=*, mark size = 1pt] (8-2*4/9,8*4/9);
\draw[white] plot[mark=*, mark size = 1pt] (8-4*4/9,4*4/9);
\draw[white] plot[mark=*, mark size = 1pt] (8-4*4/9,6*4/9);
\draw[white] plot[mark=*, mark size = 1pt] (8+2*4/9,1*4/9);
\draw[white] plot[mark=*, mark size = 1pt] (8+2*4/9,3*4/9);
\draw[white] plot[mark=*, mark size = 1pt] (8+2*4/9,5*4/9);
\draw[white] plot[mark=*, mark size = 1pt] (8+2*4/9,7*4/9);
\draw[white] plot[mark=*, mark size = 1pt] (8+4*4/9,3*4/9);
\draw[white] plot[mark=*, mark size = 1pt] (8+4*4/9,5*4/9);
\end{tikzpicture}
\caption{Treatment of nonconforming fracture discretizations.\label{fig:nonconforming.fracture}}
\end{figure}
\begin{remark}[Polygonal meshes and geometric compliance with the fracture]\label{rem:nonconforming}
  Fulfilling Assumption~\ref{ass:geom.compl} does not pose particular problems in the context of polygonal methods, even when the fracture discretization is nonconforming in the classical sense.
  Consider, e.g., the situation illustrated in Figure~\ref{fig:nonconforming.fracture}, where the fracture lies at the intersection of two nonmatching Cartesian submeshes.
  In this case, no special treatment is required provided the mesh elements in contact with the fracture are treated as pentagons with two coplanar faces instead of rectangles.
  This is possible since, as already pointed out, the set of mesh faces $\f_h$ need not coincide with the set of polygonal edges of $\t_h$.
\end{remark}
The set of vertices of the fracture is denoted by ${\cal V}_h$ and, for all $F\in\f_h^\Gamma$, we denote by ${\cal V}_F$ the vertices of $F$.
For all $F\in\f_h^\Gamma$ and all $V\in{\cal V}_F$, $\boldsymbol{\tau}_{FV}$ denotes the unit vector tangent to the fracture and oriented so that it points out of $F$.
Finally, ${\cal V}_h^{\rm D}$ is the set containing the points in $\partial\Gamma^{\rm D}$.

To avoid dealing with jumps of the problem data inside mesh elements, as well as on boundary and fracture faces, we additionally make the following
\begin{assumption}[Compliance with the problem data]\label{ass:data.compl}
  The mesh is compliant with the data, i.e., the following conditions hold:
  \begin{enumerate}[(i)]
  \item {\em Compliance with the bulk permeability.} For each mesh element $T\in\t_h$, there exists a unique sudomain $\omega_\bulk\in{\cal P}_\bulk$ (with ${\cal P}_\bulk$ partition introduced in Section~\ref{sec:setting:strong:bulk}) such that $T\subset\omega_\bulk$;
  \item {\em Compliance with the fracture thickness, normal, and tangential permeabilities.} For each fracture face $F\in\f_h^\Gamma$, there is a unique subdomain $\omega_\Gamma\in{\cal P}_{\Gamma}$ (with ${\cal P}_{\Gamma}$ partition introduced in Section~\ref{sec:setting:strong:fracture}) such that $F\subset\omega_\Gamma$;
  \item {\em Compliance with the boundary conditions.} There exist subsets $\f_h^{\rm D}$ and $\f_h^{\rm N}$ of $\f_h^{\rm b}$ such that $\overline{\partial\Omega_\bulk^{\rm N}}=\bigcup_{F\in\f_h^{\rm N}}\overline{F}$ and $\overline{\partial\Omega_\bulk^{\rm D}}=\bigcup_{F\in\f_h^{\rm D}}\overline{F}$.
  \end{enumerate}
\end{assumption}

For the $h$-convergence analysis, one needs to make assumptions on how the mesh is refined.
The notion of geometric regularity for polygonal meshes is, however, more subtle than for standard meshes.
To formulate it, we assume the existence of a matching simplicial submesh, meaning that there is a conforming triangulation $\mathfrak{T}_h$ of the domain such that each mesh element $T\in\t_h$ is decomposed into a finite number of triangles from $\mathfrak{T}_h$, and each mesh face $F\in\f_h$ is decomposed into a finite number of edges from the skeleton of $\mathfrak{T}_h$.
We denote by $\varrho\in(0,1)$ the regularity parameter such that (i) for any triangle $S\in\mathfrak{T}_h$ of diameter $h_S$ and inradius $r_S$, $\varrho h_S\le r_S$ and (ii) for any mesh element $T\in\t_h$ and any triangle $S\in\mathfrak{T}_h$ such that $S\subset T$, $\varrho h_T \le h_S$.
When considering $h$-refined mesh sequences, $\varrho$ should remain uniformly bounded away from zero.
We stress that the matching triangular submesh is merely a theoretical tool, and need not be constructed in practice.

\subsection{Local polynomial spaces and projectors}

Let an integer $l \geq 0$ be fixed, and let $X$ be a mesh element or face.
We denote by $\mathbb{P}^l(X)$ the space spanned by the restriction to $X$ of two-variate polynomials of total degree up to $l$, and define the $L^2$-orthogonal projector $\pi_X^l : L^1(X)\rightarrow\mathbb{P}^l(X)$ such that, for all $v\in L^1(X)$, $\pi_X^l v$ solves
\begin{equation}\label{def:lproj}
  (\pi_X^l v - v,w)_X = 0 \qquad \forall w\in\mathbb{P}^l(X).
\end{equation}
By the Riesz representation theorem in $\mathbb{P}^l(X)$ for the $L^2$-inner product, this defines $\pi_X^lv$ uniquely.

It has been proved in~\cite[Lemmas 1.58 and 1.59]{Di-Pietro.Ern:12} that the $L^2$-orthogonal projector on mesh elements has optimal approximation properties:
For all $s\in\{0,\ldots,l+1\}$, all $T\in\t_h$, and all $v\in H^s(T)$,
\begin{subequations}  
  \begin{equation}
    \label{eq:approx:lproj}
    |v-\pi_T^lv|_{H^m(T)}\le C h_T^{s-m}|v|_{H^s(T)} \qquad\forall m\in\{0,\ldots,s\},    
  \end{equation}
  and, if $s\ge 1$,
  \begin{equation}
    |v-\pi_T^lv|_{H^m(\f_T)}\le C h_T^{s-m-\nicefrac12}|v|_{H^s(T)} \qquad\forall m\in\{0,\ldots,s-1\},
    \label{eq:approx.trace:lproj}
  \end{equation}
\end{subequations}
with real number $C>0$ only depending on $\varrho$, $l$, $s$, and $m$, and $H^m(\f_T)$ spanned by the functions on $\partial T$ that are in $H^m(F)$ for all $F\in\f_T$.
More general $W^{s,p}$-approximation results for the $L^2$-orthogonal projector can be found in~\cite{Di-Pietro.Droniou:16}; see also~\cite{Di-Pietro.Droniou.bis:16} concerning projectors on local polynomial spaces.


\section{The Hybrid High-Order method}\label{sec:hho}

In this section we illustrate the local constructions in the bulk and in the fracture on which the HHO method hinges, formulate the discrete problem, and state the main results.

\subsection{Local construction in the bulk}\label{sec:setting:spaces}

We present here the key ingredients to discretize the bulk-based terms in problem \eqref{eq:weak}. First, we introduce the local DOF spaces for the bulk-based flux and pressure unknowns. Then, we define local divergence and flux reconstruction operators obtained from local DOFs.

In this section, we work on a fixed mesh element $T\in\t_{h}$, and denote by $\boldsymbol{K}_T \coloneq \boldsymbol{K}_{|T}\in\mathbb{P}^0(T)^{2\times 2}$ the (constant) restriction of the bulk permeability tensor to the element $T$.
We also introduce the local anisotropy ratio
\begin{equation}\label{eq:varrho.BT}
  \varrho_{\bulk,T} \coloneq \frac{\uKBT}{\lKBT},
\end{equation}
where $\uKBT$ and $\lKBT$ denote, respectively, the largest and smallest eigenvalue of $\boldsymbol{K}_T$.
In the error estimate of Theorem~\ref{thm:energyerror}, we will explicitly track the dependence of the constants on $\rho_{\bulk,T}$ in order to assess the robustness of our method with respect to the anisotropy of the diffusion coefficient.

\subsubsection{Local bulk unknowns}\label{sec:bulk.DOFs}

\begin{figure}\center
  \includegraphics{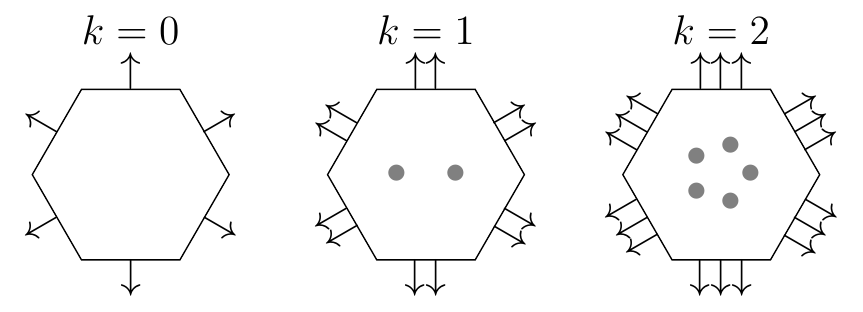}
  \caption{Local DOF space $\underline{\boldsymbol{U}}_T^k$ for a hexagonal mesh element and $k\in\{0,1,2\}$.\label{fig:dofs}}
\end{figure}
For any integer $l\ge 0$, set $\boldsymbol{U}_T^l \coloneq \boldsymbol{K}_T\nabla\mathbb{P}^l(T)$. The local DOF spaces for the bulk flux and pressure are given by (see Figure~\ref{fig:dofs})
\begin{equation}\label{eq:UT.PT}
  \underline{\boldsymbol{U}}_T^k \coloneq \boldsymbol{U}_T^k\times\left(\bigtimes_{F\in\f_T}\mathbb{P}^k(F)\right),\qquad
  P_{\bulk,T}^k \coloneq \mathbb{P}^k(T).
\end{equation}
Notice that, for $k=0$, we have $\boldsymbol{U}_T^0=\boldsymbol{K}_T\nabla\mathbb{P}^0(T)=\{\boldsymbol{0}\}$, expressing the fact that element-based flux DOFs are not needed.
A generic element $\underline{\boldsymbol{v}}_T\in\underline{\boldsymbol{U}}_T^k$ is decomposed as $\underline{\boldsymbol{v}}_T = (\boldsymbol{v}_T, (v_{TF})_{F\in\f_T})$. We define on $\underline{\boldsymbol{U}}_T^k$ and on $P_{\bulk,T}^k$, respectively, the norms $\|{\cdot}\|_{\boldsymbol{U},T}$ and $\|{\cdot}\|_{\bulk,T}$ such that, for all $\underline{\boldsymbol{v}}_T\in\underline{\boldsymbol{U}}_T^k$ and all $q_T\in P_{\bulk,T}^k$,
\begin{equation}\label{eq:normUT.normPT}
  \|\underline{\boldsymbol{v}}_T\|_{\boldsymbol{U},T}^2 \coloneq (\uKBT)^{-1}
  \left(
  \|\boldsymbol{v}_T\|_T^2 + \sum_{F\in\f_T} h_F\|v_{TF}\|_F^2
  \right),\qquad \|q_T\|_{\bulk,T} \coloneq \|q_T\|_{T},
\end{equation}
where we remind the reader that $\uKBT$ denotes the largest eigenvalue of the two-by-two matrix $\boldsymbol{K}_T$, see Section \ref{sec:setting:spaces}.
We define the local interpolation operator $\underline{\boldsymbol{I}}_T^k : H^1(T)^2 \rightarrow \underline{\boldsymbol{U}}_T^k$ such that, for all $\boldsymbol{v} \in H^1(T)^2$,
\begin{equation}\label{eq:IT}
\underline{\boldsymbol{I}}_T^k\boldsymbol{v} \coloneq (\boldsymbol{K}_T\nabla y_T, (\pi_F^k(\boldsymbol{v}\cdot\normal_{TF}))_{F\in\f_T}),
\end{equation}
where $y_T\in\mathbb{P}^k(T)$ is the solution (defined up to an additive constant) of the following Neumann problem:
\begin{equation}\label{def:interpolatorbulk:neumann}
(\boldsymbol{K}_T\nabla y_T,\nabla q_T)_T = (\boldsymbol{v},\nabla q_T)_T\qquad \forall q_T\in\mathbb{P}^{k}(T).
\end{equation}

\begin{remark}[Domain of the interpolator]\label{rem:U^+(T)}
The regularity in $H^1(T)^2$ beyond $\boldsymbol{H}(\Hdiv;T)$ is classically needed for the face interpolators to be well-defined; see, e.g.,~\cite[Section~2.5.1]{Boffi.Brezzi.ea:13} for further insight into this point.
\end{remark}

\subsubsection{Local divergence reconstruction operator}

We define the local divergence reconstruction operator $\DT : \underline{\boldsymbol{U}}_T^{k} \rightarrow P_{\bulk,T}^k$ such that, for all $\underline{\boldsymbol{v}}_T = (\boldsymbol{v}_T, (v_{TF})_{F\in\f_T}) \in \underline{\boldsymbol{U}}_T^{k}$, $\DT \underline{\boldsymbol{v}}_T$ solves
\begin{align}\label{eq:DT}
	(\DT \underline{\boldsymbol{v}}_T, q_T)_T = -(\boldsymbol{v}_T,\nabla q_T)_T + \sum_{F\in\f_T} (v_{TF},q_T)_F\qquad\forall q_T \in P_{\bulk,T}^k.
\end{align}
By the Riesz representation theorem in $P_{\bulk,T}^k$ for the $L^2$-inner product, this defines the divergence reconstruction uniquely.
The right-hand side of~\eqref{eq:DT} is designed to resemble an integration by parts formula where the role of the function represented by $\underline{\boldsymbol{v}}_T$ is played by element-based DOFs in volumetric terms and face-based DOFs in boundary terms.
With this choice, the following commuting property holds (see \cite[Lemma 2]{Di-Pietro.Ern:16}): For all $\boldsymbol{v}\in H^1(T)^2$,
\begin{align}\label{lem:commutdivbulk}
\DT\underline{\boldsymbol{I}}_T^k\boldsymbol{v} = \pi_T^k(\nabla\cdot \boldsymbol{v}).
\end{align}
We also note the following inverse inequality, obtained from~\eqref{eq:DT} setting $q_{T}=\DT\underline{\boldsymbol{v}}_{T}$ and using Cauchy--Schwarz and discrete inverse and trace inequalities (see \cite[Lemma 8]{Di-Pietro.Ern:16} for further details):
There is a real number $C>0$ independent of $h$ and of $T$, but depending on $\varrho$ and $k$, such that, for all $\underline{\boldsymbol{v}}_T\in\underline{\boldsymbol{U}}_T^k$,
\begin{equation}\label{proof:stabDiv}
h_T\|\DT \underline{\boldsymbol{v}}_T\|_T \le C \uKBT^{\nicefrac12}\|\underline{\boldsymbol{v}}_T\|_{\boldsymbol{U},T}.
\end{equation}

\subsubsection{Local flux reconstruction operator and permeability-weighted local product}

We next define the local discrete flux operator $\FT : \underline{\boldsymbol{U}}_T^k \rightarrow \boldsymbol{U}_T^{k+1}$ such that, for all $\underline{\boldsymbol{v}}_T = (\boldsymbol{v}_T, (v_{TF})_{F\in\f_T}) \in \underline{\boldsymbol{U}}_T^{k}$, $\FT\underline{\boldsymbol{v}}_T$ solves
\begin{equation}\label{def:operatorfluxbulk}
  (\FT\underline{\boldsymbol{v}}_T,\nabla w_T)_T
  = -(\DT\underline{\boldsymbol{v}}_T, w_T)_T + \sum_{F\in\f_T}(v_{TF},w_T)_F\qquad
  \forall w_T\in\mathbb{P}^{k+1}(T).
\end{equation}
By the Riesz representation theorem in $\boldsymbol{U}_T^{k+1}$ for the $(\boldsymbol{K}_{T}^{-1}\cdot,\cdot)_{T}$-inner product, this defines the flux reconstruction uniquely.
Also in this case, the right-hand side is designed so as to resemble an integration by parts formula where the role of the divergence of the function represented by $\underline{\boldsymbol{v}}_T$ is played by $\DT\underline{\boldsymbol{v}}_T$, while its normal traces are replaced by boundary DOFs.

We now have all the ingredients required to define the permeability-weighted local product $m_T : \boldsymbol{\underline{U}}_T^k\times\boldsymbol{\underline{U}}_T^k\rightarrow\Real$ such that
\begin{equation}\label{eq:mT}
  m_T(\underline{\boldsymbol{u}}_T,\underline{\boldsymbol{v}}_T)
  \coloneq (\boldsymbol{K}_T^{-1}\FT\underline{\boldsymbol{u}}_T,\FT\underline{\boldsymbol{v}}_T)_T + J_T(\underline{\boldsymbol{u}}_T,\underline{\boldsymbol{v}}_T),
\end{equation}
where the first term is the usual Galerkin contribution responsible for consistency, while $J_T : \underline{\boldsymbol{U}}_T^k\times\underline{\boldsymbol{U}}_T^k \rightarrow \Real$ is a stabilization bilinear form such that, letting $\mu_{TF} \coloneq \boldsymbol{K}_T \normal_{TF} \cdot \normal_{TF}$ for all $F\in\f_T$,
\begin{align*}
J_T(\boldsymbol{\underline{u}}_T,\boldsymbol{\underline{v}}_T) \coloneq \sum_{F\in\f_T} \frac{h_F}{\mu_{TF}}(\FT\boldsymbol{\underline{u}}_T\cdot\normal_{TF} - u_{TF},\FT\boldsymbol{\underline{v}}_T\cdot\normal_{TF} - v_{TF})_F.
\end{align*}
The role of $J_T$ is to ensure the existence of a real number $\eta_m>0$ independent of $h$, $T$, and $\boldsymbol{K}_T$, but possibly depending on $\varrho$ and $k$, such that, for all $\underline{\boldsymbol{v}}_T\in\underline{\boldsymbol{U}}_T^k$,
\begin{equation}\label{proof:equivMTST}
  \eta_m^{-1}\|\underline{\boldsymbol{v}}_T\|_{\boldsymbol{U},T}^2
  \le \|\underline{\boldsymbol{v}}_T\|_{m,T}^2\coloneq m_T(\underline{\boldsymbol{v}}_T,\underline{\boldsymbol{v}}_T)
  \le\eta_m \rho_{\bulk,T}\|\underline{\boldsymbol{v}}_T\|_{\boldsymbol{U},T}^2,
\end{equation}
with norm $\|{\cdot}\|_{\boldsymbol{U},T}$ defined by \eqref{eq:normUT.normPT}; see \cite[Lemma 4]{Di-Pietro.Ern:16} for a proof.
Additionally, we note the following consistency property for $J_T$ proved in~\cite[Lemma~9]{Di-Pietro.Ern:16}:
There is a real number $C>0$ independent of $h$, $T$, and $\boldsymbol{K}_T$, but possibly depending on $\varrho$ and $k$, such that, for all $\boldsymbol{v}=\boldsymbol{K}_T\nabla q$ with $q\in H^{k+2}(T)$,
\begin{equation}\label{eq:consistency:JT}
  J_T(\underline{\boldsymbol{I}}_T^k\boldsymbol{v},\underline{\boldsymbol{I}}_T^k\boldsymbol{v})^{\nicefrac12}
  \le C \varrho_{\bulk,T}^{\nicefrac12} \uKB^{\nicefrac12} h_T^{k+1}|q|_{H^{k+2}(T)}.
\end{equation}

\subsection{Local construction in the fracture}\label{sec:discrete.fracture}

We now focus on the discretization of the fracture-based terms in problem \eqref{eq:weak}. First, we define the local space of fracture pressure DOFs, then a local pressure reconstruction operator inspired by a local integration by parts formula.
Based on this operator,  we formulate a local discrete tangential diffusive bilinear form.
Throughout this section, we work on a fixed fracture face $F\in\f_h^\Gamma$ and we let, for the sake of brevity, $K_F\coloneq(K_\Gamma)_{|F}\in\mathbb{P}^0(F)$ with $K_\Gamma$ defined in Section~\ref{sec:setting:strong:fracture}.

\subsubsection{Local fracture unknowns}

Set $\mathbb{P}(V) \coloneq \vect\{1\}$ for all $V\in\mathcal{V}_F$. The local space of DOFs for the fracture pressure is
\begin{align*}
\underline{P}_{\Gamma,F}^k \coloneq \mathbb{P}(F)^k\times\left(\bigtimes_{V\in\mathcal{V}_F}\mathbb{P}(V)\right).
\end{align*}
In what follows, a generic element $\underline{q}_F^\Gamma\in\underline{P}_{\Gamma,F}^k$ is decomposed as $\underline{q}_F^\Gamma = (q_F^\Gamma, (q_{V}^\Gamma)_{V\in\mathcal{V}_F})$. We define on $\underline{P}_{\Gamma,F}^k$ the seminorm $\|{\cdot}\|_{\Gamma,F}$ such that, for all $\underline{q}_{F}^\Gamma\in \underline{P}_{\Gamma,F}^k$,
\begin{equation*} 
\|\underline{q}_{F}^\Gamma\|_{\Gamma,F}^2 \coloneq \|K_F^{\nicefrac12}\nabla_\tau q_F^\Gamma\|_F^2 + \sum_{V\in\mathcal{V}_F}\dfrac{K_F}{h_F}(q_F-q_V)^2(V).
\end{equation*}
We also introduce the local interpolation operator $\underline{I}_F^k : H^1(F) \rightarrow \underline{P}_{\Gamma,F}^k$ such that, for all $q\in H^1(F)$,
\begin{align*}
\underline{I}_F^k q \coloneq (\pi_F^k q, (q(V))_{V\in\mathcal{V}_F}).
\end{align*}

\subsubsection{Pressure reconstruction operator and local tangential diffusive bilinear form}

We define the local pressure reconstruction operator $\rF : \underline{P}_{\Gamma,F}^k \rightarrow \mathbb{P}^{k+1}(F)$ such that, for all $\underline{q}_F^\Gamma = (q_F^\Gamma,(q_V^\Gamma)_{V\in\mathcal{V}_F}) \in \underline{P}_{\Gamma,F}^k$, $\rF \underline{q}_F^\Gamma$ solves
\begin{equation*} 
  (K_F\nabla_\tau \rF \underline{q}_F^\Gamma, \nabla_\tau w_F^\Gamma)_F = -(q_F^\Gamma,\nabla_\tau\cdot(K_F\nabla_\tau w_F^\Gamma))_F + \sum_{V\in\mathcal{V}_F} q_V^\Gamma(K_F\nabla_\tau w_F^\Gamma\cdot \boldsymbol{\tau}_{FV})(V),  
\end{equation*}
for all $w_F^\Gamma\in\mathbb{P}^{k+1}(F)$. By the Riesz representation theorem in $\nabla\mathbb{P}^{k+1}(F)$ for the $(K_{F}\cdot,\cdot)_{F}$-inner product, this relation defines a unique element $\nabla_\tau \rF \underline{q}_F^\Gamma$, hence a polynomial $\rF \underline{q}_F^\Gamma\in\mathbb{P}^{k+1}(F)$ up to an additive constant.
This constant is fixed by additionally imposing that
\begin{equation*} 
  (\rF \underline{q}_F^\Gamma - q_F^\Gamma,1)_F = 0.
\end{equation*}
We can now define the local tangential diffusive bilinear form $d_F : \underline{P}_{\Gamma,F}^k\times\underline{P}_{\Gamma,F}^k \rightarrow \Real$ such that
\begin{equation*}
  d_F(\underline{p}_F^\Gamma,\underline{q}_F^\Gamma) \coloneq (K_F\nabla_\tau \rF\underline{p}_F^\Gamma,\nabla_\tau \rF\underline{q}_F^\Gamma)_F + j_F(\underline{p}_F^\Gamma,\underline{q}_F^\Gamma),
\end{equation*}
where the first term is the standard Galerkin contribution responsible for consistency, while $j_F : \underline{P}_{\Gamma,F}^k\times\underline{P}_{\Gamma,F}^k \rightarrow \Real$ is the stabilization bilinear form such that
\begin{equation*}
j_F(\underline{p}_F^\Gamma,\underline{q}_F^\Gamma) \coloneq \sum_{V\in\mathcal{V}_F} \frac{K_{F}}{h_F}(R_F^{k+1}\underline{p}_F^\Gamma(V) - p_V^\Gamma)(R_F^{k+1}\underline{q}_F^\Gamma(V) - q_V^\Gamma),
\end{equation*}
with $R_F^{k+1} : \underline{P}_{\Gamma,F}^k \rightarrow \mathbb{P}^{k+1}(F)$ such that, for all $\underline{q}_F^\Gamma\in\underline{P}_{\Gamma,F}^k$,
$R_F^{k+1}\underline{q}_F^\Gamma \coloneq q_F^\Gamma + (\rF\underline{q}_F^\Gamma - \pi_F^k \rF\underline{q}_F^\Gamma)$.
The role of $j_{T}$ is to ensure stability and boundedness, expressed by the existence of a real number $\eta_d>0$ independent of $h$, $F$, and of $K_F$, but possibly depending on $k$ and $\varrho$, such that, for all $\underline{q}_F^\Gamma\in\underline{P}_{\Gamma,F}^k$, the following holds (see~\cite[Lemma 4]{Di-Pietro.Ern.ea:14}):
\begin{equation}\label{lem:defstabmF}
  \eta_d^{-1}\|\underline{q}_F^\Gamma\|_{\Gamma,F}^2
  \le 
  d_F(\underline{q}_F^\Gamma,\underline{q}_F^\Gamma)
  \le \eta_d \|\underline{q}_F^\Gamma\|_{\Gamma,F}^2.
\end{equation}


\subsection{The discrete problem}\label{sec:discrete}

We define the global discrete spaces together with the corresponding interpolators and norms, formulate the discrete problem, and state the main results.

\subsubsection{Global discrete spaces}\label{sec:discrete:global}

We define the following global spaces of fully discontinuous bulk flux and pressure DOFs:
$$
\widecheck{\underline{\boldsymbol{U}}}_h^k \coloneq \bigtimes_{T\in\t_h}\underline{\boldsymbol{U}}_T^k, \qquad P_{\bulk,h}^k \coloneq \bigtimes_{T\in\t_h}P_{\bulk,T}^k,
$$
with local spaces $\underline{\boldsymbol{U}}_T^k$ and $P_{\bulk,T}^k$ defined by~\eqref{eq:UT.PT}.
We will also need the following subspace of $\widecheck{\underline{\boldsymbol{U}}}_h^k$ that incorporates (i) the continuity of flux unknowns at each interface $F\in\f_h^{\rm i}\setminus\f_h^\Gamma$ not included in the fracture and (ii) the strongly enforced homogeneous Neumann boundary condition on $\partial\Omega_\bulk^{\rm N}$:
\begin{equation}\label{eq:Uh0}
  \underline{\boldsymbol{U}}_{h,0}^k \coloneq \{\underline{\boldsymbol{v}}_h \in \widecheck{\underline{\boldsymbol{U}}}_h^k ~|~ [\![\underline{\boldsymbol{v}}_{h}]\!]_{F}  = 0 ~\forall F\in\f_h^{\text{i}}\setminus\f_h^{\Gamma} \text{ and } v_F = 0 ~\forall F\in\f_{h}^{\rm N}\},
\end{equation}
where, for all $F\in\f_h^{\rm b}$, we have set $v_F\coloneq v_{TF}$ with $T$ denoting the unique mesh element such that $F\in\f_T$, while, for all $F\in\f_h^{\rm i}$ with $F\subset\partial T_1\cap\partial T_2$ for distinct mesh elements $T_1,T_2\in\t_h$, the jump operator is such that
$$
[\![\underline{\boldsymbol{v}}_{h}]\!]_{F}\coloneq v_{T_1F} + v_{T_2F}.
$$
Notice that this quantity is the discrete counterpart of the jump of the normal flux component since, for $i\in\{1,2\}$, $v_{T_iF}$ can be interpreted as the normal flux exiting $T_i$.

We also define the global space of fracture-based pressure unknowns and its subspace with strongly enforced homogeneous Dirichlet boundary condition on $\partial\Gamma^{\rm D}$ as follows:
\begin{align*}
  \underline{P}_{\Gamma,h}^k \coloneq \Bigg(
  \bigtimes_{F\in\f_h^\Gamma}\mathbb{P}^k(F)
  \Bigg)\times\Bigg(
  \bigtimes_{V\in\mathcal{V}_{h}}\mathbb{P}(V)
  \Bigg), \quad
  \underline{P}_{\Gamma,h,0}^k \coloneq \{ \underline{q}_{h}^\Gamma \in \underline{P}_{\Gamma,h}^k ~|~ q_V^\Gamma = 0~\forall V \in \mathcal{V}_{h}^\text{D} \}.
\end{align*}
A generic element $\underline{q}_h^\Gamma$ of $\underline{P}_{\Gamma,h}^k$ is decomposed as $\underline{q}_h^\Gamma=( (q_F)_{F\in\f_h^\Gamma}, (q_V)_{V\in\mathcal{V}_h})$ and, for all $F\in\f_h^\Gamma$, we denote by  $\underline{q}_F^\Gamma=(q_F^\Gamma,(q_V^\Gamma)_{v\in\mathcal{V}_F})$ its restriction to $\underline{P}_{\Gamma,F}^k$.

\subsubsection{Discrete norms and interpolators}\label{sec:discrete:global:norms}

We equip the DOF spaces $\widecheck{\underline{\boldsymbol{U}}}_h^k$, $P_{\bulk,h}^k$, and $\underline{P}_{\Gamma,h}^k$ respectively, with the norms $\|{\cdot}\|_{\boldsymbol{U},\xi,h}$ and $\|{\cdot}\|_{\bulk,h}$, and the seminorm $\|{\cdot}\|_{\Gamma,h}$ such that for all $\underline{\boldsymbol{v}}_h\in \underline{\boldsymbol{U}}_h^k$, all $q_h\in P_{\bulk,h}^k$, and all $\underline{q}_h^\Gamma\in \underline{P}_{\Gamma,h}^k$,
\begin{equation*}
  \begin{gathered}
    \|\underline{\boldsymbol{v}}_h\|_{\boldsymbol{U},\xi,h}^2 \coloneq \hspace*{-3px}\sum_{T\in\t_{h}} \|\underline{\boldsymbol{v}}_T\|_{\boldsymbol{U},T}^2 + |\underline{\boldsymbol{v}}_h|_{\xi,h}^2,\quad
    |\underline{\boldsymbol{v}}_h|_{\xi,h}^2\coloneq\hspace*{-3px}\sum_{F\in\f^\Gamma_{h}} \hspace*{-2px}\Big(\lambda_F^\xi\|[\![\underline{\boldsymbol{v}}_h]\!]_F \|_F^2 + \lambda_F\|\{\!\{\underline{\boldsymbol{v}}_h\}\!\}_F \|_F^2 \Big),
    \\
    \|q_h\|_{\bulk,h}^2 \coloneq \sum_{T\in\t_{h}} \|q_T\|_{\bulk,T}^2,\quad \|\underline{q}_h^\Gamma\|_{\Gamma,h}^2 \coloneq \sum_{F\in\f_h^\Gamma} \|\underline{q}_F^\Gamma\|_{\Gamma,F}^2,
  \end{gathered}
\end{equation*}
where, for the sake of brevity, we have set $\lambda_F\coloneq(\lambda_\Gamma)_{|F}$ and $\lambda_F^\xi\coloneq(\lambda_\Gamma^\xi)_{|F}$ (see~\eqref{eq:lambda.lambdaxi} for the definition of $\lambda_\Gamma$ and $\lambda_\Gamma^\xi$), and we have defined the average operator such that, for all $F\in\f_h^\Gamma$ and all $\underline{\boldsymbol{v}}_h\in\widecheck{\underline{\boldsymbol{U}}}_h^k$,
$$
  \{\!\{ \underline{\boldsymbol{v}}_{h} \}\!\}_F \coloneq \frac12\sum_{T\in\t_F}v_{TF} (\normal_{TF}\cdot\normal_\Gamma).
$$
Using the arguments of~\cite[Proposition 5]{Di-Pietro.Ern:15}, it can be proved that $\|{\cdot}\|_{\Gamma,h}$ is a norm on $\underline{P}_{\Gamma,h,0}^k$.

Let now $H^1(\t_h)^2$ denote the space spanned by vector-valued functions whose restriction to each mesh element $T\in\t_h$ lies in $H^1(T)^2$. 
We define the global interpolators $\underline{\boldsymbol{I}}_{h}^k :H^1(\t_h)^2 \rightarrow \widecheck{\underline{\boldsymbol{U}}}_h^k$ and $\underline{{I}}_{h}^k : H^1(\Gamma) \rightarrow \underline{P}_{\Gamma,h}^k$ such that, for all $\boldsymbol{v}\in H^1(\t_h)^2$ and all $q\in H^1(\Gamma)$,
\begin{equation}\label{def:globalinterpoperator}
  \underline{\boldsymbol{I}}_{h}^k \boldsymbol{v} \coloneq
  \big( \underline{\boldsymbol{I}}_{T}^{k}\boldsymbol{v}_{|T} \big)_{T\in\t_{h}},
  \qquad
  \underline{I}_{h}^k q \coloneq
  \big( (\pi_F^k q)_{F\in\f_h^{\Gamma}}, (q(V))_{V\in\mathcal{V}_h}\big),
\end{equation}
where, for all $T\in\t_h$, the local interpolator $\underline{\boldsymbol{I}}_{T}^{k}$ is defined by~\eqref{eq:IT}.
We also denote by $\pi_h^k$ the global $L^2$-orthogonal projector on $P_{\bulk,h}^k$ such that, for all $q\in L^1(\Omega_\bulk)$,
$$
(\pi_h^k q)_{|T} \coloneq \pi_T^k q_{|T}\qquad\forall T\in\t_h.
$$

\subsubsection{Discrete problem}\label{subsec:discretprobl}

At the discrete level, the counterparts of the continuous bilinear forms defined in Section~\ref{subsec:weakform} are the bilinear forms
$a_{h}^\xi : \widecheck{\underline{\boldsymbol{U}}}_h^k\times\widecheck{\underline{\boldsymbol{U}}}_h^k \rightarrow \Real$, $b_{h} : \widecheck{\underline{\boldsymbol{U}}}_h^k\times P_{\bulk,h}^k\rightarrow \Real$, $c_{h} : \widecheck{\underline{\boldsymbol{U}}}_h^k\times \underline{P}_{\Gamma,h}^k\rightarrow \Real$, and $d_{h} : \underline{P}_{\Gamma,h}^k\times\underline{P}_{\Gamma,h}^k \rightarrow \Real$ such that
\begin{align}\label{def:ah}
  a_{h}^\xi(\underline{\boldsymbol{u}}_h,\underline{\boldsymbol{v}}_h) \coloneq &\sum_{T\in\t_{h}} m_T(\boldsymbol{\underline{u}}_T,\boldsymbol{\underline{v}}_T)\\
&+ \sum_{F\in\f_{h}^\Gamma}\left( (\lambda_F^\xi[\![\underline{\boldsymbol{u}}_{h}]\!]_F,[\![\underline{\boldsymbol{v}}_{h}]\!]_F)_F + (\lambda_F \{\!\{\underline{\boldsymbol{u}}_{h}\}\!\}_F,\{\!\{\underline{\boldsymbol{v}}_{h}\}\!\}_F)_F\right),%
  \notag\\
  b_{h}(\underline{\boldsymbol{u}}_h,p_h) &\coloneq \sum_{T\in\t_{h}} (D_T^k\underline{\boldsymbol{u}}_T, p_T)_T,\label{def:bh}\\
c_{h}(\boldsymbol{\underline{u}}_h,\underline{p}_{h}^\Gamma) &\coloneq \sum_{F\in\f_{h}^\Gamma} ([\![ \boldsymbol{\underline{u}}_h ]\!]_F,p_F^\Gamma)_F,\label{def:ch}\\
d_{h}(\underline{p}_{h}^\Gamma,\underline{q}_{h}^\Gamma) &\coloneq \sum_{F\in\f_{h}^\Gamma}d_F(\underline{p}_F^\Gamma,\underline{q}_F^\Gamma).\label{def:dh}
\end{align}
The HHO discretization of problem \eqref{eq:weak} reads : Find $(\boldsymbol{\underline{u}}_h,p_h,\underline{p}_{h,0}^\Gamma)\in\boldsymbol{\underline{U}}_{h,0}^k \times P_{\bulk,h}^k \times \underline{P}_{\Gamma,h,0}^k$ such that, for all $(\boldsymbol{\underline{v}}_h,q_h,\underline{q}_{h}^\Gamma)\in\boldsymbol{\underline{U}}_{h,0}^k \times P_{\bulk,h}^k \times \underline{P}_{\Gamma,h,0}^k$,
\begin{subequations}
  \label{eq:discret}
  \begin{alignat}{5}
  	&a_h^\xi(\underline{\boldsymbol{u}}_h,\underline{\boldsymbol{v}}_h) &- b_h(\underline{\boldsymbol{v}}_h,p_h) &+ c_{h}(\underline{\boldsymbol{v}}_h,\underline{p}_{h,0}^\Gamma) &=&-\sum_{F\in\f_h^\text{D}} (g_\bulk,v_F)_{F},   \label{eq:discret:1}
  	\\
        &b_h(\underline{\boldsymbol{u}}_h,q_h)&& &=& \sum_{T\in\t_{h}} (f,q_T)_T, \label{eq:discret:2}\\
        -&c_{h}(\underline{\boldsymbol{u}}_h,\underline{q}_{h}^\Gamma) && + d_{h}(\underline{p}_{h,0}^\Gamma, \underline{q}_{h}^\Gamma) ~&=& \sum_{F\in\f^\Gamma_{h}}(\ell_{F} f_\Gamma,q_F^\Gamma)_F - d_{h}(\underline{p}_{\mathrm{D},h}^\Gamma, \underline{q}_{h}^\Gamma), \label{eq:discret:3}
  \end{alignat}
\end{subequations}
where, for all $F\in\f_{h}^{\rm D}$, we have set $v_{F}\coloneq v_{TF}$ with $T\in\t_{h}$ unique element such that $F\subset\partial T\cap\partial\Omega$ in~\eqref{eq:discret:1}, while $\underline{p}_{{\rm D},h}^\Gamma = \big( (p_{{\rm D},F}^\Gamma)_{F\in\f_h^\Gamma}, (p_{{\rm D},V}^\Gamma)_{V\in{\cal V}_h}\big)\in \underline{P}_{\Gamma,h}^k$ is such that
\begin{align*}
  p_{{\rm D},F}^\Gamma \equiv 0\quad\forall F\in\f^\Gamma_{h},\qquad
  p_{{\rm D},V}^\Gamma = g_\Gamma(V)\quad\forall V\in\mathcal{V}_{h}^\text{D},\qquad
  p_{{\rm D},V}^\Gamma = 0\quad\forall V\in\mathcal{V}_{h}\setminus \mathcal{V}_{h}^\text{D}.
\end{align*}
The discrete fracture pressure $\underline{p}_{h}^\Gamma\in\underline{P}_{\Gamma,h}^k$ is finally computed as $\underline{p}_{h}^\Gamma = \underline{p}_{h,0}^\Gamma + \underline{p}_{{\rm D},h}^\Gamma.$

\begin{remark}[Implementation]\label{rem:implementation}
  In the practical implementation, all bulk flux DOFs and all bulk pressure DOFs up to one constant value per element can be statically condensed by solving small saddle point problems inside each element.
  This corresponds to the first static condensation procedure discussed in~\cite[Section~3.4]{Di-Pietro.Ern:16}, to which we refer the reader for further details.
\end{remark}

We next write a more compact equivalent reformulation of problem~\eqref{eq:discret}.
Define the Cartesian product space $\underline{\boldsymbol{X}}_h^k \coloneq \boldsymbol{\underline{U}}_{h,0}^k \times P_{\bulk,h}^k \times \underline{P}_{\Gamma,h,0}^k$ as well as the bilinear form $\mathcal{A}_h^{\xi} :\underline{\boldsymbol{X}}_h^k\times\underline{\boldsymbol{X}}_h^k\rightarrow \Real$ such that
\begin{equation}\label{def:Ahxi}
  \begin{aligned}
    \mathcal{A}_h^{\xi}((\boldsymbol{\underline{u}}_h,p_h,\underline{p}_h^\Gamma),(\boldsymbol{\underline{v}}_h,q_h,\underline{q}_h^\Gamma))
    &\coloneq a_h^\xi(\underline{\boldsymbol{u}}_h,\underline{\boldsymbol{v}}_h) + b_h(\underline{\boldsymbol{u}}_h,q_h) - b_h(\underline{\boldsymbol{v}}_h,p_h)
  \\
  &\qquad + c_{h}(\underline{\boldsymbol{v}}_h,\underline{p}_{h}^\Gamma) - c_{h}(\underline{\boldsymbol{u}}_h,\underline{q}_{h}^\Gamma) + d_{h}(\underline{p}_{h}^\Gamma, \underline{q}_{h}^\Gamma).
  \end{aligned}
\end{equation}
Then, problem \cref{eq:discret} is equivalent to:
Find $(\boldsymbol{\underline{u}}_h,p_h,\underline{p}_{h,0}^\Gamma)\in\underline{\boldsymbol{X}}_h^k$ such that, for all \mbox{$(\boldsymbol{\underline{v}}_h,q_h,\underline{q}_h^\Gamma)\in \underline{\boldsymbol{X}}_h^k$,}
\begin{equation}\label{eq:dicret:oneline}
  \begin{aligned}
    \mathcal{A}_h^{\xi}((\boldsymbol{\underline{u}}_h,p_h,\underline{p}_{h,0}^\Gamma),(\boldsymbol{\underline{v}}_h,q_h,\underline{q}_h^\Gamma))
    =
    \sum_{T\in\t_{h}} (f,q_T)_T&
    + \sum_{F\in\f^\Gamma_{h}}(\ell_F f_\Gamma,q_F^\Gamma)_F\\
    &- \sum_{F\in\f_h^\text{D}} (g_\bulk,v_F)_{F}- d_{h}(\underline{p}_{{\rm D},h}^\Gamma, \underline{q}_{h}^\Gamma).
    \end{aligned}
\end{equation}
\begin{remark}[Extension to three space dimensions]\label{rem:3d}
  The proposed method can be extended to the case of a three-dimensional domain with fracture corresponding to the intersection of the domain with a plane.
  The main differences are linked to the fracture terms, and can be summarized as follows:
  \begin{inparaenum}[(i)]
  \item the tangential permeability of the fracture is a uniformly elliptic, $2\times 2$ matrix-valued field instead of a scalar;
  \item the fracture is discretized by means of a two-dimensional mesh $\mathcal{F}_{h}^{\Gamma}$ composed of element faces, and vertex-based DOFs are replaced by discontinuous polynomials of degree up to $k$ on the skeleton (i.e., the union of the edges) of $\mathcal{F}_{h}^{\Gamma}$;
  \item all the terms involving pointwise evaluations at vertices are replaced by integrals on the edges of $\mathcal{F}_{h}^{\Gamma}$.
  \end{inparaenum}
  Similar stability and error estimates as in the two-dimensional case can be proved in three space dimensions.
  A difference is that the right-hand side of the error estimate will additionally depend on the local anisotropy ratio of the tangential permeability of the fracture, arguably with a power of $\nicefrac12$.
\end{remark}%

\subsection{Main results}\label{subsec:results}

In this section we report the main results of the analysis of our method, postponing the details of the proofs to Section~\ref{sec:stability:proof}.
For the sake of simplicity, we will assume that
\begin{equation}\label{eq:dirichlet.only}
  \partial\Omega_\bulk^{\rm N}=\emptyset,\quad g_\bulk\equiv 0,\qquad
  \qquad \partial\Gamma^{\rm N}=\emptyset,\quad g_\Gamma\equiv 0
\end{equation}
which means that homogeneous Dirichlet boundary conditions on the pressure are enforced on both the external boundary of the bulk region and on the boundary of the fracture.
This corresponds to the situation when the motion of the fluid is driven by the volumetric source terms $f$ in the bulk region and $f_{\Gamma}$ in the fracture.
The results illustrated below and in Section~\ref{sec:stability:proof} can be adapted to more general boundary conditions at the price of heavier notations and technicalities that we want to avoid here.

In the error estimate of Theorem~\ref{thm:energyerror} below, we track explicitly the dependence of the multiplicative constants on the following quantites and bounds thereof: the bulk permeability $\boldsymbol{K}$, the tangential fracture permeability $\kappa_\Gamma^\tau$, the normal fracture permeability $\kappa_\Gamma^n$, and the fracture thickness $\ell_\Gamma$, which we collectively refer to in the following as the \emph{problem data}.

We equip the space $\underline{\boldsymbol{X}}_h^k$ with the norm $\|{\cdot}\|_{\boldsymbol{X},h}$ such that, for all $(\boldsymbol{\underline{v}}_h,q_h,\underline{q}_h^\Gamma)\in\underline{\boldsymbol{X}}_h^k$,
\begin{align}\label{def:Xnorm}
\|(\boldsymbol{\underline{v}}_h,q_h,\underline{q}_h^\Gamma)\|_{\boldsymbol{X},h}^2 \coloneq \|\boldsymbol{\underline{v}}_h\|_{\boldsymbol{U},\xi,h}^2 + \|q_h\|_{\bulk,h}^2 + \|\underline{q}_h^\Gamma\|_{\Gamma,h}^2.
\end{align}

\begin{theorem}[Stability]\label{thm:stability}
  Assume~\eqref{eq:dirichlet.only}.
  Then, there exists a real number $\gamma>0$ independent of $h$, but possibly depending on the problem geometry, on $\varrho$, $k$, and on the problem data, such that, for all $\underline{\boldsymbol{z}}_h\in\underline{\boldsymbol{X}}_h^k$,
  \begin{equation}\label{thm:stability:eq}
    \|\underline{\boldsymbol{z}}_h\|_{\boldsymbol{X},h}
    \le\gamma
    \sup_{\underline{\boldsymbol{y}}_h
      \in\underline{\boldsymbol{X}}_h^k, \|\underline{\boldsymbol{y}}_h\|_{\boldsymbol{X},h} = 1}
    \mathcal{A}_h^{\xi}(\underline{\boldsymbol{z}}_h,\underline{\boldsymbol{y}}_h).
  \end{equation}
  Consequently, problem \eqref{eq:dicret:oneline} admits a unique solution.
\end{theorem}

\begin{proof}
  See Section~\ref{sec:stability:proof}.
\end{proof}

We next provide an a priori estimate of the discretization error.
Let $(\boldsymbol{u},p,p_\Gamma)\in\boldsymbol{U}\times P_\bulk\times P_\Gamma$ and $(\underline{\boldsymbol{u}}_h,p_h,\underline{p}_h^\Gamma)\in\underline{\boldsymbol{X}}_h$ denote, respectively, the unique solutions to problems \eqref{eq:weak} and~\eqref{eq:discret} (recall that, owing to~\eqref{eq:dirichlet.only}, $p_\Gamma=p_{\Gamma,0}$ and $\underline{p}_h^\Gamma=\underline{p}_{h,0}^\Gamma$).
We further assume that $\boldsymbol{u}\in H^1(\t_h)^2$, and we estimate the error defined as the difference between the discrete solution $(\underline{\boldsymbol{u}}_h,p_h,\underline{p}_h^\Gamma)$ and the following projection of the exact solution:
\begin{align}\label{def:interpsol}
  (\widehat{\underline{\boldsymbol{u}}}_h, \widehat{p}_h, \underline{\widehat{p}}_h^\Gamma)
  \coloneq
  (\boldsymbol{\underline{I}}_h^k\boldsymbol{u}, \pi_h^k p, \underline{I}_h^k p_{\Gamma})\in\boldsymbol{\underline{X}}_h.
\end{align}

\begin{theorem}[Error estimate]\label{thm:energyerror}
  Let~\eqref{eq:dirichlet.only} hold true, and denote by $(\boldsymbol{u},p,p_\Gamma)\in\boldsymbol{U}\times P_\bulk\times P_\Gamma$ and $(\underline{\boldsymbol{u}}_h,p_h,\underline{p}_h^\Gamma)\in\underline{\boldsymbol{X}}_h^k$ the unique solutions to problems \eqref{eq:weak} and~\eqref{eq:discret}, respectively.
  Assume the additional regularity $p_{|T}\in H^{k+2}(T)$ for all $T\in\t_h$ and $(p_\Gamma)_{|F}\in H^{k+2}(F)$ for all $F\in\f_h^\Gamma$.
  Then, there exist a real number $C>0$ independent of $h$ and of the problem data, but possibly depending on $\varrho$ and $k$, such that
  \begin{equation}\label{eq:energyerror}
    \begin{multlined}
      \|\boldsymbol{\underline{u}}_h - \widehat{\boldsymbol{\underline{u}}}_h \|_{\boldsymbol{U},\xi,h}
      + \| \underline{p}_h^\Gamma - \widehat{\underline{p}}_h^\Gamma \|_{\Gamma,h}
      + \chi \|p_h-\widehat{p}_h\|_{B,h}
      \\
      \le C\left(
      \sum_{T\in\t_h}\varrho_{\bulk,T}\uKBT h_T^{2(k+1)} \|p\|_{H^{k+2}(T)}^2  
      + \sum_{F\in\f_h^\Gamma}K_F h_F^{2(k+1)}\|p_\Gamma\|_{H^{k+2}(F)}^2
      \right)^{\nicefrac12},
    \end{multlined}
  \end{equation}
  with $\chi>0$ independent of $h$ but possibly depending on $\varrho$, $k$, and on the problem geometry and data.
\end{theorem}
\begin{proof}
  See Section~\ref{sec:stability:proof}.
\end{proof}
\begin{remark}[Error norm and robustness]\label{rem:robustness}
  The error norm in the left-hand side of~\eqref{eq:energyerror} is selected so as to prevent the right-hand side from depending on the global bulk anisotropy ratio $\varrho_\bulk$ (see~\eqref{eq:varrho.B}).
  As a result, for both the error on the bulk flux measured by $\|\boldsymbol{\underline{u}}_h - \widehat{\boldsymbol{\underline{u}}}_h \|_{\boldsymbol{U},\xi,h}$ and the error on the fracture pressure measured by $\| \underline{p}_h^\Gamma - \widehat{\underline{p}}_h^\Gamma \|_{\Gamma,h}$, we have: (i) as in more standard discretizations, full robustness  with respect to the heterogeneity of $\boldsymbol{K}$ and $K_{\Gamma}$, meaning that the right-hand side does not depend on the jumps of these quantities; (ii) partial robustness with respect to the anisotropy of the bulk permeability, with a mild dependence on the square root of $\varrho_{\bulk,T}$ (see~\eqref{eq:varrho.BT}).
  As expected, robustness is not obtained for the $L^{2}$-error on the pressure in the bulk, which is multiplied by a data-dependent real number $\chi$.
  
  In the context of primal HHO methods, more general, possibly nonlinear diffusion terms including, as a special case, variable diffusion tensors inside the mesh elements have been recently considered in \cite{Di-Pietro.Droniou:16,Di-Pietro.Droniou.bis:16}. In this case, one can expect the error estimate to depend on the square root of the ratio of the Lipschitz module and the coercivity constant of the diffusion field; see~\cite[Eq. (3.1)]{Di-Pietro.Droniou.bis:16}. The extension to the mixed HHO formulation considered here for the bulk region can be reasonably expected to behave in a similar way. The details are postponed to a future work.
\end{remark}
\begin{remark}[$L^2$-supercloseness of bulk and fracture pressures]\label{rem:supercloseness}
  Using arguments based on the Aubin--Nitsche trick, one could prove under further regularity assumptions on the problem geometry that the $L^2$-errors $\|p_h-\widehat{p}_h\|_{\bulk,h}$ and $\|p_h^\Gamma-\widehat{p}_h^\Gamma\|_{\Gamma,h}$ converge as $h^{k+2}$, where we have denoted by $p_h^\Gamma$ and $\widehat{p}_h^\Gamma$ the broken polynomial functions on $\Gamma$ such that $(p_h^\Gamma)_{|F}\coloneq p_F^\Gamma$ and $(\widehat{p}_h^\Gamma)\coloneq\widehat{p}_F^\Gamma$ for all $F\in\f_h^\Gamma$.
  This supercloseness behaviour is typical of HHO methods (cf., e.g.,~\cite[Theorem~7]{Di-Pietro.Ern:16} and~\cite[Theorem~10]{Di-Pietro.Ern.ea:14}), and is confirmed by the numerical example of Section~\ref{sec:num.tests:convergence}; see, in particular, Figures~\ref{fig:cv.2} and~\ref{fig:cv.1}.
\end{remark}


\section{Numerical results}\label{sec:num.tests}
\pgfqkeys{/pgfplots}{ cycle list name = black white }
We provide an extensive numerical validation of the method on a set of model problems.

\subsection{Convergence}\label{sec:num.tests:convergence}

We start by a non physical numerical test that demonstrates the convergence properties of the method.
We approximate problem \eqref{eq:discret} on the square domain $\Omega = (0,1)^2$ crossed by the fracture $\Gamma = \{\boldsymbol{x}\in\Omega ~|~ x_1 = 0.5\}$ with $\partial\Omega_\bulk^\text{N} = \partial\Gamma^\text{N} = \emptyset$. 
We consider the exact solution corresponding to the bulk and fracture pressures
\begin{align*}
  p(\boldsymbol{x}) =
  \begin{cases}
    \sin(4 x_1) \cos(\pi x_2) & \text{if $x_1 < 0.5$} \\
    \cos(4 x_1) \cos(\pi x_2) & \text{if $x_1>0.5$}
  \end{cases},\qquad
  p_\Gamma(\boldsymbol{x}) = \xi (\cos(2) + \sin(2)) \cos(\pi x_2),
\end{align*}
and let $\boldsymbol{u}_{|\Omega_{\bulk,i}} = - \nabla p_{|\Omega_{\bulk,i}}$ for $i\in\{1,2\}$.
We take here $\xi=\nicefrac34$, $\kappa_\Gamma^\tau = 1$, {$\ell_\Gamma = 0.01$ and 
\begin{align}\label{def:perm.conv.anal}
	\boldsymbol{K} \coloneq \left[
  \begin{tabular}{cc}
    $\kappa_{\Gamma}^n/(2\ell_\Gamma)$ & $0$ \\
    $0$ & $1$
  \end{tabular}\right],
\end{align}
where $\kappa_n^{\Gamma} > 0$ is the normal permeability of the fracture.}
The expression of the source terms $f$, $f_\Gamma$, and of the Dirichlet data $g_\bulk$ and $g_\Gamma$ are inferred from \eqref{eq:discret}.
It can be checked that, with this choice, the quantities $[\![ p ]\!]_\Gamma$, $[\![ \boldsymbol{u} ]\!]_\Gamma$, and $\{\!\{\boldsymbol{u}\}\!\}_\Gamma$ are not identically zero on the fracture.
\begin{figure}\centering
  \begin{minipage}[t]{0.25\textwidth}
    \begin{center}
      \begin{tikzpicture}[scale=3]
        \draw plot coordinates {(0.00000000,0.50000000) (0.25000000,0.50000000) (0.15000000,0.65000000) (0.00000000,0.50000000)};
        \draw plot coordinates {(0.25000000,0.50000000) (0.32500000,0.67500000) (0.15000000,0.65000000) (0.25000000,0.50000000)};
        \draw plot coordinates {(0.25000000,0.50000000) (0.50000000,0.50000000) (0.32500000,0.67500000) (0.25000000,0.50000000)};
        \draw plot coordinates {(0.50000000,0.50000000) (0.50000000,0.75000000) (0.32500000,0.67500000) (0.50000000,0.50000000)};
        \draw plot coordinates {(0.50000000,0.75000000) (0.35000000,0.85000000) (0.32500000,0.67500000) (0.50000000,0.75000000)};
        \draw plot coordinates {(0.50000000,0.75000000) (0.50000000,1.00000000) (0.35000000,0.85000000) (0.50000000,0.75000000)};
        \draw plot coordinates {(0.50000000,1.00000000) (0.25000000,1.00000000) (0.35000000,0.85000000) (0.50000000,1.00000000)};
        \draw plot coordinates {(0.35000000,0.85000000) (0.25000000,1.00000000) (0.17500000,0.82500000) (0.35000000,0.85000000)};
        \draw plot coordinates {(0.25000000,1.00000000) (0.00000000,1.00000000) (0.17500000,0.82500000) (0.25000000,1.00000000)};
        \draw plot coordinates {(0.00000000,1.00000000) (0.00000000,0.75000000) (0.17500000,0.82500000) (0.00000000,1.00000000)};
        \draw plot coordinates {(0.00000000,0.75000000) (0.15000000,0.65000000) (0.17500000,0.82500000) (0.00000000,0.75000000)};
        \draw plot coordinates {(0.00000000,0.75000000) (0.00000000,0.50000000) (0.15000000,0.65000000) (0.00000000,0.75000000)};
        \draw plot coordinates {(0.15000000,0.65000000) (0.32500000,0.67500000) (0.17500000,0.82500000) (0.15000000,0.65000000)};
        \draw plot coordinates {(0.32500000,0.67500000) (0.35000000,0.85000000) (0.17500000,0.82500000) (0.32500000,0.67500000)};
        \draw plot coordinates {(0.50000000,0.50000000) (0.75000000,0.50000000) (0.65000000,0.65000000) (0.50000000,0.50000000)};
        \draw plot coordinates {(0.75000000,0.50000000) (0.82500000,0.67500000) (0.65000000,0.65000000) (0.75000000,0.50000000)};
        \draw plot coordinates {(0.75000000,0.50000000) (1.00000000,0.50000000) (0.82500000,0.67500000) (0.75000000,0.50000000)};
        \draw plot coordinates {(1.00000000,0.50000000) (1.00000000,0.75000000) (0.82500000,0.67500000) (1.00000000,0.50000000)};
        \draw plot coordinates {(1.00000000,0.75000000) (0.85000000,0.85000000) (0.82500000,0.67500000) (1.00000000,0.75000000)};
        \draw plot coordinates {(1.00000000,0.75000000) (1.00000000,1.00000000) (0.85000000,0.85000000) (1.00000000,0.75000000)};
        \draw plot coordinates {(1.00000000,1.00000000) (0.75000000,1.00000000) (0.85000000,0.85000000) (1.00000000,1.00000000)};
        \draw plot coordinates {(0.85000000,0.85000000) (0.75000000,1.00000000) (0.67500000,0.82500000) (0.85000000,0.85000000)};
        \draw plot coordinates {(0.75000000,1.00000000) (0.50000000,1.00000000) (0.67500000,0.82500000) (0.75000000,1.00000000)};
        \draw plot coordinates {(0.50000000,1.00000000) (0.50000000,0.75000000) (0.67500000,0.82500000) (0.50000000,1.00000000)};
        \draw plot coordinates {(0.50000000,0.75000000) (0.65000000,0.65000000) (0.67500000,0.82500000) (0.50000000,0.75000000)};
        \draw plot coordinates {(0.50000000,0.75000000) (0.50000000,0.50000000) (0.65000000,0.65000000) (0.50000000,0.75000000)};
        \draw plot coordinates {(0.65000000,0.65000000) (0.82500000,0.67500000) (0.67500000,0.82500000) (0.65000000,0.65000000)};
        \draw plot coordinates {(0.82500000,0.67500000) (0.85000000,0.85000000) (0.67500000,0.82500000) (0.82500000,0.67500000)};
        \draw plot coordinates {(0.00000000,0.00000000) (0.25000000,0.00000000) (0.15000000,0.15000000) (0.00000000,0.00000000)};
        \draw plot coordinates {(0.25000000,0.00000000) (0.32500000,0.17500000) (0.15000000,0.15000000) (0.25000000,0.00000000)};
        \draw plot coordinates {(0.25000000,0.00000000) (0.50000000,0.00000000) (0.32500000,0.17500000) (0.25000000,0.00000000)};
        \draw plot coordinates {(0.50000000,0.00000000) (0.50000000,0.25000000) (0.32500000,0.17500000) (0.50000000,0.00000000)};
        \draw plot coordinates {(0.50000000,0.25000000) (0.35000000,0.35000000) (0.32500000,0.17500000) (0.50000000,0.25000000)};
        \draw plot coordinates {(0.50000000,0.25000000) (0.50000000,0.50000000) (0.35000000,0.35000000) (0.50000000,0.25000000)};
        \draw plot coordinates {(0.50000000,0.50000000) (0.25000000,0.50000000) (0.35000000,0.35000000) (0.50000000,0.50000000)};
        \draw plot coordinates {(0.35000000,0.35000000) (0.25000000,0.50000000) (0.17500000,0.32500000) (0.35000000,0.35000000)};
        \draw plot coordinates {(0.25000000,0.50000000) (0.00000000,0.50000000) (0.17500000,0.32500000) (0.25000000,0.50000000)};
        \draw plot coordinates {(0.00000000,0.50000000) (0.00000000,0.25000000) (0.17500000,0.32500000) (0.00000000,0.50000000)};
        \draw plot coordinates {(0.00000000,0.25000000) (0.15000000,0.15000000) (0.17500000,0.32500000) (0.00000000,0.25000000)};
        \draw plot coordinates {(0.00000000,0.25000000) (0.00000000,0.00000000) (0.15000000,0.15000000) (0.00000000,0.25000000)};
        \draw plot coordinates {(0.15000000,0.15000000) (0.32500000,0.17500000) (0.17500000,0.32500000) (0.15000000,0.15000000)};
        \draw plot coordinates {(0.32500000,0.17500000) (0.35000000,0.35000000) (0.17500000,0.32500000) (0.32500000,0.17500000)};
        \draw plot coordinates {(0.50000000,0.00000000) (0.75000000,0.00000000) (0.65000000,0.15000000) (0.50000000,0.00000000)};
        \draw plot coordinates {(0.75000000,0.00000000) (0.82500000,0.17500000) (0.65000000,0.15000000) (0.75000000,0.00000000)};
        \draw plot coordinates {(0.75000000,0.00000000) (1.00000000,0.00000000) (0.82500000,0.17500000) (0.75000000,0.00000000)};
        \draw plot coordinates {(1.00000000,0.00000000) (1.00000000,0.25000000) (0.82500000,0.17500000) (1.00000000,0.00000000)};
        \draw plot coordinates {(1.00000000,0.25000000) (0.85000000,0.35000000) (0.82500000,0.17500000) (1.00000000,0.25000000)};
        \draw plot coordinates {(1.00000000,0.25000000) (1.00000000,0.50000000) (0.85000000,0.35000000) (1.00000000,0.25000000)};
        \draw plot coordinates {(1.00000000,0.50000000) (0.75000000,0.50000000) (0.85000000,0.35000000) (1.00000000,0.50000000)};
        \draw plot coordinates {(0.85000000,0.35000000) (0.75000000,0.50000000) (0.67500000,0.32500000) (0.85000000,0.35000000)};
        \draw plot coordinates {(0.75000000,0.50000000) (0.50000000,0.50000000) (0.67500000,0.32500000) (0.75000000,0.50000000)};
        \draw plot coordinates {(0.50000000,0.50000000) (0.50000000,0.25000000) (0.67500000,0.32500000) (0.50000000,0.50000000)};
        \draw plot coordinates {(0.50000000,0.25000000) (0.65000000,0.15000000) (0.67500000,0.32500000) (0.50000000,0.25000000)};
        \draw plot coordinates {(0.50000000,0.25000000) (0.50000000,0.00000000) (0.65000000,0.15000000) (0.50000000,0.25000000)};
        \draw plot coordinates {(0.65000000,0.15000000) (0.82500000,0.17500000) (0.67500000,0.32500000) (0.65000000,0.15000000)};
        \draw plot coordinates {(0.82500000,0.17500000) (0.85000000,0.35000000) (0.67500000,0.32500000) (0.82500000,0.17500000)};
        
        \draw[very thick] (0.5,0) -- (0.5,1);
        \draw plot[mark=*, mark size = 0.3pt] (0.5,0.);
        \draw plot[mark=*, mark size = 0.3pt] (0.5,0.25);
        \draw plot[mark=*, mark size = 0.3pt] (0.5,0.5);
        \draw plot[mark=*, mark size = 0.3pt] (0.5,0.75);
        \draw plot[mark=*, mark size = 0.3pt] (0.5,1);
      \end{tikzpicture}
      \subcaption{Triangular\label{fig:meshfamilies:triangular}}
    \end{center}
  \end{minipage}
  \hspace{0.5cm}
  \begin{minipage}[t]{0.25\textwidth}
    \begin{center}
      \begin{tikzpicture}[scale=3]
        \draw plot coordinates {(0.00000000,0.00000000) (0.2,0.00000000) (0.2,0.2) (0.00000000,0.2) (0.00000000,0.00000000)};
        \draw plot coordinates {(0.20000000,0.00000000) (0.4,0.0) (0.40000000,0.2) (0.20000000,0.20000000)};
        \draw plot coordinates {(0.40000000,0.00000000) (0.6,0.0) (0.60000000,0.2) (0.40000000,0.20000000)};
        \draw plot coordinates {(0.60000000,0.00000000) (0.8,0.0) (0.80000000,0.2) (0.60000000,0.20000000)};
        \draw plot coordinates {(0.80000000,0.00000000) (1,0.0) (1.0000000,0.2) (0.80000000,0.20000000)};

        \draw plot coordinates {(0.00000000,0.20000000) (0.0,0.4) (0.20000000,0.4) (0.20000000,0.20000000)};
        \draw plot coordinates {(0.20000000,0.40000000) (0.4,0.4) (0.40000000,0.2) };
        \draw plot coordinates {(0.40000000,0.40000000) (0.6,0.4) (0.60000000,0.2) };
        \draw plot coordinates {(0.60000000,0.40000000) (0.8,0.4) (0.80000000,0.2) };
        \draw plot coordinates {(0.80000000,0.40000000) (1,0.4) (1,0.2) };

        \draw plot coordinates {(0.00000000,0.40000000) (0.0,0.6) (0.20000000,0.6) (0.20000000,0.40000000)};
        \draw plot coordinates {(0.20000000,0.60000000) (0.4,0.6) (0.40000000,0.4) };
        \draw plot coordinates {(0.40000000,0.60000000) (0.6,0.6) (0.60000000,0.4) };
        \draw plot coordinates {(0.60000000,0.60000000) (0.8,0.6) (0.80000000,0.4) };
        \draw plot coordinates {(0.80000000,0.60000000) (1,0.6) (1,0.4) };

        \draw plot coordinates {(0.00000000,0.60000000) (0.0,0.8) (0.20000000,0.8) (0.20000000,0.60000000)};
        \draw plot coordinates {(0.20000000,0.80000000) (0.4,0.8) (0.40000000,0.6) };
        \draw plot coordinates {(0.40000000,0.80000000) (0.6,0.8) (0.60000000,0.6) };
        \draw plot coordinates {(0.60000000,0.80000000) (0.8,0.8) (0.80000000,0.6) };
        \draw plot coordinates {(0.80000000,0.80000000) (1,0.8) (1,0.6) };

        \draw plot coordinates {(0.00000000,0.80000000) (0.0,1) (0.20000000,1) (0.20000000,0.80000000)};
        \draw plot coordinates {(0.20000000,1.0000000) (0.4,1) (0.40000000,0.8) };
        \draw plot coordinates {(0.40000000,1) (0.6,1) (0.60000000,0.8) };
        \draw plot coordinates {(0.60000000,1) (0.8,1) (0.80000000,0.8) };
        \draw plot coordinates {(0.80000000,1) (1,1) (1,0.8) };

        \draw plot coordinates{(0,0.1) (1,0.1)};
        \draw plot coordinates{(0,0.3) (1,0.3)};
        \draw plot coordinates{(0,0.5) (1,0.5)};
        \draw plot coordinates{(0,0.7) (1,0.7)};
        \draw plot coordinates{(0,0.9) (1,0.9)};

        \draw plot coordinates{(0.1,0) (0.1,1)};
        \draw plot coordinates{(0.3,0) (0.3,1)};
        \draw plot coordinates{(0.5,0) (0.5,1)};
        \draw plot coordinates{(0.7,0) (0.7,1)};
        \draw plot coordinates{(0.9,0) (0.9,1)};

        \draw[very thick] (0.5,0) -- (0.5,1);
        \draw plot[mark=*, mark size = 0.3pt] (0.5,0.);
        \draw plot[mark=*, mark size = 0.3pt] (0.5,0.1);
        \draw plot[mark=*, mark size = 0.3pt] (0.5,0.2);
        \draw plot[mark=*, mark size = 0.3pt] (0.5,0.3);
        \draw plot[mark=*, mark size = 0.3pt] (0.5,0.4);
        \draw plot[mark=*, mark size = 0.3pt] (0.5,0.5);
        \draw plot[mark=*, mark size = 0.3pt] (0.5,0.6);
        \draw plot[mark=*, mark size = 0.3pt] (0.5,0.7);
        \draw plot[mark=*, mark size = 0.3pt] (0.5,0.8);
        \draw plot[mark=*, mark size = 0.3pt] (0.5,0.9);
        \draw plot[mark=*, mark size = 0.3pt] (0.5,1);
      \end{tikzpicture}
      \subcaption{Cartesian\label{fig:meshfamilies:cartesian}}
    \end{center}
  \end{minipage}
  \hspace{0.5cm}
  \begin{minipage}[t]{0.25\textwidth}
    \begin{center}
      \begin{tikzpicture}[scale=3]
        \draw plot coordinates {(0,0) (1,0) (1,1) (0,1) (0,0)};
        \draw plot coordinates {(0,0.25) (0.5,0.25)};
        \draw plot coordinates {(0,0.5) (0.5,0.5)};
        \draw plot coordinates {(0,0.75) (0.5,0.75)};
        \draw plot coordinates {(0.25,0) (0.25,1)};
        
        \draw plot coordinates {(0.5,0.125) (1,0.125)};
        \draw plot coordinates {(0.5,0.375) (1,0.375)};
        \draw plot coordinates {(0.5,0.625) (1,0.625)};
        \draw plot coordinates {(0.5,0.875) (1,0.875)};
        \draw plot coordinates {(0.75,0) (0.75,1)};

%
%
		\draw[very thick] (0.5,0) -- (0.5,1);
        \draw plot[mark=*, mark size = 0.3pt] (0.5,0);
        \draw plot[mark=*, mark size = 0.3pt] (0.5,0.125);
        \draw plot[mark=*, mark size = 0.3pt] (0.5,0.25);
        \draw plot[mark=*, mark size = 0.3pt] (0.5,0.375);
        \draw plot[mark=*, mark size = 0.3pt] (0.5,0.5);
        \draw plot[mark=*, mark size = 0.3pt] (0.5,0.625);
        \draw plot[mark=*, mark size = 0.3pt] (0.5,0.75);
        \draw plot[mark=*, mark size = 0.3pt] (0.5,0.875);
        \draw plot[mark=*, mark size = 0.3pt] (0.5,1);
      \end{tikzpicture}
      \subcaption{\label{fig:meshfamilies:locref}Nonconforming}
    \end{center}
  \end{minipage}
  \caption{Mesh families for the numerical tests\label{fig:meshfamilies}}
\end{figure}
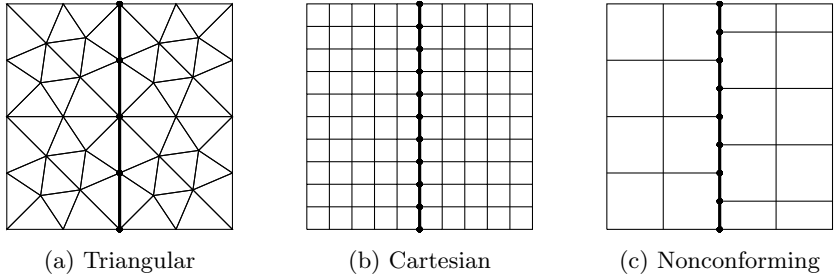
We consider the triangular, Cartesian, and nonconforming mesh families of Figure \ref{fig:meshfamilies} and monitor the following errors:
\begin{align}\label{eq:error.shortcuts}
\underline{\boldsymbol{e}}_{h} \coloneq \boldsymbol{\underline{u}}_h - \widehat{\underline{\boldsymbol{u}}}_h,\qquad \epsilon_{h} \coloneq p_h - \widehat{p}_h,\qquad \underline{\epsilon}_{h}^\Gamma \coloneq \underline{p}_h^\Gamma - \underline{\widehat{p}}_h^\Gamma,\qquad \epsilon_{h}^\Gamma \coloneq p_h^\Gamma - \widehat{p}_h^\Gamma,
\end{align}
where $\widehat{\underline{\boldsymbol{u}}}_h$, $\widehat{p}_h$, and $\underline{\widehat{p}}_h^\Gamma$ are the broken fracture pressures defined by \eqref{def:interpsol}, while $p_h^\Gamma$ and $\widehat{p}_h^\Gamma$ are defined as in Remark~\ref{rem:supercloseness}.
Notice that, while the triangular and Cartesian mesh families can be handled by standard finite element discretizations, this is not the case for the nonconforming mesh.
This kind of nonconforming meshes appear, e.g., when the fracture occurs between two plates, and the mesh of each bulk subdomain is designed to be compliant with the permeability values therein.

{We display in Figure \ref{fig:cv.2} and \ref{fig:cv.1} various error norms as a function of the meshsize, obtained with different values of the normal fracture permeability $\kappa_\Gamma^n \in \{2\ell_\Gamma, 1\}$ in order to show (i) the convergence rates, and (ii) the influence of the global anisotropy ratio $\varrho_\bulk$ on the value of the error, both predicted by Theorem \ref{thm:energyerror}. By choosing $\kappa_\Gamma^n = 2\ell_\Gamma$, we obtain an homogeneous bulk permeability tensor $\boldsymbol{K} = \boldsymbol{I}_2$ so the value of the error is not impacted by the global anisotropy ratio $\varrho_\bulk$ (since it is equal to $1$ in that case); see Figure \ref{fig:cv.2}.
On the other hand, letting $\kappa_\Gamma^n = 1$, we obtain a global anisotropy ratio $\varrho_\bulk = 50$ and we can clearly see the impact on the value of the error $\|\underline{\boldsymbol{e}}_h\|_{\boldsymbol{U},\xi,h}$ in Figure \ref{fig:cv.1}.}
For both configurations, the orders of convergence predicted by Theorem \ref{thm:energyerror} are confirmed numerically for $\|\underline{\boldsymbol{e}}_h\|_{\boldsymbol{U},\xi,h}$ and $\|\underline{\epsilon}_h^\Gamma\|_{\Gamma,h}$ (and even a slightly better convergence rate on Cartesian and nonconforming meshes). 
Moreover, convergence in $h^{k+2}$ is observed for the $L^2$-norms of the bulk and fracture pressures, corresponding to $\|\epsilon_h\|_{\bulk,h}$ and $\|\epsilon_h^\Gamma\|_\Gamma$, respectively; see Remark~\ref{rem:supercloseness} on this point.
\begin{figure}
  \begin{center}
    \begin{minipage}[b]{0.04\columnwidth}
      \begin{tikzpicture}
      	\draw node[scale=0.8,rotate=90]{\hspace*{0.8cm}Error $\|\epsilon_h\|_{B,h}$};
      \end{tikzpicture}
    \end{minipage}
    \hfill
    \begin{minipage}[b]{0.30\columnwidth}
      \begin{tikzpicture}[scale=0.48]
        \begin{loglogaxis}
          \addplot table[x=meshsize,y=err_ph_L2] {k0_mesh1_Kn0dot02.dat};
          \addplot table[x=meshsize,y=err_ph_L2] {k1_mesh1_Kn0dot02.dat};
          \addplot table[x=meshsize,y=err_ph_L2] {k2_mesh1_Kn0dot02.dat};
          \logLogSlopeTriangle{0.90}{0.4}{0.1}{2}{black};
          \logLogSlopeTriangle{0.90}{0.4}{0.1}{3}{black};
          \logLogSlopeTriangle{0.90}{0.4}{0.1}{4}{black};
        \end{loglogaxis}
      \end{tikzpicture}
    \end{minipage}
    \hfill
    \begin{minipage}[b]{0.30\columnwidth}
      \begin{tikzpicture}[scale=0.48]
        \begin{loglogaxis}[
            legend style = { 
              at={(0.5,1.2)},
              anchor = north,
              legend columns=-1
            }
          ]
          \addplot table[x=meshsize,y=err_ph_L2] {k0_mesh2_Kn0dot02.dat};
          \addplot table[x=meshsize,y=err_ph_L2] {k1_mesh2_Kn0dot02.dat};
          \addplot table[x=meshsize,y=err_ph_L2] {k2_mesh2_Kn0dot02.dat};
          \logLogSlopeTriangle{0.90}{0.4}{0.1}{2}{black};
          \logLogSlopeTriangle{0.90}{0.4}{0.1}{3}{black};
          \logLogSlopeTriangle{0.90}{0.4}{0.1}{4}{black};
          \legend{$k=0$ , $k=1$ , $k=2$};
        \end{loglogaxis}
      \end{tikzpicture}
    \end{minipage}
    \hfill
    \begin{minipage}[b]{0.31\columnwidth}
      \begin{tikzpicture}[scale=0.48]
        \begin{loglogaxis}
          \addplot table[x=meshsize,y=err_ph_L2]{k0_mesh5_Kn0dot02.dat};
          \addplot table[x=meshsize,y=err_ph_L2]{k1_mesh5_Kn0dot02.dat};
          \addplot table[x=meshsize,y=err_ph_L2]{k2_mesh5_Kn0dot02.dat};
          \logLogSlopeTriangle{0.90}{0.4}{0.1}{2}{black};
          \logLogSlopeTriangle{0.90}{0.4}{0.1}{3}{black};
          \logLogSlopeTriangle{0.90}{0.4}{0.1}{4}{black};
        \end{loglogaxis}
      \end{tikzpicture}
    \end{minipage}
  \\[0.5cm]
    \begin{minipage}[b]{0.04\columnwidth}
      \begin{tikzpicture}
      	\draw node[scale=0.8,rotate=90]{\hspace*{0.8cm}Error $\|\underline{\boldsymbol{e}}_h\|_{\boldsymbol{U},\xi,h}$};
      \end{tikzpicture}
    \end{minipage}
    \hfill
    \begin{minipage}[b]{0.3\columnwidth}
      \begin{tikzpicture}[scale=0.48]
        \begin{loglogaxis}
          \addplot table[x=meshsize,y=err_Uh_H] {k0_mesh1_Kn0dot02.dat};
          \addplot table[x=meshsize,y=err_Uh_H] {k1_mesh1_Kn0dot02.dat};
          \addplot table[x=meshsize,y=err_Uh_H] {k2_mesh1_Kn0dot02.dat};
          \logLogSlopeTriangle{0.90}{0.4}{0.1}{1}{black};
          \logLogSlopeTriangle{0.90}{0.4}{0.1}{2}{black};
          \logLogSlopeTriangle{0.90}{0.4}{0.1}{3}{black};
        \end{loglogaxis}
      \end{tikzpicture}
    \end{minipage}
    \hfill
    \begin{minipage}[b]{0.3\columnwidth}
      \begin{tikzpicture}[scale=0.48]
        \begin{loglogaxis}
          \addplot table[x=meshsize,y=err_Uh_H] {k0_mesh2_Kn0dot02.dat};
          \addplot table[x=meshsize,y=err_Uh_H] {k1_mesh2_Kn0dot02.dat};
          \addplot table[x=meshsize,y=err_Uh_H] {k2_mesh2_Kn0dot02.dat};
          \logLogSlopeTriangle{0.90}{0.4}{0.1}{1}{black};
          \logLogSlopeTriangle{0.90}{0.4}{0.1}{2}{black};
          \logLogSlopeTriangle{0.90}{0.4}{0.1}{3}{black};
        \end{loglogaxis}
      \end{tikzpicture}
    \end{minipage}
    \hfill
    \begin{minipage}[b]{0.31\columnwidth}
      \begin{tikzpicture}[scale=0.48]
        \begin{loglogaxis}
          \addplot table[x=meshsize,y=err_Uh_H] {k0_mesh5_Kn0dot02.dat};
          \addplot table[x=meshsize,y=err_Uh_H] {k1_mesh5_Kn0dot02.dat};
          \addplot table[x=meshsize,y=err_Uh_H] {k2_mesh5_Kn0dot02.dat};
          \logLogSlopeTriangle{0.90}{0.4}{0.1}{1}{black};
          \logLogSlopeTriangle{0.90}{0.4}{0.1}{2}{black};
          \logLogSlopeTriangle{0.90}{0.4}{0.1}{3}{black};
        \end{loglogaxis}
      \end{tikzpicture}
    \end{minipage}
  \\[0.5cm]
    \begin{minipage}[b]{0.04\columnwidth}
      \begin{tikzpicture}
      	\draw node[scale=0.8,rotate=90]{\hspace*{1cm}Error $\|\epsilon_h^\Gamma\|_{\Gamma}$};
      \end{tikzpicture}
    \end{minipage}
    \hfill
    \begin{minipage}[b]{0.3\columnwidth}
      \begin{tikzpicture}[scale=0.48]
        \begin{loglogaxis}
          \addplot table[x=meshsize,y=err_pFract_L2] {k0_mesh1_Kn0dot02.dat};
          \addplot table[x=meshsize,y=err_pFract_L2] {k1_mesh1_Kn0dot02.dat};
          \addplot table[x=meshsize,y=err_pFract_L2] {k2_mesh1_Kn0dot02.dat};
          \logLogSlopeTriangle{0.90}{0.4}{0.1}{2}{black};
          \logLogSlopeTriangle{0.90}{0.4}{0.1}{3}{black};
          \logLogSlopeTriangle{0.90}{0.4}{0.1}{4}{black};
        \end{loglogaxis}
      \end{tikzpicture}
    \end{minipage}
    \hfill
    \begin{minipage}[b]{0.3\columnwidth}
      \begin{tikzpicture}[scale=0.48]
        \begin{loglogaxis}
          \addplot table[x=meshsize,y=err_pFract_L2] {k0_mesh2_Kn0dot02.dat};
          \addplot table[x=meshsize,y=err_pFract_L2] {k1_mesh2_Kn0dot02.dat};
          \addplot table[x=meshsize,y=err_pFract_L2] {k2_mesh2_Kn0dot02.dat};
          \logLogSlopeTriangle{0.90}{0.4}{0.1}{2}{black};
          \logLogSlopeTriangle{0.90}{0.4}{0.1}{3}{black};
          \logLogSlopeTriangle{0.90}{0.4}{0.1}{4}{black};
        \end{loglogaxis}
      \end{tikzpicture}
    \end{minipage}
    \hfill
    \begin{minipage}[b]{0.31\columnwidth}
      \begin{tikzpicture}[scale=0.48]
        \begin{loglogaxis}
          \addplot table[x=meshsize,y=err_pFract_L2] {k0_mesh5_Kn0dot02.dat};
          \addplot table[x=meshsize,y=err_pFract_L2] {k1_mesh5_Kn0dot02.dat};
          \addplot table[x=meshsize,y=err_pFract_L2] {k2_mesh5_Kn0dot02.dat};
          \logLogSlopeTriangle{0.90}{0.4}{0.1}{2}{black};
          \logLogSlopeTriangle{0.90}{0.4}{0.1}{3}{black};
          \logLogSlopeTriangle{0.90}{0.4}{0.1}{4}{black};
        \end{loglogaxis}
      \end{tikzpicture}
    \end{minipage}
  \\[0.5cm]
    \begin{minipage}[b]{0.04\columnwidth}
      \begin{tikzpicture}
      	\draw node[scale=0.8,rotate=90]{\hspace*{2cm}Error $\|\underline{\epsilon}_h^\Gamma\|_{\Gamma,h}$};
      \end{tikzpicture}
    \end{minipage}
    \hfill
    \begin{minipage}[b]{0.3\columnwidth}
      \begin{tikzpicture}[scale=0.48]
        \begin{loglogaxis}
          \addplot table[x=meshsize,y=err_pFract_en] {k0_mesh1_Kn0dot02.dat};
          \addplot table[x=meshsize,y=err_pFract_en] {k1_mesh1_Kn0dot02.dat};
          \addplot table[x=meshsize,y=err_pFract_en] {k2_mesh1_Kn0dot02.dat};
          \logLogSlopeTriangle{0.90}{0.4}{0.1}{1}{black};
          \logLogSlopeTriangle{0.90}{0.4}{0.1}{2}{black};
          \logLogSlopeTriangle{0.90}{0.4}{0.1}{3}{black};
        \end{loglogaxis}
      \end{tikzpicture}
      \subcaption*{Triangular}
    \end{minipage}
    \hfill
    \begin{minipage}[b]{0.3\columnwidth}
      \begin{tikzpicture}[scale=0.48]
        \begin{loglogaxis}
          \addplot table[x=meshsize,y=err_pFract_en] {k0_mesh2_Kn0dot02.dat};
          \addplot table[x=meshsize,y=err_pFract_en] {k1_mesh2_Kn0dot02.dat};
          \addplot table[x=meshsize,y=err_pFract_en] {k2_mesh2_Kn0dot02.dat};
          \logLogSlopeTriangle{0.90}{0.4}{0.1}{1}{black};
          \logLogSlopeTriangle{0.90}{0.4}{0.1}{2}{black};
          \logLogSlopeTriangle{0.90}{0.4}{0.1}{3}{black};
        \end{loglogaxis}
      \end{tikzpicture}
      \subcaption*{Cartesian}
    \end{minipage}
    \hfill
    \begin{minipage}[b]{0.31\columnwidth}
      \begin{tikzpicture}[scale=0.48]
        \begin{loglogaxis}
          \addplot table[x=meshsize,y=err_pFract_en] {k0_mesh5_Kn0dot02.dat};
          \addplot table[x=meshsize,y=err_pFract_en] {k1_mesh5_Kn0dot02.dat};
          \addplot table[x=meshsize,y=err_pFract_en] {k2_mesh5_Kn0dot02.dat};
          \logLogSlopeTriangle{0.90}{0.4}{0.1}{1}{black};
          \logLogSlopeTriangle{0.90}{0.4}{0.1}{2}{black};
          \logLogSlopeTriangle{0.90}{0.4}{0.1}{3}{black};
        \end{loglogaxis}
      \end{tikzpicture}
      \subcaption*{Nonconforming}
    \end{minipage}
  \end{center}
  \caption{{Errors vs. $h$ for the test case of Section~\ref{sec:num.tests:convergence} on the mesh families introduced in Figure \ref{fig:meshfamilies} with $\kappa_\Gamma^n = 2\ell_\Gamma$}\label{fig:cv.2}}
  
\end{figure}
\begin{figure}
  \begin{center}
    \begin{minipage}[b]{0.04\columnwidth}
      \begin{tikzpicture}
      	\draw node[scale=0.8,rotate=90]{\hspace*{0.8cm}Error $\|\epsilon_h\|_{B,h}$};
      \end{tikzpicture}
    \end{minipage}
    \hfill
    \begin{minipage}[b]{0.30\columnwidth}
      \begin{tikzpicture}[scale=0.48]
        \begin{loglogaxis}
          \addplot table[x=meshsize,y=err_ph_L2] {k0_mesh1_Kn1.dat};
          \addplot table[x=meshsize,y=err_ph_L2] {k1_mesh1_Kn1.dat};
          \addplot table[x=meshsize,y=err_ph_L2] {k2_mesh1_Kn1.dat};
          \logLogSlopeTriangle{0.90}{0.4}{0.1}{2}{black};
          \logLogSlopeTriangle{0.90}{0.4}{0.1}{3}{black};
          \logLogSlopeTriangle{0.90}{0.4}{0.1}{4}{black};
        \end{loglogaxis}
      \end{tikzpicture}
    \end{minipage}
    \hfill
    \begin{minipage}[b]{0.30\columnwidth}
      \begin{tikzpicture}[scale=0.48]
        \begin{loglogaxis}[
            legend style = { 
              at={(0.5,1.2)},
              anchor = north,
              legend columns=-1
            }
          ]
          \addplot table[x=meshsize,y=err_ph_L2] {k0_mesh2_Kn1.dat};
          \addplot table[x=meshsize,y=err_ph_L2] {k1_mesh2_Kn1.dat};
          \addplot table[x=meshsize,y=err_ph_L2] {k2_mesh2_Kn1.dat};
          \logLogSlopeTriangle{0.90}{0.4}{0.1}{2}{black};
          \logLogSlopeTriangle{0.90}{0.4}{0.1}{3}{black};
          \logLogSlopeTriangle{0.90}{0.4}{0.1}{4}{black};
          \legend{$k=0$ , $k=1$ , $k=2$};
        \end{loglogaxis}
      \end{tikzpicture}
    \end{minipage}
    \hfill
    \begin{minipage}[b]{0.31\columnwidth}
      \begin{tikzpicture}[scale=0.48]
        \begin{loglogaxis}
          \addplot table[x=meshsize,y=err_ph_L2]{k0_mesh5_Kn1.dat};
          \addplot table[x=meshsize,y=err_ph_L2]{k1_mesh5_Kn1.dat};
          \addplot table[x=meshsize,y=err_ph_L2]{k2_mesh5_Kn1.dat};
          \logLogSlopeTriangle{0.90}{0.4}{0.1}{2}{black};
          \logLogSlopeTriangle{0.90}{0.4}{0.1}{3}{black};
          \logLogSlopeTriangle{0.90}{0.4}{0.1}{4}{black};
        \end{loglogaxis}
      \end{tikzpicture}
    \end{minipage}
  \\[0.5cm]
    \begin{minipage}[b]{0.04\columnwidth}
      \begin{tikzpicture}
      	\draw node[scale=0.8,rotate=90]{\hspace*{0.8cm}Error $\|\underline{\boldsymbol{e}}_h\|_{\boldsymbol{U},\xi,h}$};
      \end{tikzpicture}
    \end{minipage}
    \hfill
    \begin{minipage}[b]{0.3\columnwidth}
      \begin{tikzpicture}[scale=0.48]
        \begin{loglogaxis}
          \addplot table[x=meshsize,y=err_Uh_H] {k0_mesh1_Kn1.dat};
          \addplot table[x=meshsize,y=err_Uh_H] {k1_mesh1_Kn1.dat};
          \addplot table[x=meshsize,y=err_Uh_H] {k2_mesh1_Kn1.dat};
          \logLogSlopeTriangle{0.90}{0.4}{0.1}{1}{black};
          \logLogSlopeTriangle{0.90}{0.4}{0.1}{2}{black};
          \logLogSlopeTriangle{0.90}{0.4}{0.1}{3}{black};
        \end{loglogaxis}
      \end{tikzpicture}
    \end{minipage}
    \hfill
    \begin{minipage}[b]{0.3\columnwidth}
      \begin{tikzpicture}[scale=0.48]
        \begin{loglogaxis}
          \addplot table[x=meshsize,y=err_Uh_H] {k0_mesh2_Kn1.dat};
          \addplot table[x=meshsize,y=err_Uh_H] {k1_mesh2_Kn1.dat};
          \addplot table[x=meshsize,y=err_Uh_H] {k2_mesh2_Kn1.dat};
          \logLogSlopeTriangle{0.90}{0.4}{0.1}{1}{black};
          \logLogSlopeTriangle{0.90}{0.4}{0.1}{2}{black};
          \logLogSlopeTriangle{0.90}{0.4}{0.1}{3}{black};
        \end{loglogaxis}
      \end{tikzpicture}
    \end{minipage}
    \hfill
    \begin{minipage}[b]{0.31\columnwidth}
      \begin{tikzpicture}[scale=0.48]
        \begin{loglogaxis}
          \addplot table[x=meshsize,y=err_Uh_H] {k0_mesh5_Kn1.dat};
          \addplot table[x=meshsize,y=err_Uh_H] {k1_mesh5_Kn1.dat};
          \addplot table[x=meshsize,y=err_Uh_H] {k2_mesh5_Kn1.dat};
          \logLogSlopeTriangle{0.90}{0.4}{0.1}{1}{black};
          \logLogSlopeTriangle{0.90}{0.4}{0.1}{2}{black};
          \logLogSlopeTriangle{0.90}{0.4}{0.1}{3}{black};
        \end{loglogaxis}
      \end{tikzpicture}
    \end{minipage}
  \\[0.5cm]
    \begin{minipage}[b]{0.04\columnwidth}
      \begin{tikzpicture}
      	\draw node[scale=0.8,rotate=90]{\hspace*{1cm}Error $\|\epsilon_h^\Gamma\|_{\Gamma}$};
      \end{tikzpicture}
    \end{minipage}
    \hfill
    \begin{minipage}[b]{0.3\columnwidth}
      \begin{tikzpicture}[scale=0.48]
        \begin{loglogaxis}
          \addplot table[x=meshsize,y=err_pFract_L2] {k0_mesh1_Kn1.dat};
          \addplot table[x=meshsize,y=err_pFract_L2] {k1_mesh1_Kn1.dat};
          \addplot table[x=meshsize,y=err_pFract_L2] {k2_mesh1_Kn1.dat};
          \logLogSlopeTriangle{0.90}{0.4}{0.1}{2}{black};
          \logLogSlopeTriangle{0.90}{0.4}{0.1}{3}{black};
          \logLogSlopeTriangle{0.90}{0.4}{0.1}{4}{black};
        \end{loglogaxis}
      \end{tikzpicture}
    \end{minipage}
    \hfill
    \begin{minipage}[b]{0.3\columnwidth}
      \begin{tikzpicture}[scale=0.48]
        \begin{loglogaxis}
          \addplot table[x=meshsize,y=err_pFract_L2] {k0_mesh2_Kn1.dat};
          \addplot table[x=meshsize,y=err_pFract_L2] {k1_mesh2_Kn1.dat};
          \addplot table[x=meshsize,y=err_pFract_L2] {k2_mesh2_Kn1.dat};
          \logLogSlopeTriangle{0.90}{0.4}{0.1}{2}{black};
          \logLogSlopeTriangle{0.90}{0.4}{0.1}{3}{black};
          \logLogSlopeTriangle{0.90}{0.4}{0.1}{4}{black};
        \end{loglogaxis}
      \end{tikzpicture}
    \end{minipage}
    \hfill
    \begin{minipage}[b]{0.31\columnwidth}
      \begin{tikzpicture}[scale=0.48]
        \begin{loglogaxis}
          \addplot table[x=meshsize,y=err_pFract_L2] {k0_mesh5_Kn1.dat};
          \addplot table[x=meshsize,y=err_pFract_L2] {k1_mesh5_Kn1.dat};
          \addplot table[x=meshsize,y=err_pFract_L2] {k2_mesh5_Kn1.dat};
          \logLogSlopeTriangle{0.90}{0.4}{0.1}{2}{black};
          \logLogSlopeTriangle{0.90}{0.4}{0.1}{3}{black};
          \logLogSlopeTriangle{0.90}{0.4}{0.1}{4}{black};
        \end{loglogaxis}
      \end{tikzpicture}
    \end{minipage}
  \\[0.5cm]
    \begin{minipage}[b]{0.04\columnwidth}
      \begin{tikzpicture}
      	\draw node[scale=0.8,rotate=90]{\hspace*{2cm}Error $\|\underline{\epsilon}_h^\Gamma\|_{\Gamma,h}$};
      \end{tikzpicture}
    \end{minipage}
    \hfill
    \begin{minipage}[b]{0.3\columnwidth}
      \begin{tikzpicture}[scale=0.48]
        \begin{loglogaxis}
          \addplot table[x=meshsize,y=err_pFract_en] {k0_mesh1_Kn1.dat};
          \addplot table[x=meshsize,y=err_pFract_en] {k1_mesh1_Kn1.dat};
          \addplot table[x=meshsize,y=err_pFract_en] {k2_mesh1_Kn1.dat};
          \logLogSlopeTriangle{0.90}{0.4}{0.1}{1}{black};
          \logLogSlopeTriangle{0.90}{0.4}{0.1}{2}{black};
          \logLogSlopeTriangle{0.90}{0.4}{0.1}{3}{black};
        \end{loglogaxis}
      \end{tikzpicture}
      \subcaption*{Triangular}
    \end{minipage}
    \hfill
    \begin{minipage}[b]{0.3\columnwidth}
      \begin{tikzpicture}[scale=0.48]
        \begin{loglogaxis}
          \addplot table[x=meshsize,y=err_pFract_en] {k0_mesh2_Kn1.dat};
          \addplot table[x=meshsize,y=err_pFract_en] {k1_mesh2_Kn1.dat};
          \addplot table[x=meshsize,y=err_pFract_en] {k2_mesh2_Kn1.dat};
          \logLogSlopeTriangle{0.90}{0.4}{0.1}{1}{black};
          \logLogSlopeTriangle{0.90}{0.4}{0.1}{2}{black};
          \logLogSlopeTriangle{0.90}{0.4}{0.1}{3}{black};
        \end{loglogaxis}
      \end{tikzpicture}
      \subcaption*{Cartesian}
    \end{minipage}
    \hfill
    \begin{minipage}[b]{0.31\columnwidth}
      \begin{tikzpicture}[scale=0.48]
        \begin{loglogaxis}
          \addplot table[x=meshsize,y=err_pFract_en] {k0_mesh5_Kn1.dat};
          \addplot table[x=meshsize,y=err_pFract_en] {k1_mesh5_Kn1.dat};
          \addplot table[x=meshsize,y=err_pFract_en] {k2_mesh5_Kn1.dat};
          \logLogSlopeTriangle{0.90}{0.4}{0.1}{1}{black};
          \logLogSlopeTriangle{0.90}{0.4}{0.1}{2}{black};
          \logLogSlopeTriangle{0.90}{0.4}{0.1}{3}{black};
        \end{loglogaxis}
      \end{tikzpicture}
      \subcaption*{Nonconforming}
    \end{minipage}
  \end{center}
  \caption{{Errors vs. $h$ for the test case of Section~\ref{sec:num.tests:convergence} on the mesh families introduced in Figure \ref{fig:meshfamilies} with $\kappa_\Gamma^n = 1$}\label{fig:cv.1}}
\end{figure}
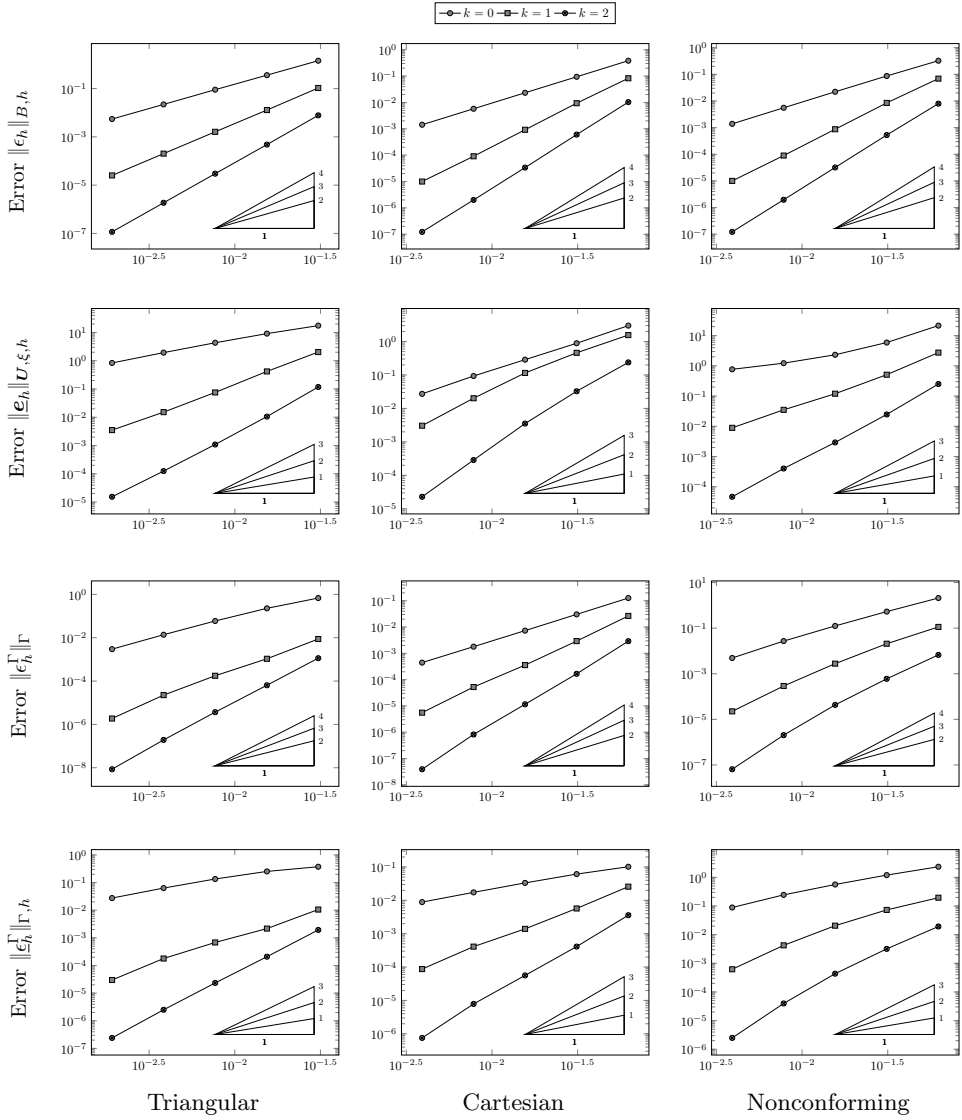

\subsection{Quarter five-spot problem\label{sec:fivespot}}

The five-spot pattern is a standard configuration in petroleum engineering used to displace and extract the oil in the basement by injecting water, steam, or gas; see, e.g., \cite{Coats.George.Chu.Marcum:74, Todd.ODell.Hirasaki:72}. The injection well sits in the center of a square, and four production wells are located at the corners. Due to the symmetry of the problem, we consider here only a quarter five-spot pattern on $\Omega= (0,1)^2$ with injection and production wells located in $(0,0)$ and $(1,1)$, respectively, and modelled by the source term $f :\Omega_\bulk\rightarrow\mathbb{R}$ such that
\begin{equation*}
\begin{aligned}
f(\boldsymbol{x}) = 200\Big(
	\tanh \big(
		200 &(0.025 - (x_1^2+x_2^2)^{\nicefrac12})
	\big)
	\\
	&- \tanh \big(
		200 (0.025 - ((x_1-1)^2+(x_2-1)^2)^{\nicefrac12})
	\big)
\Big).
\end{aligned}
\end{equation*}

\paragraph{Test 1: No fracture} In Figure~\ref{fig:5spot:nofract}, we display the pressure distribution when the domain $\Omega$ contains no fracture, i.e. $\Omega_\bulk= \Omega$; see Figure~\ref{fig:5spot:domainConfig}. The bulk tensor is given by $\boldsymbol{K} = \boldsymbol{I}_2$, and we enforce homogeneous Neumann and Dirichlet boundary conditions, respectively, on (see Figure~\ref{fig:5spot:domainConfig})
\begin{align*}
  \partial\Omega_{\rm B}^{\rm N} &= \{\boldsymbol{x}\in\partial\Omega_\bulk ~|~ x_1 = 0 \text{ or } x_2 = 0 \text{ or } \left(x_1 > \nicefrac{3}{4} \text{ and } x_2 > \nicefrac{3}{4}\right)\},
  \\
  \partial\Omega_{\rm B}^{\rm D} &= \{\boldsymbol{x}\in\partial\Omega_\bulk ~|~ \left(x_1 = 1 \text{ and } x_2 \leq \nicefrac{3}{4}\right) \text{ or } \left(x_2 = 1 \text{ and } x_1 \leq \nicefrac{3}{4}\right)\}.
\end{align*}
Since the bulk permeability is the identity matrix and there is no fracture inside the domain, the pressure decreases continuously moving from the injection well towards the production well.

\paragraph{Test 2: Permeable fracture} We now let the domain $\Omega$ be crossed by the fracture $\Gamma = \{\boldsymbol{x}\in\Omega ~|~ x_2 = 1-x_1\}$ of constant thickness $\ell_\Gamma = 10^{-2}$, and we let $f_\Gamma\equiv 0$.
In addition to the bulk boundary conditions described in \textbf{Test 1}, we enforce homogeneous Dirichlet boundary conditions on $\partial\Gamma^{\rm D} =\partial\Gamma$; see Figure~\ref{fig:5spot:domainConfig}.
The bulk and fracture permeability parameters are such that%
\begin{align*}
\boldsymbol{K} = \boldsymbol{I}_2\qquad\kappa_\Gamma^n = 1,\qquad
\kappa_\Gamma^\tau = 100,
\end{align*}
and are chosen in such a way that the fracture is permeable, which means that the fluid should be attracted by it.
The bulk pressure corresponding to this configuration is depicted in Figure~\ref{fig:bulkpressure:perm}. As shown in Figure~\ref{fig:pressurecut}, we remark that (i) in $\Omega_{\bulk,1}$, we have a lower pressure, and that the pressure decreases more slowly than in \textbf{Test 1} going from the injection well towards the fracture and (ii) in $\Omega_{\bulk,2}$, the flow caused by the production well attracts, less significantly than in \textbf{Test 1}, the flow outside the fracture.

\paragraph{Test 3: Impermeable fracture} We next consider the case of an impermeable fracture: we keep the same domain configuration as before, but we let
\begin{align*}
\kappa_\Gamma^n = 10^{-2},\qquad\kappa_\Gamma^\tau = 1.
\end{align*}
Unlike before, we observe in this case a significant jump of the bulk pressure across the fracture $\Gamma$, see Figure~\ref{fig:bulkpressure:imperm}.
This can be better appreciated in Figure~\ref{fig:pressurecut}, which contains the plots of the bulk pressure over the line $x_1=x_2$ for the various configurations considered.
\begin{figure}
\begin{center}
\begin{minipage}[b]{0.32\columnwidth}\center
\includegraphics[width=4cm]{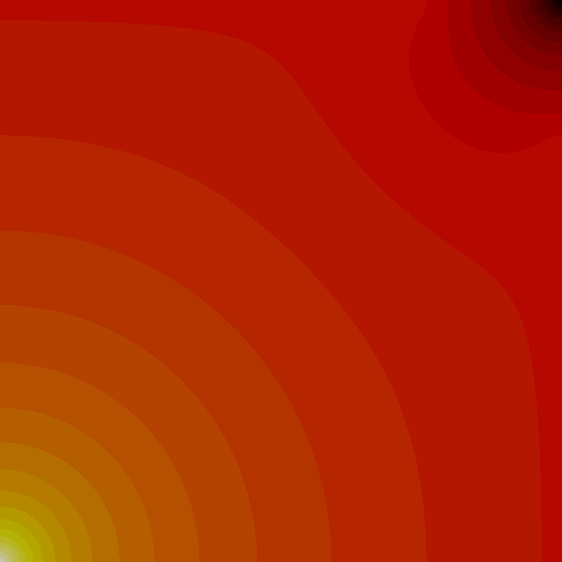} 
\vfill
\begin{minipage}[b]{1\columnwidth}\center\tiny
$-4.70\cdot10^{-1}$\hfill\includegraphics[width=0.3\columnwidth]{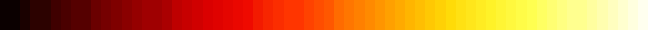} \hfill$6.17\cdot10^{-1}$
\end{minipage}
\subcaption{No fracture\label{fig:5spot:nofract}}
\end{minipage}
\hfill
\begin{minipage}[b]{0.32\columnwidth}\center
\includegraphics[width=4cm]{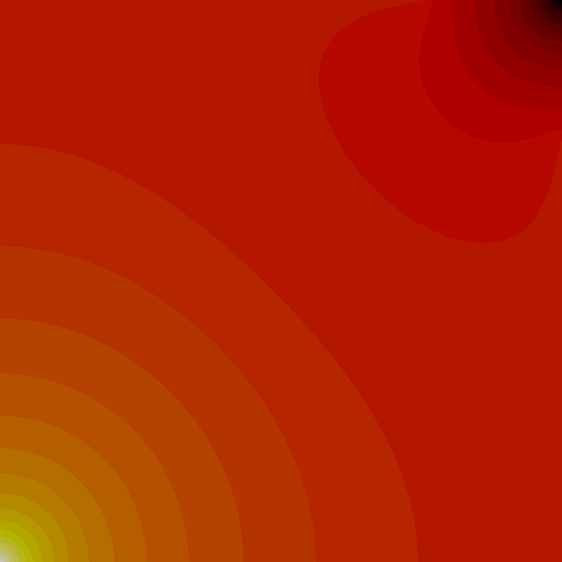} 
\vfill
\begin{minipage}[b]{1\columnwidth}\center\tiny
$-4.73\cdot10^{-1}$\hfill\includegraphics[width=0.3\columnwidth]{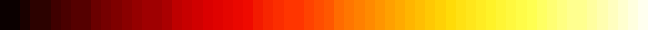} \hfill$5.96\cdot10^{-1}$
\end{minipage}
\subcaption{Permeable fracture\label{fig:bulkpressure:perm}}
\end{minipage}
\hfill
\begin{minipage}[b]{0.32\columnwidth}\center
\includegraphics[width=4cm]{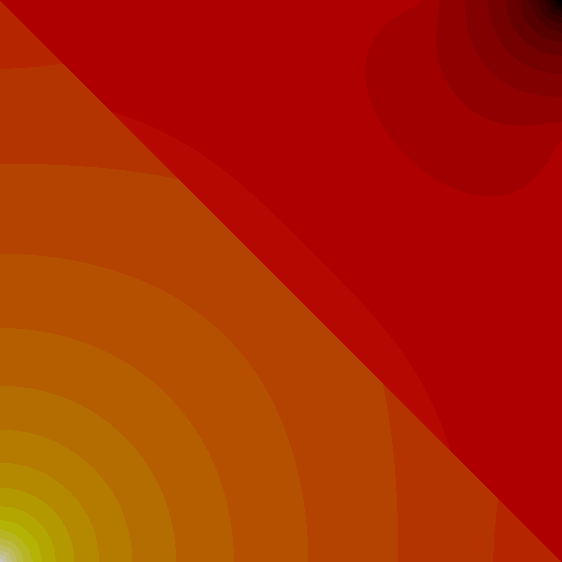} 
\vfill
\begin{minipage}[b]{1\columnwidth}\center\tiny
$-4.71\cdot10^{-1}$\hfill\includegraphics[width=0.3\columnwidth]{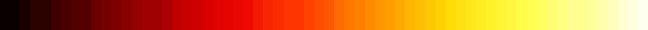} \hfill$7.54\cdot10^{-1}$
\end{minipage}
\subcaption{Impermeable fracture\label{fig:bulkpressure:imperm}}
\end{minipage}
\end{center}
\caption{\label{pic:5spotpatterns}Bulk pressure for the test cases of Section~\ref{sec:fivespot} on a triangular mesh ($h=7.68\cdot 10^{-3}$) with $k=2$}
\end{figure}

\paragraph{Flow across the fracture}
Since an exact solution is not available for the previous test cases, we provide a quantitative assessment of the convergence by monitoring the quantity
$$ 
M^{k,h}_{\nicefrac{\mathrm{p}}{\mathrm{i}}}\coloneq\sum_{F\in\f_h^\Gamma} \int_F [\![\underline{\boldsymbol{u}}_h]\!]_F,
$$
which corresponds to the global flux entering the fracture for the permeable (subscript p) and impermeable (subscript i) fractured test cases.
The index $k$ refers to the polynomial degree $k\in\{0,1,2\}$, and the index $h$ to the meshsize.
Five refinement levels of the triangular mesh depicted in Figure~\ref{fig:meshfamilies:triangular} are considered.
We plot in Figure~\ref{fig:plot.meas.err.perm} and~\ref{fig:plot.meas.err.imperm} the errors $\epsilon_{\nicefrac{\mathrm{p}}{\mathrm{i}}}\coloneq |{\rm M}^{\mathrm{r}}_{\nicefrac{\mathrm{p}}{\mathrm{i}}} - {\rm M}_{\nicefrac{\mathrm{p}}{\mathrm{i}}}^{k,h}|$ for the permeable/impermeable case ($\nicefrac{\mathrm{p}}{\mathrm{i}}$), where ${\rm M}^{r}_{\nicefrac{\mathrm{p}}{\mathrm{i}}}$ denotes the reference value obtained with $k=2$ on the fifth mesh refinement corresponding to $h=9.60\cdot 10^{-4}$.
In both cases we have convergence, with respect to the polynomial degree and the meshsize, to the reference values ${\rm M}^{r}_{\mathrm{p}} = 9.96242\cdot 10^{-2}$ and ${\rm M}^{r}_{\mathrm{i}} = 3.19922\cdot 10^{-2}$.
For the permeable test case depicted in Figure~\ref{fig:plot.meas.err.perm}, after the second refinement, increasing the polynomial degree only modestly affect the error decay, which suggests that convergence may be limited by the local regularity of the exact solution.
For the impermeable test case depicted in Figure~\ref{fig:plot.meas.err.imperm}, on the other hand, the local regularity of the exact solution seems sufficient to benefit from the increase of the approximation order.

\begin{figure}
\begin{minipage}[b]{0.68\columnwidth}
\begin{minipage}[b]{0.44\columnwidth}\center
\begin{tikzpicture}[scale=1.3]\scriptsize
\draw[fill=black!10!white] (0.4,0) arc (0:90:0.4) -- (0,0) -- (0.4,0);
\draw (0.15,0.15) node {$\boldsymbol{\oplus}$};
\draw[fill=black!10!white] (1.6,2) arc (180:270:0.4) -- (2,2) -- (1.6,2);
\draw (1.85,1.85) node {$\boldsymbol{\ominus}$};
\draw (0,0) -- (2,0) -- (2,2) -- (0,2) -- (0,0);
\draw (1,1) node{$\Omega$};
\draw (-0.15, 1) node[rotate=90]{$\boldsymbol{u}\cdot\normal_{\partial\Omega} = 0$};
\draw (1, -0.15) node {$\boldsymbol{u}\cdot\normal_{\partial\Omega} = 0$};
\draw (1.5,2) node {$\mid$};
\draw (2,1.5) node[rotate=90] {$\mid$};
\draw (2.15, 2.15) node[rotate=-45]{$\boldsymbol{u}\cdot\normal_{\partial\Omega} = 0$};
\draw (0.75,2.15) node {$p = 0$};
\draw (2.15, 0.75) node[rotate=90]{$p = 0$};
\end{tikzpicture}
\end{minipage}
\hfill
\begin{minipage}[b]{0.55\columnwidth}\center
\begin{tikzpicture}[scale=1.3]\scriptsize
\draw[fill=black!10!white] (0.4,0) arc (0:90:0.4) -- (0,0) -- (0.4,0);
\draw (0.15,0.15) node {$\boldsymbol{\oplus}$};
\draw[fill=black!10!white] (2.1,2) arc (180:270:0.4) -- (2.5,2) -- (2.1,2);
\draw (2.35,1.85) node {$\boldsymbol{\ominus}$};
\draw (0,0) -- (2,0) -- (0,2) -- (0,0);
\draw (0.5,0.5) node{$\Omega_{\bulk,1}$};
\draw (2.25,0) -- (0.25,2);
\draw[color=white, fill=white] (1.25,1.) circle (0.1);
\draw (1.25,1.) node{$\Gamma$};
\draw (2.5,0) -- (2.5,2) -- (0.5,2) -- (2.5,0);
\draw (2.0,1.5) node{$\Omega_{\bulk,2}$};
\draw (-0.15, 1) node[rotate=90]{$\boldsymbol{u}\cdot\normal_{\partial\Omega} = 0$};
\draw (1, -0.15) node {$\boldsymbol{u}\cdot\normal_{\partial\Omega} = 0$};
\draw (2,2) node {$\mid$};
\draw (2.5,1.5) node[rotate=90] {$\mid$};
\draw (2.65, 2.15) node[rotate=-45]{$\boldsymbol{u}\cdot\normal_{\partial\Omega} = 0$};
\draw (0.25,2.15) node {$p_\Gamma = 0$};
\draw (2.25,-0.15) node {$p_\Gamma = 0$};
\draw (1.25,2.15) node {$p = 0$};
\draw (2.65, 0.75) node[rotate=90]{$p = 0$};
\end{tikzpicture}
\end{minipage}
\subcaption{Domain configurations without (left) and with (right) fracture\label{fig:5spot:domainConfig}}
\end{minipage}
\hfill
  \begin{minipage}[b]{0.314\columnwidth}\center
    \includegraphics[width=1\columnwidth]{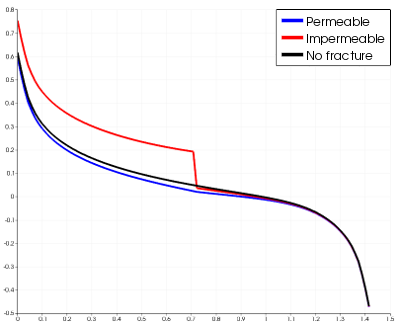} 
    \subcaption{Bulk pressure over $x_1 = x_2$\label{fig:pressurecut}}
  \end{minipage}
\vfill
  \begin{minipage}[b]{0.45\columnwidth}\center
    \begin{tikzpicture}[scale=0.7]
      \begin{axis}[
          legend style = {
            at={(0.05,0.95)}, 
            anchor = north west
          }
        ]
        \addplot table[x=meshsize,y=perm_k0] {conv_jump.dat};
        \addplot table[x=meshsize,y=perm_k1] {conv_jump.dat};
        \addplot table[x=meshsize,y=perm_k2] {conv_jump.dat};
        \legend{$k=0$ , $k=1$ , $k=2$};
      \end{axis}
    \end{tikzpicture}
    \subcaption{$\epsilon_{\rm p}$ vs. $h$\label{fig:plot.meas.err.perm}}
  \end{minipage}
  \hfill
  \begin{minipage}[b]{0.45\columnwidth}\center
    \begin{tikzpicture}[scale=0.7]
      \begin{axis}[
          legend style = {
            at={(0.05,0.95)}, 
            anchor = north west
          }
        ]
        \addplot table[x=meshsize,y=imperm_k0] {conv_jump.dat};
        \addplot table[x=meshsize,y=imperm_k1] {conv_jump.dat};
        \addplot table[x=meshsize,y=imperm_k2] {conv_jump.dat};
        \legend{$k=0$ , $k=1$ , $k=2$};
      \end{axis}
    \end{tikzpicture}
    \subcaption{$\epsilon_{\rm i}$ vs. $h$\label{fig:plot.meas.err.imperm}}
  \end{minipage}
  \caption{Domain configurations, pressure along the line $x_{1}=x_{2}$, and errors on the flow across the fracture vs. $h$ for the test cases of Section~\ref{sec:fivespot}.}
\end{figure}
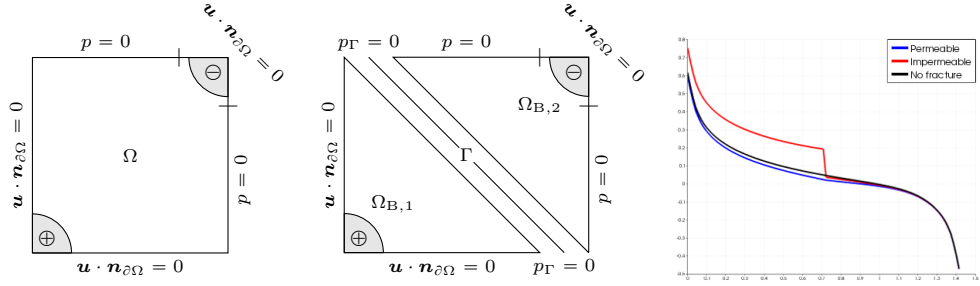
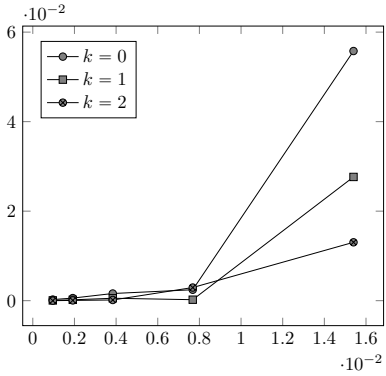
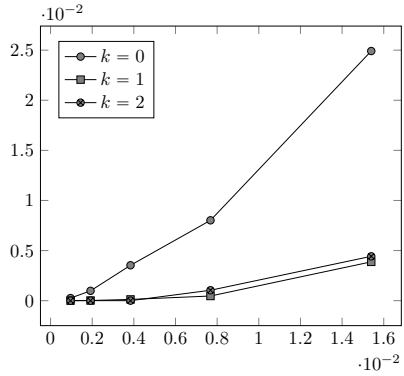

\subsection{Porous medium with random permeability}\label{sec:num.tests:random}

To show the influence of the bulk permeability tensor on the solution, we consider two piecewise constants functions $\mu_1,\mu_2:\Omega_{\bulk}\to(0,2)$ and the heterogeneous and possibly anisotropic bulk tensor $\boldsymbol{K}$ given by
\begin{align*}
  \boldsymbol{K} \coloneq \left[
  \begin{tabular}{cc}
    $\mu_1$ & $0$ \\
    $0$ & $\mu_2$
  \end{tabular}\right].
\end{align*}
For the following tests, we use a $64\times 64$ uniform Cartesian mesh ($h=3.91\cdot 10^{-3}$) and $k=2$.
The domain $\Omega \coloneq (0,1)^2$ is crossed by a fracture $\Gamma \coloneq \{0.5\}\times(0,1)$ of constant thickness $\ell_\Gamma\coloneq 10^{-2}$. We set the fracture permeability parameters $\kappa_\Gamma^n \coloneq 1$ and $\kappa_\Gamma^\tau \coloneq 100$, corresponding to a permeable fracture.
The source terms are constant and such that $f\equiv 4$ and $f_\Gamma\equiv 4$.
We enforce homogeneous Neumann boundary conditions on $\partial\Omega_{\bulk}^{\rm N} \coloneq \{\boldsymbol{x}\in\partial\Omega_\bulk ~|~ x_1 \in\{0,1\}\}$ and Dirichlet boundary conditions on $\partial\Omega_{\bulk}^{\rm D} \coloneq \{\boldsymbol{x}\in\partial\Omega_\bulk ~|~ x_2 \in\{0,1\}\}$ and $\partial\Gamma^{\rm D} \coloneq\partial\Gamma$ with
\begin{align*}
g_B(\boldsymbol{x}) \coloneq x_2\quad\forall\boldsymbol{x}\in\partial\Omega_{\bulk}^{\rm D},\qquad\qquad g_\Gamma(\boldsymbol{x}) \coloneq x_2\quad\forall\boldsymbol{x}\in\partial\Gamma^{\rm D}.
\end{align*}
\paragraph{Test 1: Homogeneous permeability} In Figure~\ref{pic:norandom}, we depict the bulk pressure distribution corresponding to $\mu_1=\mu_2\coloneq 1$. As expected, the flow is moving towards the fracture but less and less significantly as we approach the bottom of the domain since the pressure decreases with respect to the boundary conditions.

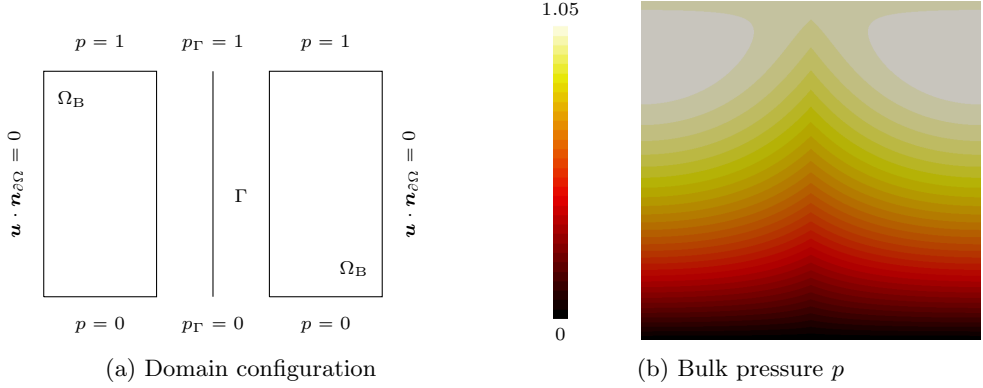
\begin{figure}
\begin{center}
\begin{minipage}[b]{0.49\columnwidth}
\begin{tikzpicture}[scale=1.5]\scriptsize
\draw (0,0) -- (1,0) -- (1,2) -- (0,2) -- (0,0);
\draw (0.25,1.75) node{$\Omega_\bulk$};
\draw (1.5,0) -- (1.5,2);
\draw (1.75,0.9) node{$\Gamma$};
\draw (2,0) -- (3,0) -- (3,2) -- (2,2) -- (2,0);
\draw (2.75,0.25) node{$\Omega_\bulk$};
\draw (-0.25, 1) node[rotate=90]{$\boldsymbol{u}\cdot\normal_{\partial\Omega} = 0$};
\draw (3.25, 1) node[rotate=90]{$\boldsymbol{u}\cdot\normal_{\partial\Omega} = 0$};
\draw (0.5,-0.25) node {$p=0$};
\draw (0.5,2.25) node {$p=1$};
\draw (2.5,-0.25) node {$p=0$};
\draw (2.5,2.25) node {$p=1$};
\draw (1.5,2.25) node {$p_\Gamma = 1$};
\draw (1.5,-0.25) node {$p_\Gamma = 0$};
\end{tikzpicture}
\subcaption{Domain configuration}
\end{minipage}
\hfill
\begin{minipage}[b]{0.49\columnwidth}\center
\begin{minipage}[b]{0.25\columnwidth}\center\scriptsize
$1.05$\vfill\includegraphics[width=4cm,angle=90]{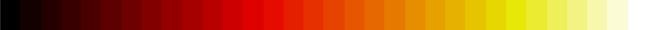} \vfill$0$
\end{minipage}
\hfill
\includegraphics[width=4.5cm]{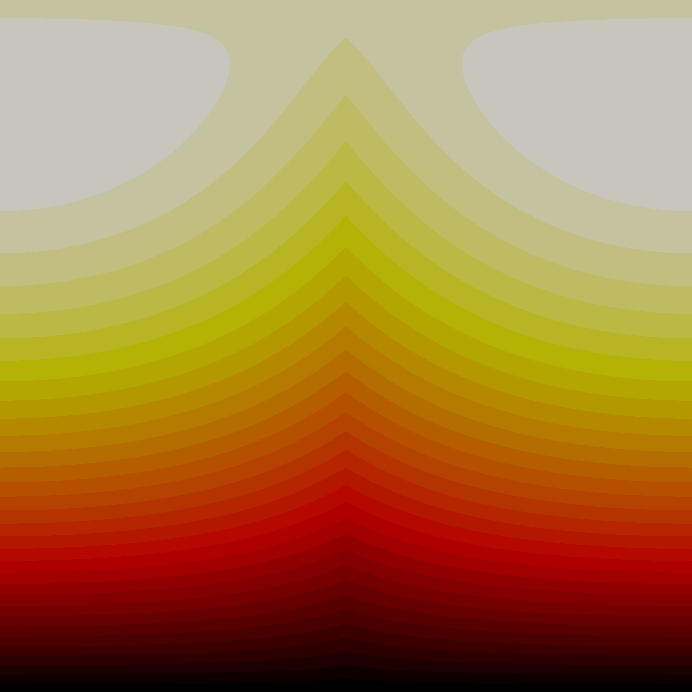} 
\subcaption{Bulk pressure $p$}
\end{minipage}
\end{center}
\caption{\label{pic:norandom}Bulk pressure for the first test case of Section~\ref{sec:num.tests:random} (homogeneous permeability).}
\end{figure}

\paragraph{Test 2: Random permeability} We next define inside the bulk region $\Omega_\bulk$ horizontal layers of random permeabilities which are separated by the fracture, and let the functions $\mu_1$ and $\mu_2$ take, inside each element, a random value between $0$ and $1$ on one side of each layer, and between $1$ and $2$ on the other side; see Figure~\ref{pic:randomvalues}. 
High permeability zones are prone to let the fluid flow towards the fracture, in contrast to the low permeability zones in which the pressure variations are larger; see Figure~\ref{pic:randomperlayer:pressure}, where dashed lines represent the different layers described above.
This qualitative behaviour is well captured by the numerical solution.

\begin{figure}
\begin{center}
\begin{minipage}[c]{0.49\columnwidth}\center
\begin{minipage}[c]{0.4\columnwidth}\center
\includegraphics[width=2.5cm]{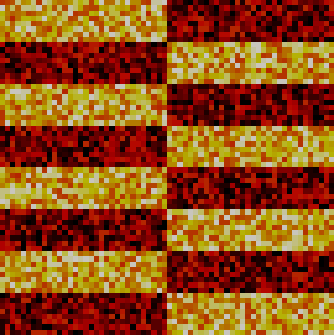}\vspace*{2.5cm}
\end{minipage}
\hfill
\begin{minipage}[c]{0.1\columnwidth}\center\scriptsize
$2$\vfill\includegraphics[width=2cm,angle=90]{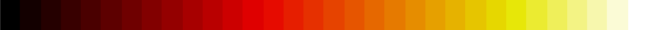} \vfill$10^{-2}$
\end{minipage}
\hfill
\begin{minipage}[c]{0.4\columnwidth}\center
\vspace*{2.5cm}\includegraphics[width=2.5cm]{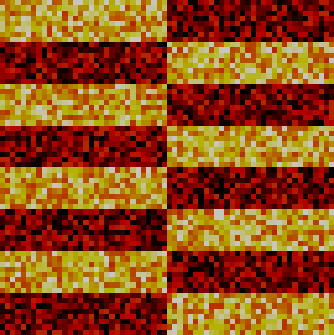}
\end{minipage}
\subcaption{Values of $\mu_1$ (left) and $\mu_2$ (right)\label{pic:randomvalues}}
\end{minipage}
\hfill
\begin{minipage}[c]{0.49\columnwidth}\center
\begin{minipage}[b]{0.25\columnwidth}\center\scriptsize
$1.14$\vfill\includegraphics[width=4cm,angle=90]{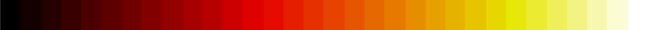} \vfill$-2.76\cdot 10^{-3}$
\end{minipage}
\hfill
\begin{tikzpicture}
\node[inner sep=0pt] (russell) at (0,0)
    {\includegraphics[width=4.5cm]{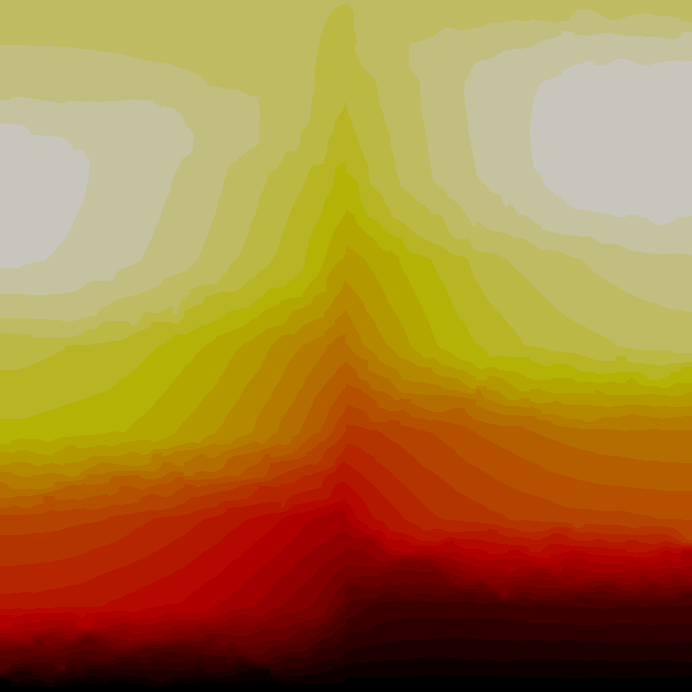}};
\draw[dashed, opacity=0.2] (0,-2.25) -- (0,2.25);
\draw[dashed, opacity=0.2] (-2.25,-1.6875) -- (2.25,-1.6875);
\draw[dashed, opacity=0.2] (-2.25,-1.125) -- (2.25,-1.125);
\draw[dashed, opacity=0.2] (-2.25,-0.5625) -- (2.25,-0.5625);
\draw[dashed, opacity=0.2] (-2.25,0) -- (2.25,0);
\draw[dashed, opacity=0.2] (-2.25,0.5625) -- (2.25,0.5625);
\draw[dashed, opacity=0.2] (-2.25,1.125) -- (2.25,1.125);
\draw[dashed, opacity=0.2] (-2.25,1.6875) -- (2.25,1.6875);
\end{tikzpicture}
\subcaption{Bulk pressure $p$\label{pic:randomperlayer:pressure}}
\end{minipage}
\end{center}
\caption{\label{pic:randomperlayer}Permeability components distribution and bulk pressure for the second test case of Section~\ref{sec:num.tests:random} (random permeability).}
\end{figure}
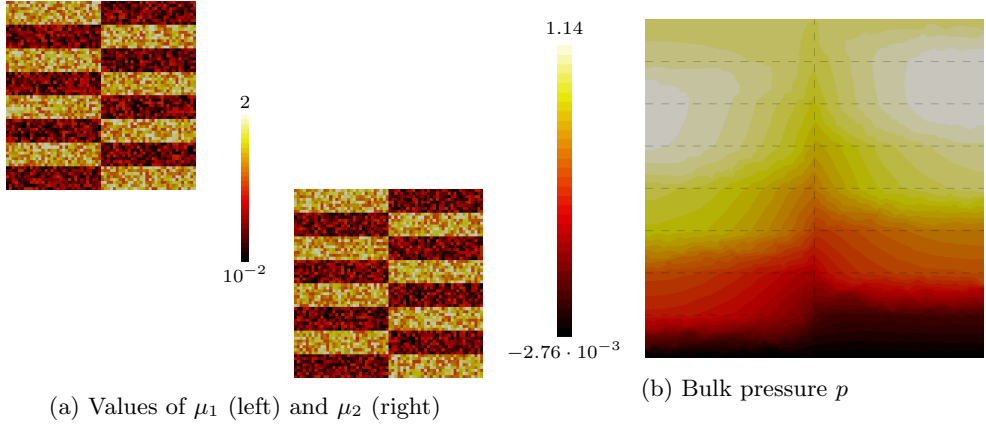


\section{Stability analysis}\label{sec:stability:proof}

This section contains the proof of Theorem~\ref{thm:stability} preceeded by the required preliminary results.
We recall that, for the sake of simplicity, we work here under the assumption that homogeneous Dirichlet boundary conditions are enforced on both the bulk and the fracture pressures; see~\eqref{eq:dirichlet.only}.
This simplifies the arguments of Lemma~\ref{lem:inf-sup.bh} below.

Recalling the definition~\eqref{def:ah} of $a_h^{\xi}$, and using~\eqref{proof:equivMTST} together with Cauchy--Schwarz inequalities, we infer the existence of a real number $\eta_a>0$ independent of $h$ and of the problem data such that, for all $\underline{\boldsymbol{v}}_h\in\widecheck{\underline{\boldsymbol{U}}}_h^k$,
\begin{equation}\label{eq:stability:ah}
  \eta_a^{-1}\|\underline{\boldsymbol{v}}_h\|_{\boldsymbol{U},\xi,h}^2
  \le \|\underline{\boldsymbol{v}}_h\|_{a,\xi,h}^2\coloneq a_h^{\xi}(\underline{\boldsymbol{v}}_h,\underline{\boldsymbol{v}}_h)
  \le \eta_a \varrho_\bulk\|\underline{\boldsymbol{v}}_h\|_{\boldsymbol{U},\xi,h}^2,
\end{equation}
with global bulk anisotropy ratio $\varrho_\bulk$ defined by~\eqref{eq:varrho.B}.
Similarly, summing~\eqref{lem:defstabmF} over $F\in\f_h^\Gamma$, it is readily inferred that it holds, for all $\underline{q}_h^\Gamma\in\underline{P}_{\Gamma,h}^k$,
\begin{equation}\label{eq:stability.dh}
  \eta_d^{-1}\|\underline{q}_h^\Gamma\|_{\Gamma,h}^2
  \le d_h(\underline{q}_h^\Gamma,\underline{q}_h^\Gamma)
  \le \eta_d \|\underline{q}_h^\Gamma\|_{\Gamma,h}^2.
\end{equation}

The following lemma contains a stability result for the bilinear form $b_h$.
\begin{lemma}[Inf-sup stability of $b_h$]\label{lem:inf-sup.bh}
  There is a real number $\beta>0$ independent of $h$, but possibly depending on $\varrho$, $k$, and on the problem geometry and data, such that, for all $q_h\in P_{\bulk,h}^k$,
  \begin{equation}\label{eq:inf-sup.bh}
    \|q_h\|_{\bulk,h}
    \le \beta
    \sup_{\underline{\boldsymbol{w}}_h\in\underline{\boldsymbol{U}}_{h,0}^k, \|\underline{\boldsymbol{w}}_h\|_{\boldsymbol{U},\xi,h}=1}b_h(\underline{\boldsymbol{w}}_h,q_h).
  \end{equation}
\end{lemma}
\begin{proof}
  We use the standard Fortin argument relying on the continuous inf-sup condition. In what follows, $a\lesssim b$ stands for the inequality $a\le C b$ with real number $C>0$ having the same dependencies as $\beta$ in~\eqref{eq:inf-sup.bh}.
  Let $q_h\in P_{\bulk,h}^k$.
  For each $i\in\{1,2\}$, the surjectivity of the continuous divergence operator from $\boldsymbol{H}(\Hdiv;\Omega_{\bulk,i})$ onto $L^2(\Omega_{\bulk,i})$ (see, e.g.,~\cite[Section~2.4.1]{Gatica:14}) yields the existence of $\boldsymbol{v}_i\in\boldsymbol{H}(\Hdiv;\Omega_{\bulk,i})$ such that
  \begin{equation}\label{eq:inf-sup.bh:v}
    \text{$\nabla\cdot\boldsymbol{v}_i=q_h$ in $\Omega_{\bulk,i}$
      and $\|\boldsymbol{v}_i\|_{\boldsymbol{H}(\Hdiv;\Omega_{\bulk,i})}\lesssim\|q_h\|_{\Omega_{\bulk,i}}$,}
  \end{equation}
  with hidden multiplicative constant depending on $\Omega_{\bulk,i}$.
  Let $\boldsymbol{v}:\Omega_\bulk\to\Real^2$ be such that $\boldsymbol{v}_{|\Omega_{\bulk,i}}=\boldsymbol{v}_i$ for $i\in\{1,2\}$.
  This function cannot be interpolated through $\underline{\boldsymbol{I}}_h^k$, as it does not belong to the space $H^1(\t_h)^2$ introduced in Section~\ref{sec:discrete:global:norms}; see also Remark~\ref{rem:U^+(T)} on this point.
  However, since we have assumed Dirichlet boundary conditions (cf.~\eqref{eq:dirichlet.only}), following the procedure described in~\cite[Section~4.1]{Gatica:14} one can construct smoothings $\tilde{\boldsymbol{v}}_i\in H^1(\Omega_{\bulk,i})^2$, $i\in\{1,2\}$, such that
    \begin{equation}\label{eq:inf-sup.bh:tv}
      \text{$\nabla\cdot\tilde{\boldsymbol{v}}_i=\nabla\cdot\boldsymbol{v}_i$ in $\Omega_{\bulk,i}$
        and $\|\tilde{\boldsymbol{v}}_i\|_{H^1(\Omega_{\bulk,i})^2}\lesssim\|\boldsymbol{v}_i\|_{\boldsymbol{H}(\Hdiv;\Omega_{\bulk,i})}$.}
  \end{equation}
  Let now $\tilde{\boldsymbol{v}}:\Omega_\bulk\to\Real^2$ be such that $\tilde{\boldsymbol{v}}_{|\Omega_{\bulk,i}}=\tilde{\boldsymbol{v}}_i$ for $i\in\{1,2\}$.
  The function $\tilde{\boldsymbol{v}}$ belongs to $\boldsymbol{U}\cap H^1(\t_h)^2$, and it can be easily checked that $\underline{\boldsymbol{I}}_h^k\tilde{\boldsymbol{v}}\in\underline{\boldsymbol{U}}_{h,0}^k$.
  Recalling the definition~\eqref{eq:normUT.normPT} of the $\|{\cdot}\|_{\boldsymbol{U},T}$-norm and using the boundedness of the $L^2$-orthogonal projector in the corresponding $L^2$-norm together with local continuous trace inequalities (see, e.g.,~\cite[Lemma~1.49]{Di-Pietro.Ern:12}), one has that
  \begin{equation}\label{eq:inf-sup.bh:est.norm.Ih.bulk}
    \sum_{T\in\t_h}\|\underline{\boldsymbol{I}}_T^k\tilde{\boldsymbol{v}}\|_{\boldsymbol{U},T}^2 
    \lesssim
    \sum_{i=1}^2\|\tilde{\boldsymbol{v}}_i\|_{H^1(\Omega_{\bulk,i})^2}^2
    \lesssim
    \sum_{i=1}^2\|\boldsymbol{v}_i\|_{\boldsymbol{H}(\Hdiv;\Omega_{\bulk,i})}^2
    \lesssim\|q_h\|_{\bulk,h}^2,
  \end{equation}
  where we have used~\eqref{eq:inf-sup.bh:tv} in the second inequality and~\eqref{eq:inf-sup.bh:v} in the third.
  The hidden constant depends here on $\lKB^{-1}$.
  Moreover, using a triangle inequality, the fact that $\lambda_F^\xi \le \lambda_F = (\lambda_\Gamma)_{|F} \le \ulkn$ (see~\eqref{eq:bnd.kappan}) for all $F\in\f_h^\Gamma$, the boundedness of the $L^2$-orthogonal projector, and a global continuous trace inequality in each bulk subdomain $\Omega_{\bulk,i}$, $i\in\{1,2\}$, we get
  \begin{equation}\label{eq:inf-sup.bh:est.norm.Ih.fract}
    |\underline{\boldsymbol{I}}_h^k\tilde{\boldsymbol{v}}|_{\xi,h}^2 
    \lesssim
    \sum_{i=1}^2 \|(\tilde{\boldsymbol{v}}_i)_{|\Gamma}\cdot\normal_{\Gamma}\|_{\Gamma}^2
    \lesssim
    \sum_{i=1}^2\|\tilde{\boldsymbol{v}}_i\|_{H^1(\Omega_{\bulk,i})^2}^2
    \lesssim\|q_h\|_{\bulk,h}^2,
  \end{equation}  
  where we have used~\eqref{eq:inf-sup.bh:tv} and~\eqref{eq:inf-sup.bh:v} in the third inequality.
  The hidden constant depends here on $\ulkn$ and on the inverse of the diameters of the bulk subdomains.
  Combining~\eqref{eq:inf-sup.bh:est.norm.Ih.bulk} and~\eqref{eq:inf-sup.bh:est.norm.Ih.fract}, and naming $\beta$ the hidden constant, we conclude that
  \begin{equation}\label{eq:inf-sup.bh:est.norm.Ih}
    \|\underline{\boldsymbol{I}}_h^k\tilde{\boldsymbol{v}}\|_{\boldsymbol{U},\xi,h} \le \beta\|q_h\|_{\bulk,h}.
  \end{equation}
  Finally,~\eqref{eq:inf-sup.bh:tv} together with the commuting property~\eqref{lem:commutdivbulk} of the local divergence reconstruction operator gives
  \begin{equation}\label{eq:inf-sup.bh:div.tv}
    \pi_T^k(\nabla\cdot\boldsymbol{v})
    = \pi_T^k(\nabla\cdot\tilde{\boldsymbol{v}})
    = D_T^k\underline{\boldsymbol{I}}_T^k\tilde{\boldsymbol{v}}_{|T}\qquad\forall T\in\t_h.
  \end{equation}
  Gathering all of the above properties, we infer that
  $$
  \|q_h\|_{\bulk,h}^2  
  = b(\boldsymbol{v},q_h)
  = b(\tilde{\boldsymbol{v}},q_h)
  = b_h(\underline{\boldsymbol{I}}_h^k\tilde{\boldsymbol{v}},q_h),
  $$
  where we have used~\eqref{eq:inf-sup.bh:v} together with the definition~\eqref{def:cont.bilin.forms} of $b$ in the first equality,~\eqref{eq:inf-sup.bh:tv} in the second, and~\eqref{eq:inf-sup.bh:div.tv} along with the definition~\eqref{eq:discret:2} of $b_h$ to conclude.
  Finally, factoring $\|\underline{\boldsymbol{I}}_h^k\tilde{\boldsymbol{v}}\|_{\boldsymbol{U},\xi,h}$, using the linearity of $b_h$ in its first argument, and denoting by $\$$ the supremum in~\eqref{eq:inf-sup.bh}, we get
  $$
  \|q_h\|_{\bulk,h}^2
  \le \$ \|\underline{\boldsymbol{I}}_h^k\tilde{\boldsymbol{v}}\|_{\boldsymbol{U},\xi,h}
  \le\beta\$\|q_h\|_{\bulk,h},
  $$
  where the conclusion follows from~\eqref{eq:inf-sup.bh:est.norm.Ih}.
  This proves~\eqref{eq:inf-sup.bh}.
\end{proof}

We next recall the following Poincar\'{e} inequality, which is a special case of the discrete Sobolev embeddings proved in~\cite[Proposition 5.4]{Di-Pietro.Droniou:16}: There exist a real number $C_{\rm P}>0$ independent of $h$ and of the problem data (but possibly depending on $\Gamma$ and $k$) such that, for all $\underline{q}_h^\Gamma = ((q_F^\Gamma)_{F\in\f_h^\Gamma},(q_V^\Gamma)_{V\in{\cal V}_h})\in\underline{P}_{\Gamma,h,0}^k$,
\begin{equation}\label{eq:poincare}
  \|q_h^\Gamma\|_\Gamma\le C_{\rm P} \lKG^{-\nicefrac12}\|\underline{q}_h^\Gamma\|_{\Gamma,h},
\end{equation}
where $q_h^\Gamma$ is the piecewise polynomial function on $\Gamma$ such that $(q_h^\Gamma)_{|F}=q_F^\Gamma$ for all $F\in\f_h^\Gamma$.

Using the Cauchy--Schwarz inequality together with the fact that $\lambda_F^\xi=(\lambda_\Gamma^\xi)_{|F}\ge\llkn\left(\frac\xi2-\frac14\right)$ for all $F\in\f_h^\Gamma$ (see~\eqref{eq:lambda.lambdaxi} and and~\eqref{eq:bnd.kappan}) and the Poincar\'{e} inequality~\eqref{eq:poincare}, we can prove the following boundedness property for the bilinear form $c_h$ defined by~\eqref{def:ch}:
For all $\underline{\boldsymbol{v}}_h\in\underline{\boldsymbol{U}}_{h,0}^k$ and all $\underline{q}_h^\Gamma\in\underline{P}_{\Gamma,h,0}^k$, it holds that
\begin{equation}\label{eq:boundedness:ch}
  |c_h(\underline{\boldsymbol{v}}_h,\underline{q}_h^\Gamma)|
  \le\eta_c\llkn^{-\nicefrac12}|\underline{\boldsymbol{v}}_h|_{\xi,h} \|\underline{q}_h^\Gamma\|_{\Gamma,h},\qquad
  \eta_c\coloneq C_{\rm P}\left(\frac\xi2-\frac14\right)^{-\nicefrac12}.
\end{equation}

We are now ready to prove Theorem~\ref{thm:stability}.
\begin{proof}[Proof of Theorem~\ref{thm:stability}]
  Let $\underline{\boldsymbol{z}}_h\coloneq(\boldsymbol{\underline{w}}_h,r_h,\underline{r}_h^\Gamma)\in\underline{\boldsymbol{X}}_h^k$.
  In the spirit of \cite[Lemma~4.38]{Ern.Guermond:04}, the proof proceeds in three steps.

  \paragraph{Step 1: Control of the flux in the bulk and of the pressure in the fracture}
  Using the coercivity~\eqref{eq:stability:ah} of the bilinear form $a_h^\xi$ and~\eqref{eq:stability.dh} of the bilinear form $d_h$, it is inferred that
  \begin{align}\label{eq:inf-sup.ah:1}
    \mathcal{A}_h^\xi(\underline{\boldsymbol{z}}_h,\underline{\boldsymbol{z}}_h)
    \ge \eta_a^{-1}\|\underline{\boldsymbol{w}}_h\|_{\boldsymbol{U},\xi,h}^2
    + \eta_d^{-1}\|\underline{r}_h^\Gamma\|_{\Gamma,h}^2.
  \end{align}

  \paragraph{Step 2: Control of the pressure in the bulk}
  The inf-sup condition~\eqref{eq:inf-sup.bh} on the bilinear form $b_h$ gives the existence of $\underline{\boldsymbol{v}}_h \in \underline{\boldsymbol{U}}_{h,0}^k$ such that
  \begin{equation}\label{eq:inf-sup.ah:vh}
    \text{$\|r_h\|_{\bulk,h}^2 = -b_h(\underline{\boldsymbol{v}}_h,r_h)$
      and $\|\underline{\boldsymbol{v}}_h\|_{\boldsymbol{U},\xi,h} \le\beta\|r_h\|_{\bulk,h}$.}
  \end{equation}
  Using the definition \eqref{def:Ahxi} of $\mathcal{A}_h^\xi$, it is readily inferred that
  \begin{equation}\label{eq:inf-sup.ah:2}
    \begin{aligned}
    \mathcal{A}_h^\xi(\underline{\boldsymbol{z}}_h, (\underline{\boldsymbol{v}}_h,0,\underline{0}))
    &= \|r_h\|_{\bulk,h}^2 + a_h^\xi(\underline{\boldsymbol{w}}_h,\underline{\boldsymbol{v}}_h) + c_h(\underline{\boldsymbol{v}}_h,\underline{r}_h^\Gamma)
    \\
    &\ge \|r_h\|_{\bulk,h}^2 - |a_h^\xi(\underline{\boldsymbol{w}}_h,\underline{\boldsymbol{v}}_h)| - |c_h(\underline{\boldsymbol{v}}_h,\underline{r}_h^\Gamma)|.
    \end{aligned}
  \end{equation}
  Using the continuity of $a_h^\xi$ expressed by the second inequality in~\eqref{eq:stability:ah} followed by Young's inequality, we infer that it holds, for all $\epsilon>0$,
  \begin{equation}\label{eq:inf-sup.ah:3}
    |a_h^\xi(\underline{\boldsymbol{w}}_h,\underline{\boldsymbol{v}}_h)|
    \le\eta_a\varrho_\bulk\|\underline{\boldsymbol{w}}_h\|_{\boldsymbol{U},\xi,h}\|\underline{\boldsymbol{v}}_h\|_{\boldsymbol{U},\xi,h}
    \le\frac{\epsilon}{4}\|\underline{\boldsymbol{v}}_h\|_{\boldsymbol{U},\xi,h}^2 + \frac{(\eta_a\varrho_\bulk)^2}{\epsilon}\|\underline{\boldsymbol{w}}_h\|_{\boldsymbol{U},\xi,h}^2.
  \end{equation}
  Similarly, the boundedness~\eqref{eq:boundedness:ch} of $c_h$ followed by Young's inequality gives
  \begin{equation}\label{eq:inf-sup.ah:4}
    |c_h(\underline{\boldsymbol{v}}_h,\underline{r}_h^\Gamma)|
    \le\eta_c\llkn^{-\nicefrac12}\|\underline{\boldsymbol{v}}_h\|_{\boldsymbol{U},\xi,h}\|\underline{r}_h^\Gamma\|_{\Gamma,h}
    \le\frac{\epsilon}{4}\|\underline{\boldsymbol{v}}_h\|_{\boldsymbol{U},\xi,h}^2
    + \frac{\eta_c^2}{\epsilon\llkn}\|\underline{r}_h^\Gamma\|_{\Gamma,h}^2.
  \end{equation}
  Plugging~\eqref{eq:inf-sup.ah:3} and~\eqref{eq:inf-sup.ah:4} into~\eqref{eq:inf-sup.ah:2}, selecting $\epsilon=\beta^{-2}$, and using the bound in~\eqref{eq:inf-sup.ah:vh}, we arrive at
  \begin{equation}\label{eq:inf-sup.ah:5}
    \mathcal{A}_h^\xi(\underline{\boldsymbol{z}}_h, (\underline{\boldsymbol{v}}_h,0,\underline{0}))
    \ge\frac12\|r_h\|_{\bulk,h}^2 - C_1 \|\underline{\boldsymbol{w}}_h\|_{\boldsymbol{U},\xi,h}^2
    - C_2 \|\underline{r}_h^\Gamma\|_{\Gamma,h}^2,
  \end{equation}
  with $C_1\coloneq(\eta_a\beta\varrho_\bulk)^2$ and $C_2\coloneq(\eta_c\beta)^2/\llkn$.

  \paragraph{Step 3: Conclusion}
  Setting $\alpha\coloneq(1+C_1\eta_a+C_2\eta_d)^{-1}/2$ and combining~\eqref{eq:inf-sup.ah:1} with~\eqref{eq:inf-sup.ah:5}, we infer that
  \begin{multline*}
    \mathcal{A}_h^\xi(\underline{\boldsymbol{z}}_h, (1-\alpha)\underline{\boldsymbol{z}}_h + \alpha(\underline{\boldsymbol{v}}_h,0,\underline{0}))
    \\
    \ge
    \frac\alpha2\|r_h\|_{\bulk,h}^2
    + \eta_a^{-1}\left(1-\alpha(1+C_1\eta_a)\right)\|\underline{\boldsymbol{w}}_h\|_{\boldsymbol{U},\xi,h}^2
    + \eta_d^{-1}\left(1-\alpha(1+C_2\eta_d)\right)\|\underline{r}_h^\Gamma\|_{\Gamma,h}^2.
  \end{multline*}
  Denoting by $\$$ the supremum in the right-hand side of~\eqref{thm:stability:eq}, we infer from the previous inequality that
  \begin{equation}\label{eq:inf-sup.ah:7}
    C_3\|\underline{\boldsymbol{z}}_h\|_{\boldsymbol{X},h}^2\le
    \mathcal{A}_h^\xi(\underline{\boldsymbol{z}}_h, (1-\alpha)\underline{\boldsymbol{z}}_h + \alpha(\underline{\boldsymbol{v}}_h,0,\underline{0}))
    \le \$\|(1-\alpha)\underline{\boldsymbol{z}}_h+\alpha(\underline{\boldsymbol{v}}_h,0,\underline{0})\|_{\boldsymbol{X},h}
  \end{equation}
  with $C_3\coloneq\min\left(\nicefrac\alpha2,\eta_a^{-1}(1-\alpha(1+C_1\eta_a)),\eta_d^{-1}(1-\alpha(1+C_2\eta_d))\right)>0$.
  Finally, observing that, by the definition~\eqref{def:Xnorm} of the $\|{\cdot}\|_{\boldsymbol{X},h}$-norm together with~\eqref{eq:inf-sup.ah:vh}, it holds that
  $\|(\underline{\boldsymbol{v}}_h,0,\underline{0})\|_{\boldsymbol{X},h}\le\beta\|r_h\|_{\bulk,h}\le\beta\|\underline{\boldsymbol{z}}_h\|_{\boldsymbol{X},h},$
  \eqref{eq:inf-sup.ah:7} gives~\eqref{thm:stability:eq} with $\gamma=C_3^{-1}(1+\beta)$.
\end{proof}


\section{Error analysis}\label{sec:error.analysis}

This section contains the proof of Theorem~\ref{thm:energyerror} preceeded by the required preliminary results.
As in the previous section, we work under the assumption that homogeneous Dirichlet boundary conditions are enforced on both the bulk and the fracture pressures; see~\eqref{eq:dirichlet.only}.
In what follows, $a\lesssim b$ means $a\le Cb$ with real number $C>0$ independent of $h$ and of the problem data, but possibly depending on $\varrho$, $k$, and on the problem geometry.

For all $T\in\t_h$, we define the local elliptic projection $\widecheck{p}_T\in\mathbb{P}^{k+1}(T)$ of the bulk pressure $p$ such that
\begin{equation}\label{eq:cpT}
  \text{%
    $(\boldsymbol{K}_T\nabla(\widecheck{p}_T-p),\nabla w)_T = 0$ for all $w\in\mathbb{P}^{k+1}(T)$ and $(\widecheck{p}_T-p,1)_T = 0$.%
  }
\end{equation}
Adapting the results of~\cite[Lemma 3]{Di-Pietro.Ern.ea:14}, it can be proved that the following approximation properties hold for all $T\in\t_h$ provided that $p_{|T}\in H^{k+2}(T)$:
\begin{equation}\label{eq:approx:cpT}
  \begin{aligned}
    &\|\boldsymbol{K}_T^{\nicefrac12}\nabla(p-\widecheck{p}_T)\|_T
    + h_T^{\nicefrac12}\|\boldsymbol{K}_T^{\nicefrac12}\nabla(p_{|T}-\widecheck{p}_T)\|_{\partial T}
    \\
    &\qquad
    + \lKBT^{\nicefrac12}h_T^{-1}\|p-\widecheck{p}_T\|_T
    + \lKBT^{\nicefrac12}h_T^{-\nicefrac12}\|p_{|T}-\widecheck{p}_T\|_{\partial T}
    \lesssim\uKBT^{\nicefrac12} h_T^{k+1}\|p\|_{H^{k+2}(T)}.
  \end{aligned}
\end{equation}
Note that we need to specify that the trace of $p$ and of the corresponding flux are taken from the side of $T$ in boundary norms, since these quantities are possibly two-valued on fracture faces.
We also introduce the broken polynomial function $\widecheck{p}_h$ such that
$$
(\widecheck{p}_h)_{|T}=\widecheck{p}_T\qquad\forall T\in\t_h.
$$

The following boundedness result for the bilinear form $b_h$ defined by~\eqref{def:bh} can be proved using~\eqref{proof:stabDiv}:
For all $\underline{\boldsymbol{v}}_h\in\widecheck{\underline{\boldsymbol{U}}}_h^k$ and all $q_h\in P_{\bulk,h}^k$,
\begin{equation}\label{eq:continuity:bh}
  \begin{aligned}
    |b_h(\underline{\boldsymbol{v}}_h,q_h)|
    &\lesssim\left(
    \sum_{T\in\t_h} \|\underline{\boldsymbol{v}}_T\|_{\boldsymbol{U},T}^2
    \right)^{\nicefrac12}\times
    \left(
    \sum_{T\in\t_h} \uKBT h_T^{-2}\|q_T\|_T^2
    \right)^{\nicefrac12}\\
    &\lesssim\|\underline{\boldsymbol{v}}_h\|_{m,h}
    \left(
    \sum_{T\in\t_h} \uKBT h_T^{-2}\|q_T\|_T^2
    \right)^{\nicefrac12},
  \end{aligned}
\end{equation}
where, to obtain the second inequality, we have used the first bound in~\eqref{proof:equivMTST} and summed over $T\in\t_h$ to infer
$$
\sum_{T\in\t_h}
\|\underline{\boldsymbol{v}}_T\|_{\boldsymbol{U},T}^2
\lesssim\|\underline{\boldsymbol{v}}_h\|_{m,h}^2
\coloneq\sum_{T\in\t_h}\|\underline{\boldsymbol{v}}_T\|_{m,T}^2.
$$

Finally, we note the following consistency property for the bilinear form $d_h$ defined by~\eqref{def:dh}, which can be inferred from~\cite[Theorem~8]{Di-Pietro.Ern.ea:14}:
For all $q\in H_0^1(\Gamma)$ such that $q\in H^{k+2}(F)$ for all $F\in\f_h^\Gamma$,
\begin{equation}\label{eq:consistency:dh}
  \begin{aligned}
  \sup_{\underline{r}_h^\Gamma\in\underline{P}_{\Gamma,h,0}^k, \|\underline{r}_h^\Gamma\|_{\Gamma,h}=1} \hspace{-0.5ex}\Bigg(
  \sum_{F\in\f_h^\Gamma}(\nabla_\tau{\cdot}(K_F\nabla_\tau q),r_F^\Gamma)_F
  &+d_h(\underline{I}_h^k q,\underline{r}_h^\Gamma)
  \Bigg)
  \\
  & \lesssim\hspace{-0.5ex}\left(
  \sum_{F\in\f_h^\Gamma} K_F h_F^{2(k+1)}\|q\|_{H^{k+2}(F)}^2
  \right)^{\nicefrac12}\hspace{-2ex}.
  \end{aligned}
\end{equation}

We are now ready to prove the error estimate.
\begin{proof}[Proof of Theorem \ref{thm:energyerror}]
  The proof proceeds in five steps:
  in {\bf Step 1} we derive an estimate for the discretization error measured by the left-hand side of~\eqref{eq:energyerror} in terms of a conformity error;
  in {\bf Step 2} we bound the different components of the conformity error;
  in {\bf Step 3} we combine the previous results to obtain~\eqref{eq:energyerror}.
  {\bf Steps 4-5} contain the proofs of technical results used in {\bf Step 2}.

  \begin{remark}[Role of {\bf Step 1}]
    The discretization error in the left-hand side of~\eqref{eq:energyerror} can be clearly estimated in terms of a conformity error using the inf-sup condition on $\mathcal{A}_h^\xi$ proved in Theorem~\ref{thm:stability}. Proceeding this way, however, we would end up with constants depending on the problem data (and, in particular, on the global bulk anisotropy ratio $\varrho_\bulk$ defined by~\eqref{eq:varrho.B}) in the right-hand side of~\eqref{eq:energyerror}.
    This is to be avoided if one wants to have a sharp indication of the behaviour of the method for strongly anisotropic bulk permeability tensors.
  \end{remark}
  
  In what follows, we use the shortcut notation for the error components introduced in~\eqref{eq:error.shortcuts}.

  \paragraph{Step 1: Basic error estimate}
  Recalling the definitions~\eqref{def:Ahxi} of $\mathcal{A}_h^{\xi}$ and~\eqref{eq:stability:ah} of the norm $\|{\cdot}\|_{a,\xi,h}$, and using the coercivity of $d_h$ expressed by the first inequality in~\eqref{eq:stability.dh}, we have that
  \begin{equation}\label{eq:basic.est:1}
    \|\underline{\boldsymbol{e}}_h\|_{a,\xi,h}^2
    + \|\underline{\epsilon}_h^\Gamma\|_{\Gamma,h}^2
    \lesssim
    \mathcal{A}_h^\xi((\underline{\boldsymbol{e}}_h,\epsilon_h,\underline{\epsilon}_h^\Gamma),(\underline{\boldsymbol{e}}_h,\epsilon_h,\underline{\epsilon}_h^\Gamma))
    = 
    \mathcal{E}_{h,1}(\underline{\boldsymbol{e}}_h)
    + \mathcal{E}_{h,2}(\epsilon_h)
    + \mathcal{E}_{h,3}(\underline{\epsilon}_h^\Gamma),
  \end{equation}
  where the linear forms
  $\mathcal{E}_{h,1}:\boldsymbol{\underline{U}}_{h,0}^k\to\Real$,
  $\mathcal{E}_{h,2}:P_{\bulk,h}^k\to\Real$, and
  $\mathcal{E}_{h,3}:\underline{P}_{\Gamma,h,0}^k\to\Real$
  correspond to the components of the conformity error and are defined such that
  \begin{subequations}\label{eq:err.decomp}
    \begin{align}\label{def:Eh1}
      \mathcal{E}_{h,1}(\boldsymbol{\underline{v}}_h)
      &\coloneq
      - a_h^\xi(\widehat{\boldsymbol{\underline{u}}}_h,\boldsymbol{\underline{v}}_h)
      + b_h(\underline{\boldsymbol{v}}_h,\widehat{p}_h)
      - c_h(\underline{\boldsymbol{v}}_h,\widehat{\underline{p}}_h^\Gamma),
      \\\label{def:Eh2}
      \mathcal{E}_{h,2}(q_h)
      &\coloneq
      \sum_{T\in\t_{h}}(f,q_T)_T - b_h(\widehat{\boldsymbol{\underline{u}}}_h,q_h),      
      \\\label{def:Eh3}
      \mathcal{E}_{h,3}(\underline{q}_h^\Gamma)
      &\coloneq
      \sum_{F\in\f_{h}^\Gamma}(\ell_F f_\Gamma,q_F^\Gamma)_F
      + c_h(\widehat{\boldsymbol{\underline{u}}}_h,\underline{q}_h^\Gamma)        
      - d_h(\widehat{\underline{p}}_h^\Gamma,\underline{q}_h^\Gamma).
    \end{align}
  \end{subequations}
  We next estimate the error $\epsilon_h$ on the bulk pressure.
  The inf-sup condition~\eqref{eq:inf-sup.bh} yields the existence of $\underline{\boldsymbol{v}}_h\in\underline{\boldsymbol{U}}_{h,0}^k$ such that
  \begin{equation}\label{eq:err.est:vh}
  \text{$\|\epsilon_h\|_{\bulk,h}^2 = -b_h(\underline{\boldsymbol{v}}_h,\epsilon_h)$
    and $\|\underline{\boldsymbol{v}}_h\|_{\boldsymbol{U},\xi,h} \le\beta\|\epsilon_h\|_{\bulk,h}$.}
  \end{equation}
  Hence,
  $$
  \begin{aligned}
    \|\epsilon_h\|_{\bulk,h}^2
    &= b_h(\underline{\boldsymbol{v}}_h,p_h) - b_h(\underline{\boldsymbol{v}}_h,\widehat{p}_h)
    \\
    &=a_h^\xi(\underline{\boldsymbol{u}}_h,\underline{\boldsymbol{v}}_h)
    + c_h(\underline{\boldsymbol{v}}_h,\underline{p}_h^\Gamma)
    - b_h(\underline{\boldsymbol{v}}_h,\widehat{p}_h)
    \\
    &= a_h^\xi(\underline{\boldsymbol{e}}_h,\underline{\boldsymbol{v}}_h)
    + c_h(\underline{\boldsymbol{v}}_h,\underline{\epsilon}_h^\Gamma)
    - \mathcal{E}_{h,1}(\underline{\boldsymbol{v}}_h),
  \end{aligned}
  $$
  where we have used the linearity of $b_h$ in its second argument in the first line,~\eqref{eq:discret:1} in the second line (recall that $g_\bulk\equiv 0$ owing to~\eqref{eq:dirichlet.only}), and we have inserted $\pm\big(a_h^\xi(\widehat{\underline{\boldsymbol{u}}}_h,\underline{\boldsymbol{v}}_h)
  + c_h(\underline{\boldsymbol{v}}_h,\widehat{\underline{p}}_h^\Gamma)\big)$ to conclude.
  Using the Cauchy--Schwarz inequality together with~\eqref{eq:stability:ah} for the first term, the boundedness~\eqref{eq:boundedness:ch} of the second, and the linearity of $\mathcal{E}_{h,1}$ together with the second bound in~\eqref{eq:stability:ah} for the third, we get
  $$
  \|\epsilon_h\|_{\bulk,h}^2
  \lesssim\left(
  \varrho_\bulk^{\nicefrac12}\|\underline{\boldsymbol{e}}_h\|_{a,\xi,h}
  + \llkn^{-\nicefrac12}\|\underline{\epsilon}_h^\Gamma\|_{\Gamma,h}
  + \varrho_\bulk^{\nicefrac12}\mathcal{E}_{h,1}(\underline{\boldsymbol{v}}_h/\|\underline{\boldsymbol{v}}_h\|_{a,\xi,h})
  \right)\|\underline{\boldsymbol{v}}_h\|_{\boldsymbol{U},\xi,h}.
  $$
  Using the inequality in~\eqref{eq:err.est:vh} to bound the second factor,  and naming $\chi$ the hidden constant, we arrive at
  \begin{equation}\label{eq:basic.est:2}
    \chi\|\epsilon_h\|_{\bulk,h}
    \le
    \|\underline{\boldsymbol{e}}_h\|_{a,\xi,h}
    + \|\underline{\epsilon}_h^\Gamma\|_{\Gamma,h}
    + \mathcal{E}_{h,1}(\underline{\boldsymbol{v}}_h/\|\underline{\boldsymbol{v}}_h\|_{\boldsymbol{U},\xi,h}).
  \end{equation}

  \paragraph{Step 2: Bound of the conformity error components}
  We proceed to bound the conformity error components for a generic $(\boldsymbol{\underline{v}}_h,q_h,\underline{q}_h^\Gamma)\in\boldsymbol{\underline{X}}_h$.

    To bound $\mathcal{E}_{h,1}$, we use the following reformulations of the first and second contribution, whose proofs are given in {\bf Steps 4-5} below:
  \begin{equation}\label{eq:Eh11}
    \begin{aligned}
      a_h^\xi(\widehat{\boldsymbol{\underline{u}}}_h,\boldsymbol{\underline{v}}_h)&=
      \sum_{F\in\f_{h}^\Gamma}\left(
      (\lambda_F^\xi[\![\boldsymbol{u}]\!]_\Gamma\cdot\normal_\Gamma,[\![\underline{\boldsymbol{v}}_{h}]\!]_F)_F
      + (\lambda_F \{\!\{\boldsymbol{u}\}\!\}_\Gamma\cdot\normal_\Gamma,\{\!\{\underline{\boldsymbol{v}}_{h}\}\!\}_F)_F
      \right)
      \\
      &\qquad
      +\sum_{T\in\t_h}\sum_{F\in\f_T}(\boldsymbol{K}_T\nabla(\widecheck{p}_T-p_{|T})\cdot\normal_{TF},\pi_F^kw_T - \pi_T^kw_T)_F
      \\
      &\qquad\qquad
      -\sum_{T\in\t_h}(\nabla p,\FT\underline{\boldsymbol{v}}_T)_T-\sum_{T\in\t_h} J_T(\widehat{\underline{\boldsymbol{u}}}_T,\underline{\boldsymbol{v}}_T)_T,
    \end{aligned}
  \end{equation}
  where, for all $T\in\t_h$, $w_T\in\mathbb{P}^{k+1}(T)$ is such that $\FT\underline{\boldsymbol{v}}_T=\boldsymbol{K}_T\nabla w_T$ and
  \begin{equation}\label{eq:Eh12}
    \begin{aligned}
      b_h(\underline{\boldsymbol{v}}_h,\widehat{p}_h)&=
      b_h(\underline{\boldsymbol{v}}_h,\pi_h^k(p-\widecheck{p}_h))     
      + \sum_{T\in\t_h}\sum_{F\in\f_T}(\widecheck{p}_T-p_{|T}, v_{TF})_F
      + c_h(\underline{\boldsymbol{v}}_h, \widehat{\underline{p}}_h^\Gamma)
      \\
      &\qquad+\sum_{F\in\f_{h}^\Gamma}\left(
      (\lambda_F^\xi[\![\boldsymbol{u}]\!]_\Gamma\cdot\normal_\Gamma,[\![\underline{\boldsymbol{v}}_{h}]\!]_F)_F
      + (\lambda_F \{\!\{\boldsymbol{u}\}\!\}_\Gamma\cdot\normal_\Gamma,\{\!\{\underline{\boldsymbol{v}}_{h}\}\!\}_F)_F
      \right)
      \\
      &\qquad\qquad-\sum_{T\in\t_h}(\nabla p,\FT\underline{\boldsymbol{v}}_T)_T.      
    \end{aligned}
  \end{equation}
  Using~\eqref{eq:Eh11} and~\eqref{eq:Eh12} in~\eqref{def:Eh1}, we infer that
  $$
  \begin{aligned}
    \mathcal{E}_{h,1}(\boldsymbol{\underline{v}}_h)
    &= b_h(\underline{\boldsymbol{v}}_h,\pi_h^k(p-\widecheck{p}_h))
    + \sum_{T\in\t_h}\sum_{F\in\f_T}(\widecheck{p}_T-p_{|T}, v_{TF})_F
    \\
    &
    -\sum_{T\in\t_h}\sum_{F\in\f_T}(\boldsymbol{K}_T\nabla(\widecheck{p}_T-p_{|T})\cdot\normal_{TF},\pi_F^kw_T - \pi_T^kw_T)_F
    + \sum_{T\in\t_h} J_T(\widehat{\underline{\boldsymbol{u}}}_T,\underline{\boldsymbol{v}}_T)_T.
  \end{aligned}
  $$
  Using the boundedness~\eqref{eq:continuity:bh} of $b_h$ together with the third bound in~\eqref{eq:approx:cpT} to estimate the first term, Cauchy--Schwarz inequalities together with the fourth bound in~\eqref{eq:approx:cpT} and the first bound in~\eqref{proof:equivMTST} to estimate the second term, Cauchy--Schwarz inequalities together with the fact that $
  {h_T^{-\nicefrac12}\|\pi_F^kw_T-\pi_T^kw_T\|_F}\lesssim {h_T^{-1}\|w_T-\pi_T^kw_T\|_T}\lesssim
  {\lKBT^{-\nicefrac12}\|\FT\underline{\boldsymbol{v}}_T\|_T}$ (a consequence of the $L^2(F)$-boundedness of $\pi_F^k$ and~\eqref{eq:approx.trace:lproj} with $l=k+1$, $m=0$, and $s=1$) to estimate the third term, and~\eqref{eq:consistency:JT} to estimate the fourth term, we infer that
  \begin{equation}\label{eq:Eh1}
    |\mathcal{E}_{h,1}(\boldsymbol{\underline{v}}_h)|\lesssim\left(
    \sum_{T\in\t_h}\varrho_{\bulk,T}\uKBT h_T^{2(k+1)} \|p\|_{H^{k+2}(T)}^2
    \right)^{\nicefrac12}\|\underline{\boldsymbol{v}}_h\|_{m,h}.
  \end{equation}

  For the second error component, using~\eqref{eq:strong:bulk:2}, the definition~\eqref{def:bh} of the bilinear form $b_h$, and the commuting property~\eqref{lem:commutdivbulk} of the local divergence reconstruction, we get
  \begin{equation}\label{eq:Eh2}
    \mathcal{E}_{h,2}(\boldsymbol{\underline{v}}_h)
    = \sum_{T\in\t_h}
    (\nabla\cdot\boldsymbol{u} - \pi_T^k(\nabla\cdot\boldsymbol{u}),q_T)_T
    =0,
  \end{equation}
  where we have used the fact that $q_T\in\mathbb{P}^k(T)$ and the definition~\eqref{def:lproj} of $\pi_{T}^{k}$ to conclude.

  We next observe that, for all $F\in\f_T^\Gamma$ such that $F\subset\partial T_1\cap\partial T_2$ for distinct mesh elements $T_1,T_2\in\t_h$,
  \begin{subequations}\label{eq:Eh11:4}
      \begin{align}\label{eq:Eh11:4:jump}
        [\![\widehat{\underline{\boldsymbol{u}}}_h]\!]_F
        &= \pi_F^k\left(
        \boldsymbol{u}_{|T_1}\cdot\normal_{T_1F} + \boldsymbol{u}_{|T_2}\cdot\normal_{T_2F}
        \right)
        = \pi_F^k\left([\![\boldsymbol{u}]\!]\cdot\normal_\Gamma\right),
        \\ \label{eq:Eh11:4:average}
        \{\!\{\widehat{\underline{\boldsymbol{u}}}_h\}\!\}_F
        &= \frac12\pi_F^k\left(\boldsymbol{u}_{|T_1}\cdot\normal_{\Gamma}
        + \boldsymbol{u}_{|T_2}\cdot\normal_\Gamma    
        \right)
        = \pi_F^k\left(\{\!\{\boldsymbol{u}\}\!\}\cdot\normal_\Gamma\right).
      \end{align}
  \end{subequations}
  For the third error component, we can then write
  $$
  \begin{aligned}
    \mathcal{E}_{h,3}(q_h)
    &= \sum_{F\in\f_h^\Gamma}(\ell_F f_\Gamma + [\![ \widehat{\underline{\boldsymbol{u}}}_h ]\!]_F, q_F^\Gamma)_F
    - d_h(\widehat{\underline{p}}_h^\Gamma,\underline{q}_h^\Gamma)
    \\
    &= \sum_{F\in\f_h^\Gamma}(\ell_F f_\Gamma + [\![ \boldsymbol{u} ]\!]_\Gamma\cdot\normal_\Gamma, q_F^\Gamma)_F
    - d_h(\widehat{\underline{p}}_h^\Gamma,\underline{q}_h^\Gamma)
    \\
    &= -\sum_{F\in\f_h^\Gamma}(\nabla_\tau\cdot(K_F\nabla_\tau p_\Gamma),q_F^\Gamma)_F
    - d_h(\widehat{\underline{p}}_h^\Gamma,\underline{q}_h^\Gamma),
  \end{aligned}
  $$
  where we have expanded the bilinear form $c_h$ according to its definition~\eqref{def:ch} in the first line, we have used~\eqref{eq:Eh11:4:jump} followed by~\eqref{def:lproj} and the fact that $q_F^\Gamma\in\mathbb{P}^k(F)$ to remove $\pi_F^k$ in the second line, and we have concluded invoking~\eqref{eq:strong:fract:1}.
  The consistency property~\eqref{eq:consistency:dh} then gives
  \begin{equation}\label{eq:Eh3}
    |\mathcal{E}_{h,3}(q_h)|\lesssim\left(
    \sum_{F\in\f_h^\Gamma}K_F h_F^{2(k+1)}\|p_\Gamma\|_{H^{k+2}(F)}^2
    \right)\|\underline{q}_h^\Gamma\|_{\Gamma,h}.
  \end{equation}

  \paragraph{Step 3: Conclusion}
  Using~\eqref{eq:Eh1},~\eqref{eq:Eh2}, and~\eqref{eq:Eh3} with $(\underline{\boldsymbol{v}}_h,q_h,\underline{q}_h^\Gamma)=(\underline{\boldsymbol{e}}_h,\epsilon_h,\underline{\epsilon}_h^\Gamma)$ to estimate the right-hand side of~\eqref{eq:basic.est:1}, and recalling that $\|\underline{\boldsymbol{e}}_h\|_{m,h}\le\|\underline{\boldsymbol{e}}_h\|_{a,\xi,h}$, we infer that
  \begin{equation}\label{eq:err.est:bulk.flux+fracture}
    \begin{aligned}
    \|\underline{\boldsymbol{e}}_h\|_{a,\xi,h} + \|\underline{\epsilon}_h^\Gamma\|_{\Gamma,h}
    \lesssim\Bigg(
    \sum_{T\in\t_h}\varrho_{\bulk,T}\uKBT &h_T^{2(k+1)} \|p\|_{H^{k+2}(T)}^2
    \\
    &+ \sum_{F\in\f_h^\Gamma}K_F h_F^{2(k+1)}\|p_\Gamma\|_{H^{k+2}(F)}^2
    \Bigg)^{\nicefrac12},
    \end{aligned}
  \end{equation}
  which, in view of the first inequality in~\eqref{eq:stability:ah}, gives the bounds on the first and second term in the left-hand side of~\eqref{eq:energyerror}.
  Plugging~\eqref{eq:err.est:bulk.flux+fracture} and~\eqref{eq:Eh1} into~\eqref{eq:basic.est:2}, and recalling that $\|\underline{\boldsymbol{v}}_{h}\|_{m,h}\le\|\underline{\boldsymbol{v}}_{h}\|_{a,\xi,h}$ gives the estimate for the third term in the left-hand side of~\eqref{eq:energyerror}.

  \paragraph{Step 4: Proof of~\eqref{eq:Eh11}}%
  For every mesh element $T\in\t_h$, we have that
  \begin{equation}\label{eq:Eh11:1}
    \begin{aligned}
      (\boldsymbol{K}_T^{-1}\FT\widehat{\underline{\boldsymbol{u}}}_T,\FT\underline{\boldsymbol{v}}_T)_T
      &= (\FT\widehat{\underline{\boldsymbol{u}}}_T,\nabla w_T)_T
      \\
      &= -(\DT\widehat{\underline{\boldsymbol{u}}}_T, w_T)_T + \sum_{f\in\f_T}(\widehat{u}_{TF},w_T)_F
      \\
      &= -(\pi_T^k(\nabla\cdot\boldsymbol{u}), w_T)_T + \sum_{f\in\f_T}(\pi_F^k(\boldsymbol{u}\cdot\normal_{TF}),w_T)_F
      \\
      &= -(\nabla\cdot\boldsymbol{u},\pi_T^kw_T)_T + \sum_{f\in\f_T}(\boldsymbol{u}\cdot\normal_{TF},\pi_F^kw_T)_F
      \\
      &= (\boldsymbol{u},\nabla\pi_T^kw_T)_T + \sum_{f\in\f_T}(\boldsymbol{u}\cdot\normal_{TF},\pi_F^kw_T-\pi_T^kw_T)_F,
    \end{aligned}
  \end{equation}
  where we have used the fact that $\FT\underline{\boldsymbol{v}}_T=\boldsymbol{K}_T\nabla w_T$ in the first line, the definition~\eqref{def:operatorfluxbulk} of $\FT\widehat{\underline{\boldsymbol{u}}}_T$ in the second line, the commuting property~\eqref{lem:commutdivbulk} together with the definition~\eqref{def:globalinterpoperator} of $\boldsymbol{\underline{I}}_h^k$ in the third line, the definition~\eqref{def:lproj} of the $L^2$-orthogonal projectors $\pi_T^k$ and $\pi_F^k$ to pass to the fourth line, and an integration by parts to conclude.
  
  On the other hand, recalling again that $\FT\underline{\boldsymbol{v}}_T=\boldsymbol{K}_T\nabla w_T$ and using the definition~\eqref{eq:cpT} of the local elliptic projection, we have that
  \begin{equation}\label{eq:Eh11:2}
    \begin{aligned}
      &(\nabla p,\FT\underline{\boldsymbol{v}}_T)_T
      =(\boldsymbol{K}_T\nabla p,\nabla w_T)_T
      = (\boldsymbol{K}_T\nabla\widecheck{p}_T,\nabla w_T)_T
      \\
      &\qquad\qquad= -(\nabla\cdot(\boldsymbol{K}_T\nabla\widecheck{p}_T), w_T)_T
      + \sum_{F\in\f_T}(\boldsymbol{K}_T\nabla\widecheck{p}_T\cdot\normal_{TF}, w_T)_T
      \\
      &\qquad\qquad= -(\nabla\cdot(\boldsymbol{K}_T\nabla\widecheck{p}_T),\pi_T^k w_T)_T
      + \sum_{F\in\f_T}(\boldsymbol{K}_T\nabla\widecheck{p}_T\cdot\normal_{TF}, \pi_F^k w_T)_T
      \\
      &\qquad\qquad=
      (\boldsymbol{K}_T\nabla p,\nabla\pi_T^kw_T)_T
      + \sum_{F\in\f_T}(\boldsymbol{K}_T\nabla\widecheck{p}_T\cdot\normal_{TF},\pi_F^kw_T-\pi_T^kw_T)_F,
    \end{aligned}
  \end{equation}
  where we have used an integration by parts to pass to the second line, the definition~\eqref{def:lproj} of the $L^2$-orthogonal projectors $\pi_T^k$ and $\pi_F^k$ together with the fact that $\nabla\cdot(\boldsymbol{K}_T\nabla\widecheck{p}_T)\in\mathbb{P}^{k-1}(T)\subset\mathbb{P}^k(T)$ and $(\boldsymbol{K}_T\nabla\widecheck{p}_T)_{|F}\cdot\normal_{TF}\in\mathbb{P}^k(F)$ for all $F\in\f_T$ (since $w_T\in\mathbb{P}^{k+1}(T)$ and $\boldsymbol{K}_T\in\mathbb{P}^0(T)^{2\times 2}$) in the second line, and again an integration by parts together with the definition~\eqref{eq:cpT} to replace $\widecheck{p}_T$ by $p$ in the first term and conclude.

  Summing~\eqref{eq:Eh11:1} and~\eqref{eq:Eh11:2}, using~\eqref{eq:strong:bulk:1} to replace $\boldsymbol{u}$ by $-\boldsymbol{K}\nabla p$, and rearranging the terms, we finally obtain 
  \begin{equation}\label{eq:Eh11:3}
    \begin{aligned}
    (\boldsymbol{K}_T^{-1}\FT\widehat{\underline{\boldsymbol{u}}}_T,\FT\underline{\boldsymbol{v}}_T)_T
    &=  - (\nabla p,\FT\underline{\boldsymbol{v}}_T)_T
    \\
    &\qquad    + \sum_{F\in\f_T}(\boldsymbol{K}_T\nabla(\widecheck{p}_T-p)\cdot\normal_{TF},\pi_F^kw_T-\pi_T^kw_T)_F.
    \end{aligned}
  \end{equation}

  Using~\eqref{eq:Eh11:3} for the consistency term in $m_T(\widehat{\underline{\boldsymbol{u}}}_T,\underline{\boldsymbol{v}}_T)$ (see~\eqref{eq:mT}), plugging the resulting relation into the expression of $a_h^\xi(\widehat{\underline{\boldsymbol{u}}}_h,\underline{\boldsymbol{v}}_h)$ (see~\eqref{def:ah}), and accounting for~\eqref{eq:Eh11:4} in the fracture terms of $a^\xi_h(\widehat{\underline{\boldsymbol{u}}}_h,\underline{\boldsymbol{v}}_h)$ (where $\pi_F^k$ can be cancelled using~\eqref{def:lproj} after observing that $\lambda_F^\xi[\![\underline{\boldsymbol{v}}_h]\!]_F\in\mathbb{P}^k(F)$ and $\lambda_F[\![\underline{\boldsymbol{v}}_h]\!]_F\in\mathbb{P}^k(F)$ for all $F\in\f_h^\Gamma$) gives~\eqref{eq:Eh11}.
  
  \paragraph{Step 5: Proof of~\eqref{eq:Eh12}}
  We have that
  \begin{equation}\label{eq:Eh12:1}
    \begin{aligned}
      &b_h(\underline{\boldsymbol{v}}_h,\widehat{p}_h)
      = b_h(\underline{\boldsymbol{v}}_h,\pi_h^k(p-\widecheck{p}_h))
      + b_h(\underline{\boldsymbol{v}}_h,\pi_h^k\widecheck{p}_h)
      \\
      &\qquad=
      b_h(\underline{\boldsymbol{v}}_h,\pi_h^k(p-\widecheck{p}_h))
      +\hspace{-1.125ex}\sum_{T\in\t_h}(\widecheck{p}_T,\DT\underline{\boldsymbol{v}}_T)_F
      \\
      &\qquad=
      b_h(\underline{\boldsymbol{v}}_h,\pi_h^k(p-\widecheck{p}_h))    
      +\hspace{-1.125ex}\sum_{T\in\t_h}\left(
      \sum_{F\in\f_T}(\widecheck{p}_T, v_{TF})_F
      -(\nabla\widecheck{p}_T,\FT\underline{\boldsymbol{v}}_T)_T
      \right)
      \\
      &\qquad=
      b_h(\underline{\boldsymbol{v}}_h,\pi_h^k(p-\widecheck{p}_h))
      +\hspace{-1.125ex}\sum_{T\in\t_h}\hspace{-0.25ex}\sum_{F\in\f_T}(\widecheck{p}_T-p_{|T}, v_{TF})_F
      -\hspace{-1.125ex}\sum_{T\in\t_h}(\nabla p,\FT\underline{\boldsymbol{v}}_T)_T
      \\
      &\qquad\qquad+\hspace{-1.125ex}\sum_{T\in\t_h}\hspace{-0.25ex}\sum_{F\in\f_T}(p_{|T}, v_{TF})_F,
    \end{aligned}
  \end{equation}
  where we have inserted $\pm\pi_h^k\widecheck{p}_h$ into the second argument of $b_h$ and used its linearity in the first line, expanded the second term according to its definition~\eqref{def:bh} and cancelled the projector since $\DT\underline{\boldsymbol{v}}_T\in\mathbb{P}^k(T)$ for all $T\in\t_h$ in the second line, used the definition~\eqref{def:operatorfluxbulk} of $\FT\underline{\boldsymbol{v}}_T$ (with $w_T=\widecheck{p}_T$) in the third line, and we have inserted $\pm\sum_{T\in\t_h}\sum_{F\in\f_T}(p_{|T}, v_{TF})_F$ to pass to the fourth line, where~\eqref{eq:cpT} was also used to write $p$ instead of $\widecheck{p}_T$ in the third term.

  Let us consider the last term in~\eqref{eq:Eh12:1}. Rearranging the sums and using the fact that $p=0$ on every boundary face $F\in\f_h^{\rm b}$ owing to~\eqref{eq:dirichlet.only}, it is inferred that
  $$
  \begin{aligned}
    \sum_{T\in\t_h}\sum_{F\in\f_T}(p_{|T}, v_{TF})_F
    = \sum_{F\in\f_h}\sum_{T\in\t_F}(p_{|T}, v_{TF})_F
    = \hspace*{-10px}\sum_{\substack{F\in\f_h^{\rm i}\\F\subset\partial T_1\cap\partial T_2}}\int_F\left(p_{|T_1} v_{T_1F} + p_{|T_2} v_{T_2F}\right).
  \end{aligned}
  $$
  If $F\in\f_h^{\rm i}\setminus\f_h^\Gamma$, the integrand vanishes since $v_{T_1F}+v_{T_2F}=0$ (see the definition~\eqref{eq:Uh0} of $\underline{\boldsymbol{U}}_{h,0}^{k}$) and $p_{|T_1}-p_{|T_2}=0$ since the jumps of the bulk pressure vanish across interfaces in the bulk region.
  If, on the other hand, $F\in\f_h^\Gamma$, assuming without loss of generality that $T_{i}\subset\Omega_{\bulk,i}$ for $i\in\{1,2\}$, it can be checked that $p_{|T_1} v_{T_1F} + p_{|T_2} v_{T_2F}=[\![p]\!]_\Gamma\{\!\{\underline{\boldsymbol{v}}_h\}\!\}_F+\{\!\{p\}\!\}_\Gamma[\![\underline{\boldsymbol{v}}_h]\!]_F$.
  In conclusion, we have that
  \begin{equation}\label{eq:Eh12:2}
  \int_F\left(p_{|T_1} v_{T_1F} + p_{|T_2} v_{T_2F}\right)
  = \begin{cases}
    0 & \text{if $F\in\f_h^{\rm i}\setminus\f_h^\Gamma$,}
    \\
    ([\![p]\!]_\Gamma,\{\!\{\underline{\boldsymbol{v}}_h\}\!\}_F)_F
    + (\{\!\{p\}\!\}_\Gamma,[\![\underline{\boldsymbol{v}}_h]\!]_F)_F
    & \text{if $F\in\f_h^\Gamma$}.
  \end{cases}
  \end{equation}
  Plugging~\eqref{eq:Eh12:2} into~\eqref{eq:Eh12:1}, and using~\eqref{eq:strong:couplcond} to replace $[\![p]\!]_\Gamma$ and $\{\!\{p\}\!\}_\Gamma$,~\eqref{eq:Eh12} follows.
\end{proof}


\bibliographystyle{plain}
\begin{footnotesize}
  \bibliography{ffpmhho}
\end{footnotesize}

\end{document}